\documentclass[10pt]{article}
\usepackage{flafter,amsmath,amssymb,latexsym,psfrag,graphicx,color,indentfirst,bm,amsthm}
\usepackage[margin= 1.85cm]{geometry}
\usepackage[T1]{fontenc}
\usepackage[utf8]{inputenc}
\usepackage[colorlinks,
linktocpage=true,
linkcolor=red,
citecolor=red,
urlcolor=red,
anchorcolor=black
]{hyperref}
\usepackage{float}
\usepackage[makeroom]{cancel}
\usepackage{mathtools}
\usepackage[numbers,sort&compress]{natbib}
\usepackage{appendix}
\usepackage{multirow}
\usepackage{caption}
\usepackage{subfigure}
\usepackage{graphicx}
\usepackage{dutchcal}
\usepackage{comment}
\allowdisplaybreaks
\usepackage{tikz}
\tikzset{mynode/.style={inner sep=2pt,fill,outer sep=0,circle}}
\usetikzlibrary{fadings}
\usepackage{amsmath,amssymb,braket,tikz}
\usetikzlibrary{tikzmark,calc}

\newtheorem{theorem}{Theorem}[section]
\newtheorem{lemma}{Lemma}[section] 
\newtheorem{corollary}{Corollary}[section] 
\newtheorem{proposition}{Proposition}[section] 
\newtheorem{assumption}{Assumption} 
\newtheorem{remark}{Remark}[section]  

\numberwithin{equation}{section}

\usepackage{mathrsfs}
\newsavebox\foobox
\newlength{\foodim}
\newcommand{\slantbox}[2][0]{\mbox{%
        \sbox{\foobox}{#2}%
        \foodim=#1\wd\foobox
        \hskip \wd\foobox
        \hskip -0.5\foodim
        \pdfsave
        \pdfsetmatrix{1 0 #1 1}%
        \llap{\usebox{\foobox}}%
        \pdfrestore
        \hskip 0.5\foodim
}}

\def\Fourier{\slantbox[-.45]{$\mathscr{F}$}}

\numberwithin{equation}{section}
\numberwithin{figure}{section}

\DeclareRobustCommand{\rchi}{{\mathpalette\irchi\relax}}
\newcommand{\irchi}[2]{\raisebox{\depth}{$#1\chi$}}
\newcommand\obullet[1]{\ThisStyle{\ensurestackMath{%
  \stackon[1pt]{\SavedStyle#1}{\SavedStyle\kern.6\LMpt\bullet}}}}
\newcommand\ocirc[1]{\ThisStyle{\ensurestackMath{%
  \stackon[1pt]{\SavedStyle#1}{\SavedStyle\kern.6\LMpt\circ}}}}
  
\title{ Dispersive Effective Model in the Time-Domain \\ for Acoustic Waves Propagating in Bubbly Media}
\author{Arpan Mukherjee\footnote{Radon Institute (RICAM), Austrian Academy of
Sciences, Altenbergerstrasse 69, A-4040, Linz, Austria (arpan.mukherjee@oeaw.ac.at). This author is supported by the Austrian Science Fund (FWF): P36942.} \ and Mourad Sini\footnote{Radon Institute (RICAM), Austrian Academy of Sciences, Altenbergerstrasse 69, A-4040, Linz, Austria (mourad.sini@oeaw.ac.at). This author is partially supported by the Austrian Science Fund (FWF): P36942.}}

\begin{document}
\maketitle
\begin{abstract}

We derive the effective medium theory for the linearized time-domain acoustic waves propagating in a bubbly media. The analysis is done in the time-domain avoiding the need to use Fourier transformation. This allows considering general incident waves avoiding band limited ones as usually used in the literature. Most importantly, the outcome is as follows:
\begin{enumerate}
       \item As the bubbles are resonating, with the unique subwavelength Minnaert resonance, the derived effective wave model is dispersive. Precisely, the effective acoustic model is an integro-differential one with a time-convolution term highlighting the resonance effect.

       \item The periodicity in distributing the cluster of bubbles is not needed, contrary to the case of using traditional two-scale homogenization procedures. Precisely, given any $C^1$-smooth function $K$, we can distribute the bubbles so that locally the number of such bubbles is dictated by $K$. In addition to its dispersive character, the effective model is affected by the function $K$. Such freedom and generality is appreciable in different applied sciences including materials sciences and mathematical imaging.
\end{enumerate}
 
\end{abstract}
\tableofcontents


\section{Introduction}          

\subsection{Time-Dependent
 Acoustic Waves Propagating in Bubbly Media}
\noindent
Let $\mathrm{D}$ be a bounded $C^2$-regular domain in $\mathbb R^3$ and consider the following transmission acoustic problem:
\begin{align}\label{hyperbolic}
\begin{cases}   \mathrm{k}^{-1}(\mathrm{x})u_{\mathrm{t}\mathrm{t}}- \text{div} \rho^{-1}(\mathrm{x)}\nabla u = 0 & \text{in}\;  (\mathbb{R}^{3}\setminus \partial\mathrm{D}) \times (0,\mathrm{T}),\\
 u\big|_{+} = u\big|_{-}  & \text{on}\; \partial\mathrm{D},\\
 \rho_{{\mathrm{c}}}^{-1}\partial_\nu u\big|_{+} = \rho_{\mathrm{b}}^{-1} \partial_\nu u\big|_{-} & \text{on} \; \partial \mathrm{D}, \\
 u(\mathrm{x},0)= u_\mathrm{t}(\mathrm{x},0) = 0 & \text{for}\; \mathrm{x} \in \mathbb{R}^3.
\end{cases}
\end{align}
We define $\rho$ as the mass density, given by $\rho = \rho_{\mathrm{b}}\rchi_{\mathrm{D}} + \rho_{c}\rchi_{\mathbb{R}^{3}\setminus\overline{\mathrm{D}}}$, where $\rho_{\mathrm{b}}$ and $\rho_{c}$ represent the mass densities within the domain $\mathrm{D}$ and its complement, respectively. Similarly, $\mathrm{k}$ represents the bulk modulus, defined as $\mathrm{k} = \mathrm{k}_{\mathrm{b}}\rchi_{\mathrm{D}} + \mathrm{k}_{\mathrm{c}}\rchi_{\mathbb{R}^{3}\setminus\overline{\mathrm{D}}}$, with $\mathrm{k}_{\mathrm{b}}$ and $\mathrm{k}_{\mathrm{c}}$ denoting the bulk modulus of the bubble and the acoustic medium, respectively. Here, we denote  
$
\partial_\nu u \big|_{\pm}(\mathrm{x},\mathrm{t}) := \lim_{\mathrm{h}\to 0}\nabla u(\mathrm{x}\pm \mathrm{h}\nu_\mathrm{x},\mathrm{t})\cdot \nu_\mathrm{x},
$
where $\nu$ represents the outward normal vector to $\partial\mathrm{D}$. 
The motivation for this study comes from the propagation of acoustic waves in bubbly media. The set \(\mathrm{D}\) is therefore defined as the union of small subdomains denoted by \(\mathrm{D} := \bigcup_{i=1}^M \mathrm{D}_i\), where each \(\mathrm{D}_i\) has the form \(\mathrm{D}_i = \delta \mathrm{B}_i + \mathrm{z}_i\). Here, for \(i = 1, 2, \ldots, M\), \(z_i\) and \(B_i\) represent the location and shape of a bubble, respectively. The parameter \(\delta \in \mathbb{R}^+\), with \(\delta \ll 1\), indicates the relative radius of \(\mathrm{D}_i\) compared to \(\mathrm{B}_i\). The domains \(\mathrm{B}_i\) are uniformly bounded with \(\mathcal{C}^2\) regularity and contain the origin. The maximum radius \(\delta\) is taken to be small compared to the maximum radius of the \(B_i\)'s, \(\delta \ll 1\).
We define \(\mathrm{d}\) as the minimum distance between bubbles, denoted by \(\mathrm{d}_{\mathrm{ij}} = \textbf{dist}(\mathrm{D}_i, \mathrm{D}_j)\) for \(i \neq j\), where \(\textbf{dist}\) represents the distance function, i.e.,
\[
\mathrm{d} := \min_{1 \leq i, j \leq \mathrm{M}} \mathrm{d}_{\mathrm{ij}}.
\]
Furthermore, we denote \(\varepsilon\) as the maximum diameter among the microbubbles, given by:
\[
\varepsilon := \max_{1 \leq i \leq \mathrm{M}} \textbf{diam}(\mathrm{D}_i) = \delta \max_{1 \leq i \leq \mathrm{M}} \textbf{diam}(\mathrm{B}_i),
\]
where \(\textbf{diam}\) represents the diameter function.
\newline

\noindent
We assume that the bubbles are filled with gas. Hence, they have a small bulk modulus \(\mathrm{k}_b\) and mass densities \(\rho_{\mathrm{b}}\) compared to the background liquid in which the bubbles are injected. Therefore, their sizes (i.e., their radii) can be designed to have the following scaling properties:
\begin{equation}\label{scale}
    \big[\mathrm{k}(x), \rho(x)\big] := 
    \begin{cases}
       [k_c, \rho_{\mathrm{c}}] & \text{in} \; \mathbb{R}^3 \setminus \mathrm{D}_i, \\
       [\mathrm{k}_{\mathrm{b}_i}, \rho_{\mathrm{b}_i}] = \big[\overline{\mathrm{k}}_{b_i} \delta^2, \overline{\rho}_{b_i} \delta^2\big] & \text{in} \; \mathrm{D}_i,
    \end{cases}
\end{equation}
where \(\overline{\mathrm{k}}_{b_i}\) and \(\overline{\rho}_{b_i}\) are independent of \(\delta\). With such scaling, we observe that the speed of propagation in \(\bigcup_{i=1}^\mathrm{M}\mathrm{D}_i\) and \(\mathbb{R}^{3} \setminus \overline{\bigcup_{i=1}^\mathrm{M}\mathrm{D}_i}\) is of order 1, i.e., \(\frac{\rho}{\mathrm{k}} \sim 1\).

\noindent
In this paper, we are interested in the following regimes for modeling the cluster distributed in a $3$D-bounded domain
\begin{align} \label{regimes}
    M \sim d^{-3} \quad \text{and} \quad d \sim \delta^3, \quad \delta \ll 1.
\end{align}
The model described above, with its specified conditions and regimes, is related to the linearized model for acoustic wave propagation in bubbly media; see \cite{C-M-P-T-1, C-M-P-T-2}. There are several works devoted to related time-harmonic settings; see \cite{habib-sini, AFGLZ-17, FH, DGS-21, GS-21, MPS}. Regarding the time-domain regimes, we derived in \cite[Theorem 1.1]{Arpan-Sini-JEvEq}, see also \cite{Arpan-Sini-SIMA}, the asymptotic expansion of the solution to the above problem in the presence of a cluster of bubbles using the time-domain boundary integral equation method under the described regimes with appropriate conditions on the denseness of the cluster of bubbles. These expansions, and the related regimes, are stated in the following proposition.

\noindent
To state this proposition, let us briefly introduce the required function spaces. We define the real-valued Sobolev space \(\mathrm{H}_0^\mathrm{r}(0,\mathrm{T})\) for \(\mathrm{r} \in \mathbb{R}\) and \(\mathrm{T} \in (0, \infty]\) as follows:
\[
\mathrm{H}_0^\mathrm{r}(0,\mathrm{T}) := \Big\{\mathrm{g}|_{(0,\mathrm{T})}: \mathrm{g} \in \mathrm{H}^\mathrm{r}(\mathbb{R}) \; \text{and} \; \mathrm{g}_{(-\infty,0)} \equiv 0\Big\}.
\]
Similarly, we can introduce a similar concept for functions taking values in a Hilbert space \(\mathrm{X}\), denoted by \(\mathrm{H}^\mathrm{r}_{0}(0,\mathrm{T};\mathrm{X})\). For \(\sigma > 0\) and \(\mathrm{r} \in \mathbb{Z}_+\), we define
\[
\mathrm{H}^\mathrm{r}_{0,\sigma}(0,\mathrm{T};\mathrm{X}) := \Big\{\mathrm{f} \in \mathrm{H}^\mathrm{r}_{0}(0,\mathrm{T};\mathrm{X}): \sum_{\mathrm{n}=0}^\mathrm{r}\int_0^\mathrm{T}\mathrm{e}^{-2\sigma \mathrm{t}} \Vert \partial_\mathrm{t}^\mathrm{n} \mathrm{f}(\cdot,\mathrm{t}) \Vert_\mathrm{X} \, d\mathrm{t} < \infty \Big\}.
\]

\begin{proposition}\label{mainthmulti}\cite[Theorem 1.1]{Arpan-Sini-JEvEq}
We take as an incident wave field \(u^\textbf{in}\) originating from a point source located at \(\mathrm{x}_0 \in \mathbb{R}^{3} \setminus \overline{\bigcup_{i=1}^\mathrm{M} \mathrm{D}_i}\) of the form
\begin{align}\label{incident-wave}
u^\textbf{in}(\mathrm{x}, \mathrm{t}, \mathrm{x}_0) := \frac{\lambda(\mathrm{t} - \mathrm{c}_0^{-1} \vert \mathrm{x} - \mathrm{x}_0 \vert)}{\vert \mathrm{x} - \mathrm{x}_0 \vert},
\end{align}
where \(\lambda \in \mathcal{C}^9(\mathbb{R})\) is a causal signal (i.e., it vanishes for \(\mathrm{t} < 0\)). We also denote by \(\mathrm{c}_0 = \sqrt{\frac{\mathrm{k}_\mathrm{c}}{\rho_\mathrm{c}}}\) the constant wave speed in \(\mathbb{R}^3 \setminus \overline{\mathrm{D}}\).

\noindent
We introduce the grad-harmonic subspace
\begin{align}\nonumber
    \nabla \mathbb{H}_{\text{arm}} := \Big\{ u \in \big(\mathrm{L}^2(\mathrm{B}_i)\big)^3: \exists \ \varphi \ \text{s.t.} \ u = \nabla \varphi, \; \varphi \in \mathrm{H}^1(\mathrm{B}_i) \; \text{and} \ \Delta \varphi = 0 \Big\}
\end{align}
and recall that the Magnetization potential is given by
\begin{align}\nonumber
    \bm{\mathrm{M}}_{\mathrm{B}_i}\big[f\big](\mathrm{x}) := \nabla \int_{\mathrm{B}_i} \mathop{\nabla}\limits_{\mathrm{y}} \frac{1}{4\pi |\mathrm{x} - \mathrm{y}|} \cdot  f(\mathrm{y}) \, d\mathrm{y}, \quad f \in \nabla \mathbb{H}_{\text{arm}}.
\end{align}
The Magnetization operator \(\bm{\mathrm{M}}_{\mathrm{B}_i}: \nabla \mathbb{H}_{\text{arm}} \rightarrow \nabla \mathbb{H}_{\text{arm}}\) induces a complete orthonormal basis, namely \(\big(\lambda^{(3)}_{\mathrm{n}_i}, \mathrm{e}^{(3)}_{\mathrm{n}_i}\big)_{\mathrm{n} \in \mathbb{N}}\); see for instance \cite{friedmanI}. We set \(\lambda_1^{(3)} := \min_i \max_n \lambda_{n_i}^{(3)}\). Then, under the conditions
\begin{align}\label{inversion-cond}
    \frac{\rho_\mathrm{c}}{4\pi} \text{vol}(\mathrm{B}_j) \; \Big(\frac{\delta}{\mathrm{d}}\Big)^6 \; \Big(\frac{1}{\lambda_1^{(3)}}\Big)^2 < 1 \quad \text{and} \quad \max_{1 \le i \le M} \sum_{\substack{j=1 \\ j \neq i}}^M \mathrm{q}_{ij} < \hbar_i,
\end{align}
the hyperbolic problem (\ref{hyperbolic}) is well-posed, and setting \(u := u^\textbf{in} + u^\mathrm{s}\), we have the asymptotic expansion
\begin{align}\label{assymptotic-expansion-us}
    u^\mathrm{s}(\mathrm{x}, \mathrm{t}) = \sum_{i=1}^\mathrm{M} \frac{\alpha_i \rho_\mathrm{c}}{4\pi |\mathrm{x} - \mathrm{z}_i|} \; |\mathrm{D}_i| \;\frac{\rho_{b_i}}{\mathrm{k}_{b_i}} \widetilde{\mathrm{Y}}_i\big(\mathrm{t} - \mathrm{c}_0^{-1} |\mathrm{x} - \mathrm{z}_i|\big) + \mathcal{O}(\delta^{2-l}) \; \text{as} \; \delta \to 0,
\end{align}
for \((\mathrm{x}, \mathrm{t}) \in \mathbb{R}^3 \setminus \mathrm{K} \times (0, \mathrm{T})\), with \(\overline{\mathbf{\Omega}} \subset\subset \mathrm{K}\), where \(\big(\widetilde{\mathrm{Y}}_i\big)_{i=1}^\mathrm{M}\) is the vector solution to the following non-homogeneous second-order matrix differential equation with zero initial conditions:
\begin{align}\label{matrixmulti}
\begin{cases}
    \hbar_i \frac{\mathrm{d}^2}{\mathrm{d} \mathrm{t}^2} \widetilde{\mathrm{Y}}_i(\mathrm{t}) + \widetilde{\mathrm{Y}}_i(\mathrm{t}) + \sum\limits_{\substack{j=1 \\ j \neq i}}^M \mathrm{q}_{ij} \frac{\mathrm{d}^2}{\mathrm{d} \mathrm{t}^2} \widetilde{\mathrm{Y}}_j(\mathrm{t} - \mathrm{c}_0^{-1} |\mathrm{z}_i - \mathrm{z}_j|) = \frac{\partial^2}{\partial t^2} u^\textbf{in}, \quad \text{in} \ (0, \mathrm{T}), \\
    \widetilde{\mathrm{Y}}_i(0) = \frac{\mathrm{d}}{\mathrm{d} \mathrm{t}} \widetilde{\mathrm{Y}}_i(0) = 0,
\end{cases}
\end{align}
where \(\hbar_i\) is defined as
$\hbar_i := \frac{\rho_\mathrm{c}}{2} \alpha_i \frac{\rho_{\mathrm{b}_i}}{\mathrm{k}_{\mathrm{b}_i}} \mathrm{A}_{\partial \mathrm{D}_i},$
with \(\alpha_i := \rho_{\mathrm{b}_i}^{-1} - \rho_\mathrm{c}^{-1}\) being the contrast between the inner and the outer acoustic coefficient. Furthermore, \(\bm{\mathrm{Q}} = \big(\mathrm{q}_{ij}\big)_{i,j=1}^\mathrm{M}\) is given by
\begin{align}\label{interaction-matrix}
\mathrm{q}_{ij} = \begin{cases}
0 & \text{for} \ i = j, \\
\frac{\mathrm{b}_j}{|\mathrm{z}_i - \mathrm{z}_j|} & \text{for} \ i \ne j,
\end{cases}
\end{align}
where \(\mathrm{b}_j\) is defined as
$\mathrm{b}_j := \rho_\mathrm{c} \alpha_j\ \text{vol}(D_j) \frac{\rho_{\mathrm{b}_j}}{\mathrm{k}_{\mathrm{b}_j}}.$
We also define
$\displaystyle\mathrm{A}_{\partial \mathrm{D}_i} := \frac{1}{|\partial \mathrm{D}_i|} \int_{\partial \mathrm{D}_i} \int_{\partial \mathrm{D}_i} \frac{(\mathrm{x} - \mathrm{y}) \cdot \nu_\mathrm{x}}{|\mathrm{x} - \mathrm{y}|} \, d\sigma_\mathrm{x} \, d\sigma_\mathrm{y}.$
\end{proposition}

\begin{remark} 
The aforementioned differential equation is deduced from the original theorem stated in \cite[Theorem 1.1]{Arpan-Sini-JEvEq}, employing Taylor's series expansion and integration by parts:
\begin{align}\label{toy}
\nonumber\int_{\partial\mathrm{D}_i} \partial_\nu u^\textbf{in}(\mathrm{y},\mathrm{\tau}) \, d\sigma_\mathrm{y} 
= \int_{\mathrm{D}_i} \Delta u^\textbf{in}(\mathrm{y},\mathrm{\tau}) \, d\mathrm{y} 
= \frac{\rho_\mathrm{c}}{\mathrm{k}_\mathrm{c}} \int_{\mathrm{D}_i} u_{\mathrm{t}\mathrm{t}}^\textbf{in}(\mathrm{y},\mathrm{\tau}) \, d\mathrm{y} 
= \frac{\rho_\mathrm{c}}{\mathrm{k}_\mathrm{c}}\ \text{vol}(\mathrm{D}_i) u_{\mathrm{t}\mathrm{t}}^\textbf{in}(\mathrm{z},\mathrm{\tau}) + \mathcal{O}(\delta^4),
\end{align}
and using the following change of unknown $\widetilde{\mathrm{Y}}_i = \frac{1}{|\mathrm{D}_i|}\ \frac{\mathrm{k}_{b_i}}{\rho_{b_i}}\mathrm{Y}_i.$
\end{remark}
\noindent
Due to the scaling properties introduced in (\ref{scale}), we see that 
\begin{equation}\label{Minnaert-coefficient}
\hbar_i = \frac{\rho_c}{2\overline{k}_{b_i}}\mathrm{A}_{\partial B_i} + \mathcal{O}(\delta^2)
\end{equation}
and 
\begin{equation}\label{b_j}
b_j = \frac{\rho_c}{\overline{\mathrm{k}}_{b_j}}\text{vol}(B_j)\delta + \mathcal{O}(\delta^3).
\end{equation}
Therefore, neglecting the error terms, we set the scattering coefficient $\hbar_i :=\frac{\rho_c}{2\overline{k}_{b_j}}\mathrm{A}_{\partial B_i} $, which is of order 1. Similarly, for $\mathrm{b}_j$, we have $b_j = \overline{b}_j \delta$. Thus, we set the scattering coefficient $\overline{b}_j:= \frac{\rho_c}{\overline{\mathrm{k}}_{b_j}}\text{vol}(B_j)$. In this work, we take $\hbar_i$ and $\overline{b}_j$ to be the same value for $\hbar:=\hbar_i\ \text{and}\ b := \overline{b}_j, i,j = 1, ..., M$, for every bubble. \footnote{This assumption is not necessary. We could also take the freedom of varying $\hbar_i$ and get an effective medium with the corresponding coefficient as a variable function.} This means that for every used shape $B_i$, of the bubbles, we choose the scaled bulk $\overline{k}_{b_i}$ so that $\hbar_i:=\hbar\ \text{and}\ b:=\overline{b}_j\ \text{for}\ i,j=1, ..., M$.
\bigskip

Notice that the constant $\hbar_i = \frac{\rho_c}{2\overline{k}_{b_i}}\mathrm{A}_{\partial B_i}$ represents the square of the Minnaert frequency of the bubble $D_i$, as detailed in \cite[pp. 10]{AZ18}. This frequency is also characterized, locally in the complex plane and for $\epsilon \ll 1$, as the unique resonance, up to a sign, in the sense of the resolvent of the natural Hamiltonian $k \; \text{div} \rho^{-1}(\mathrm{x}) \nabla$ stated in the weighted space $L^2(\mathbf{R}^3, k^{-1}(x)dx)$ (see \cite{L-S-2024, MPS}). The appearance of this resonant frequency is crucial for deriving the corresponding effective model discussed in Section \ref{ese}.


\subsection{The Generated Dispersive Effective Model}\label{ese}   

\begin{assumption}\label{as2}
    Let $\mathbf{\Omega}$ be a bounded domain of unit volume. Let be given a function $K$ which is $C^1-$ smooth in $\mathbf{\Omega}$ with non-vanishing derivative \footnote{Actually, we only need $K$ to have non-vanishing derivative on the set of discontinuity of the function entire part of $K$, i.e. $[K](\cdot)$. Under these conditions, the level sets \(\big\{ x \in \mathbf{\Omega} \subseteq \mathbb{R}^3 \mid \mathrm{K}(x) = n \big\}\) (where \( n \) is an integer) are 2-dimensional surfaces in \(\mathbb{R}^3\). This property is needed in section \ref{K-function}.}. We divide $\mathbf{\Omega}$ into $[\varepsilon^{-1}]$\footnote{Here, we denote by $[x]$ the unique integer $ n $ such that $ n \leq x < n + 1 $, where $ n $ represents the floor number.} periodically distributed subdomains $\Omega_m,\; m = 1, 2, \ldots , [\varepsilon^{-1}]$. Then, we further divide each $\Omega_m$ into $K(z_m)+1$\footnote{Here $z_m$ is any of the points $z_{m_l}$ inside $\Omega_m$.} subdomains $\Omega_{m_l}$ under the condition that $z_{m_l}\in \mathrm{D}_{m_l} \subseteq \Omega_{m_l}\subseteq \Omega_m.$
\end{assumption}

\begin{figure}[H]
    \begin{center}
 \begin{tikzpicture}[scale=1]

\draw (0,0) ellipse (3cm and 1.5cm);
\node[scale=1.5] at (0,-2.5) {$\mathbf{\Omega}$};


\draw[<-,red, thick] (0.25, 0.25) -- (0.25, 2.5);
\node[above, scale=1] at (0.25,2.5) {$\Omega_m$};
\draw[<-,red, thick] (-0.7, 0.25) -- (-0.7, 2.5);
\node[above, scale=1] at (-0.7,2.5) {$\Omega_i$};

\node (formula) []{};
\draw[-latex,red] ($(formula.north west)+(4.5,0)$) arc
    [
        start angle=160,
        end angle=20,
        x radius=0.5cm,
        y radius =0.5cm
    ] ;

\clip (0,0) ellipse (3cm and 1.5cm); 

\foreach \x in {-3.5,-2.5,...,3.5} {
    \draw (\x,-2) -- (\x,2);
}


\def\xellip(#1){4*sqrt(1 - (#1/2)^2)}

\foreach \y in {-1.75,-1.25,-.75,...,1.75} {
    \draw ({\xellip(\y)},\y) -- ({-\xellip(\y)},\y);
}


\coordinate (a) at (-1.25,0.1); 
\shade[ball color=cyan] (a) circle (0.1); 
\node[above, scale=0.3] at (-1.25,0.05)  {$\bullet$};  

\coordinate (f) at (-.65,0.05); 
\shade[ball color=cyan] (f) circle (0.1);
\node[above, scale=0.3] at (-.65,0)  {$\bullet$};

\coordinate (c) at (-1,-0.05); 
\shade[ball color=cyan] (c) circle (0.1);
\node[above, scale=0.3] at (-1,-0.1)  {$\bullet$};

\coordinate (e) at (-.9,0.15); 
\shade[ball color=cyan] (e) circle (0.1);
\node[above, scale=0.3] at (-.9,0.1)  {$\bullet$};

\coordinate (d) at (-.78,-.1); 
\shade[ball color=cyan] (d) circle (0.1);
\node[above, scale=0.3] at (-.78,-0.15)  {$\bullet$};

\coordinate (b) at (-1.25,-.1); 
\shade[ball color=cyan] (b) circle (0.1);
\node[above, scale=0.3] at (-1.25,-0.15)  {$\bullet$};


\coordinate (M) at (-.3,0.1); 
\shade[ball color=yellow] (M) circle (0.1);
\node[above, scale=0.3] at (-0.3,0.05)  {$\bullet$};

\coordinate (O) at (0,0); 
\shade[ball color=yellow] (O) circle (0.1);
\node[above, scale=0.3] at (0,-0.05)  {$\bullet$};

\coordinate (P) at (-.3,-0.15); 
\shade[ball color=yellow] (P) circle (0.1);
\node[above, scale=0.3] at (-.3,-0.2)  {$\bullet$};

\coordinate (Q) at (0.3,-0.05);
\shade[ball color=yellow] (Q) circle (0.1);
\node[above, scale=0.3] at (0.3,-0.1)  {$\bullet$};

\coordinate (R) at (0.2,0.15); 
\shade[ball color=yellow] (R) circle (0.1);
\node[above, scale=0.3] at (0.2,0.1)  {$\bullet$};


\coordinate (M1) at (1,0.1); 
\shade[ball color=white] (M1) circle (0.1);
\node[above, scale=0.3] at (1,0.05)  {$\bullet$};

\coordinate (O1) at (1.25,0); 
\shade[ball color=white] (O1) circle (0.1);
\node[above, scale=0.3] at (1.25,-0.05)  {$\bullet$};

\coordinate (P1) at (1,-0.15); 
\shade[ball color=white] (P1) circle (0.1);
\node[above, scale=0.3] at (1,-0.2)  {$\bullet$};

\coordinate (Q1) at (0.75,-0.05);
\shade[ball color=white] (Q1) circle (0.1);
\node[above, scale=0.3] at (0.75,-0.1)  {$\bullet$};


\coordinate (M2) at (1.75,0.1); 
\shade[ball color=white] (M2) circle (0.1);
\node[above, scale=0.3] at (1.75,0.05)  {$\bullet$};

\coordinate (O2) at (2.25,0); 
\shade[ball color=white] (O2) circle (0.1);
\node[above, scale=0.3] at (2.25,-0.05)  {$\bullet$};

\coordinate (P2) at (2,-0.15); 
\shade[ball color=white] (P2) circle (0.1);
\node[above, scale=0.3] at (2,-0.2)  {$\bullet$};


\coordinate (M3) at (1.75,0.35); 
\shade[ball color=white] (M3) circle (0.1);
\node[above, scale=0.3] at (1.75,0.3)  {$\bullet$};

\coordinate (O3) at (2,0.45); 
\shade[ball color=white] (O3) circle (0.1);
\node[above, scale=0.3] at (2,.4)  {$\bullet$};

\coordinate (P3) at (2.25,.55); 
\shade[ball color=white] (P3) circle (0.1);
\node[above, scale=0.3] at (2.25,.5)  {$\bullet$};

\coordinate (Q3) at (1.75,.55);
\shade[ball color=white] (Q3) circle (0.1);
\node[above, scale=0.3] at (1.75,.5)  {$\bullet$};

\coordinate (R3) at (2,0.65); 
\shade[ball color=white] (R3) circle (0.1);
\node[above, scale=0.3] at (2,0.6)  {$\bullet$};


\coordinate (M3) at (-1.75,0.45); 
\shade[ball color=white] (M3) circle (0.1);
\node[above, scale=0.3] at (-1.75,0.4)  {$\bullet$};

\coordinate (O3) at (-2.25,0.55); 
\shade[ball color=white] (O3) circle (0.1);
\node[above, scale=0.3] at (-2.25,.5)  {$\bullet$};


\coordinate (M2) at (-1,0.45); 
\shade[ball color=white] (M2) circle (0.1);
\node[above, scale=0.3] at (-1,0.4)  {$\bullet$};

\coordinate (O2) at (-.75,.65); 
\shade[ball color=white] (O2) circle (0.1);
\node[above, scale=0.3] at (-.75,.6)  {$\bullet$};

\coordinate (P2) at (-1.25,.6); 
\shade[ball color=white] (P2) circle (0.1);
\node[above, scale=0.3] at (-1.25,.55)  {$\bullet$};


\coordinate (M2) at (-.25,0.45); 
\shade[ball color=white] (M2) circle (0.1);
\node[above, scale=0.3] at (-.25,0.4)  {$\bullet$};


\coordinate (M3) at (0.75,0.45); 
\shade[ball color=white] (M3) circle (0.1);
\node[above, scale=0.3] at (.75,0.4)  {$\bullet$};

\coordinate (O3) at (1,0.4); 
\shade[ball color=white] (O3) circle (0.1);
\node[above, scale=0.3] at (1,.35)  {$\bullet$};


\coordinate (a) at (1.2,1); 
\shade[ball color=white] (a) circle (0.1); 
\node[above, scale=0.3] at (1.2,0.95)  {$\bullet$};  
\coordinate (f) at (1.35,1.15); 
\shade[ball color=white] (f) circle (0.1);
\node[above, scale=0.3] at (1.35,1.1)  {$\bullet$};
\coordinate (c) at (1.4,.85); 
\shade[ball color=white] (c) circle (0.1);
\node[above, scale=0.3] at (1.4,.8)  {$\bullet$};
\coordinate (e) at (1.1,0.8); 
\shade[ball color=white] (e) circle (0.1);
\node[above, scale=0.3] at (1.1,0.75)  {$\bullet$};
\coordinate (d) at (.75,.95); 
\shade[ball color=white] (d) circle (0.1);
\node[above, scale=0.3] at (.75,0.9)  {$\bullet$};
\coordinate (b) at (.95,.95); 
\shade[ball color=white] (b) circle (0.1);
\node[above, scale=0.3] at (.95,0.9)  {$\bullet$};
\coordinate (b) at (.95,1.15); 
\shade[ball color=white] (b) circle (0.1);
\node[above, scale=0.3] at (.95,1.1)  {$\bullet$};
\coordinate (b) at (.75,1.15); 
\shade[ball color=white] (b) circle (0.1);
\node[above, scale=0.3] at (.75,1.1)  {$\bullet$};


\coordinate (a) at (-.25,1); 
\shade[ball color=white] (a) circle (0.1); 
\node[above, scale=0.3] at (-.25,0.95)  {$\bullet$};  
\coordinate (f) at (0,1.15); 
\shade[ball color=white] (f) circle (0.1);
\node[above, scale=0.3] at (0,1.1)  {$\bullet$};


\coordinate (a) at (-.85,1); 
\shade[ball color=white] (a) circle (0.1); 
\node[above, scale=0.3] at (-.85,0.95)  {$\bullet$};  


\coordinate (a) at (-2.25,0); 
\shade[ball color=white] (a) circle (0.1); 
\node[above, scale=0.3] at (-2.25,-0.05)  {$\bullet$};  


\coordinate (a) at (-2.25,-0.5); 
\shade[ball color=white] (a) circle (0.1); 
\node[above, scale=0.3] at (-2.25,-0.55)  {$\bullet$};  
\coordinate (f) at (-2,-.35); 
\shade[ball color=white] (f) circle (0.1);
\node[above, scale=0.3] at (-2,-.4)  {$\bullet$};
\coordinate (a) at (-1.75,-0.5); 
\shade[ball color=white] (a) circle (0.1); 
\node[above, scale=0.3] at (-1.75,-0.55)  {$\bullet$};  


\coordinate (a) at (-1.25,-0.95); 
\shade[ball color=white] (a) circle (0.1); 
\node[above, scale=0.3] at (-1.25,-1)  {$\bullet$};  


\coordinate (a) at (-.25,-1.15); 
\shade[ball color=white] (a) circle (0.1); 
\node[above, scale=0.3] at (-.25,-1.2)  {$\bullet$};  
\coordinate (f) at (-.25,-.85); 
\shade[ball color=white] (f) circle (0.1);
\node[above, scale=0.3] at (-.25,-.9)  {$\bullet$};
\coordinate (a) at (0,-0.85); 
\shade[ball color=white] (a) circle (0.1); 
\node[above, scale=0.3] at (0,-0.9)  {$\bullet$};  
\coordinate (a) at (.25,-1.15); 
\shade[ball color=white] (a) circle (0.1); 
\node[above, scale=0.3] at (.25,-1.2)  {$\bullet$};  


\coordinate (a) at (.75,-1.15); 
\shade[ball color=white] (a) circle (0.1); 
\node[above, scale=0.3] at (.75,-1.2)  {$\bullet$};  
\coordinate (f) at (1.25,-.85); 
\shade[ball color=white] (f) circle (0.1);
\node[above, scale=0.3] at (1.25,-.9)  {$\bullet$};


\coordinate (a) at (.85,-.6); 
\shade[ball color=white] (a) circle (0.1); 
\node[above, scale=0.3] at (.85,-.65)  {$\bullet$};  


\coordinate (a) at (1.75,-.5); 
\shade[ball color=white] (a) circle (0.1); 
\node[above, scale=0.3] at (1.75,-.55)  {$\bullet$};  


\coordinate (a) at (.25,-.55); 
\shade[ball color=white] (a) circle (0.1); 
\node[above, scale=0.3] at (.25,-.6)  {$\bullet$};  
\coordinate (f) at (.05,-.45); 
\shade[ball color=white] (f) circle (0.1);
\node[above, scale=0.3] at (.05,-.5)  {$\bullet$};


\coordinate (a) at (-0.65,-.55); 
\shade[ball color=white] (a) circle (0.1); 
\node[above, scale=0.3] at (-.65,-.6)  {$\bullet$};  
\coordinate (f) at (-1.25,-.45); 
\shade[ball color=white] (f) circle (0.1);
\node[above, scale=0.3] at (-1.25,-.5)  {$\bullet$};
\coordinate (a) at (-1,-0.65); 
\shade[ball color=white] (a) circle (0.1); 
\node[above, scale=0.3] at (-1,-0.7)  {$\bullet$};  

\foreach \x in {-3.5,-1.5,.5,2.5} {
    \draw (\x,-2) -- (\x,2);
}
\end{tikzpicture}
\begin{tikzpicture}[scale=2] 

\draw[line width=1.5pt] (0,0) -- (1,0) -- (1,1) -- (0,1) -- cycle; 
\draw[line width=1.5pt] (0,1) -- (1,1) -- (1.5,1.5) -- (0.5,1.5) -- cycle; 
\draw[line width=1.5pt] (1,0) -- (1,1) -- (1.5,1.5); 
\draw[line width=1.5pt] (0,0)--(0,1)--(0.5,1.5); 

\draw[dotted, blue, line width=1.2pt] (0.5,1.5) -- (0.5,0.5)--(0,0);

\draw[dotted, blue, line width=1.2pt] (1.5,1.5) -- (1.5,0.5)--(1,0);

\draw[line width=1.5pt] (1,0) -- (2,0) -- (2,1) -- (1,1) -- cycle; 
\draw[line width=1.5pt] (1,1) -- (2,1) -- (2.5,1.5) -- (1.5,1.5) -- cycle; 
\draw[line width=1.5pt] (2,0) -- (2.5,0.5) -- (2.5,1.5) -- (2,1) -- cycle; 

\draw[dotted, blue, line width=1.2pt] (0.5,0.5) -- (2.5,0.5);

\coordinate (A) at (0.75,0.75); 
\coordinate (F) at (1.25,0.75); 
\coordinate (C) at (0.25,0.4); 
\coordinate (D) at (0.8,0.25); 
\coordinate (G) at (1,1.25); 

\coordinate (G1) at (.35,.85); 
\shade[ball color=cyan] (G1) circle (0.15);
\node[above, scale=1] at (0.35,.75) {$z_{i_3}$};

\shade[ball color=cyan] (A) circle (0.15);
\shade[ball color=cyan] (C) circle (0.15);
\shade[ball color=cyan] (D) circle (0.15);
\shade[ball color=cyan] (F) circle (0.15);
\shade[ball color=cyan] (G) circle (0.15);

\node[above, scale=1] at (0.75,0.65) {$z_{i_1}$};
\node[above, scale=1] at (1.25,0.65) {$z_{i_2}$};
\node[above, scale=1] at (0.25,.3) {$z_{i_4}$};
\node[above, scale=1] at (0.8,.2) {$z_{i_5}$};
\node[above, scale=1] at (1,1.1) {$z_{i_6}$};
\node[above, scale=0.3] at (1.2,0.7) {$\bullet$};
\node[above, scale=0.3] at (0.7,0.7) {$\bullet$};
\node[above, scale=0.3] at (0.2,.45) {$\bullet$};
\node[above, scale=0.3] at (0.75,.25) {$\bullet$};
\node[above, scale=0.3] at (1,1.2) {$\bullet$};


\coordinate (B_1) at (1.75,0.75);
\shade[ball color=yellow] (B_1) circle (0.15);
\coordinate (B_2) at (1.75,0.25); 
\shade[ball color=yellow] (B_2) circle (0.15);
\coordinate (B_3) at (2.3,0.5); 
\shade[ball color=yellow] (B_3) circle (0.15);
\coordinate (B_4) at (2.33,1.1); 
\shade[ball color=yellow] (B_4) circle (0.15);
\coordinate (B_5) at (2.15,.8);
\shade[ball color=yellow] (B_5) circle (0.15);

\node[above, scale=.8] at (1.75,0.7) {$z_{m_1}$};
\node[above, scale=0.3] at (1.73,0.7) {$\bullet$};
\node[above, scale=.8] at (1.73,0.2) {$z_{m_2}$};
\node[above, scale=0.3] at (1.73,0.2) {$\bullet$};
\node[above, scale=.8] at (2.27,0.45) {$z_{m_3}$};
\node[above, scale=0.3] at (2.23,0.45) {$\bullet$};
\node[above, scale=.8] at (2.31,1.05) {$z_{m_4}$};
\node[above, scale=0.3] at (2.31,1.05) {$\bullet$};
\node[above, scale=.8] at (2.15,.75) {$z_{m_5}$};
\node[above, scale=0.3] at (2.13,.75) {$\bullet$};

\node[below,scale=1] at (0.5,0.25) {$\Omega_i$};
\node[below, scale=1] at (1.5,0.25) {$\Omega_m$};



\end{tikzpicture}
    \end{center}
    \caption{A schematic illustration for the global and local distribution of the bubbles in $\mathbf{\Omega}$.}
\end{figure}
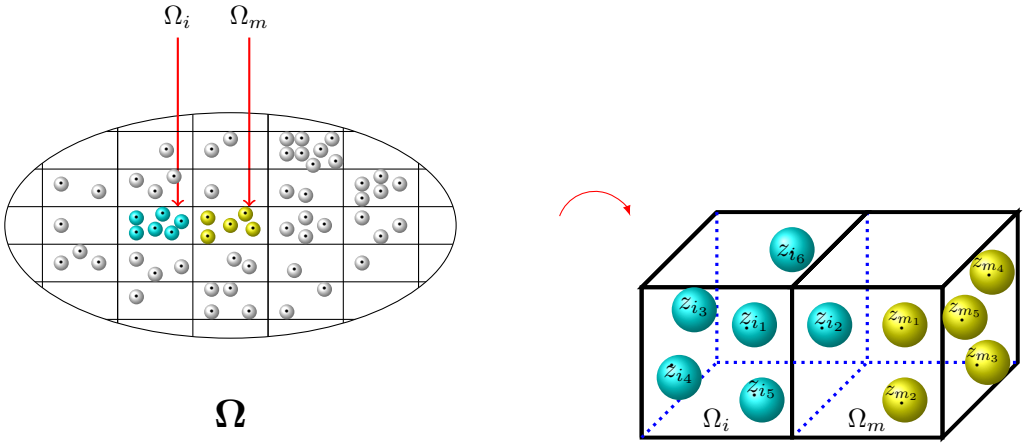


\begin{remark}
\noindent
    Given that \(\mathbf{\Omega}\) can have an arbitrary shape, the set of cubes \(\Omega_i\) intersecting \(\partial\mathbf{\Omega}\) is not empty unless \(\mathbf{\Omega}\) has a simple shape (such as a cube). In our subsequent analysis, we will require an estimation of the volume of this set. Since each \(\Omega_i\) has a volume of the order \(\varepsilon\), its maximum radius is of the order \(\varepsilon^{\frac{1}{3}}\), implying that the intersecting surfaces with \(\partial\mathbf{\Omega}\) have an area of the order \(\varepsilon^{\frac{2}{3}}\). Considering the area of \(\partial\mathbf{\Omega}\) to be of the order of one, we conclude that the number of such cubes will not exceed the order \(\varepsilon^{-\frac{2}{3}}\). Therefore, the volume of this set will not surpass the order \(\varepsilon \cdot \varepsilon^{-\frac{2}{3}} = \varepsilon^{\frac{1}{3}}\), as \(\varepsilon \to 0\).
\end{remark}

\noindent
We then state the main result of this work.
  
\begin{theorem}\label{non-periodic}
Let the conditions in Assumption \ref{as2} be satisfied. Then, under the following additional condition with \(K_\text{max} := \sup_{z_{m_l}}\big(K(z_{m_l}) + 1\big)\),
\begin{equation}\label{condition}
    \sqrt{K_\text{max}} \sum_{\substack{j=1 \\ j \neq m}}^M \sum_{\substack{i,l=1 \\ i \neq l}}^{[K(z_{m_l}) + 1]} \mathrm{q}_{\mathrm{z}_{m_l},\mathrm{z}_{j_i}} < \min\limits_{1 \le l \le K + 1}\hbar_{m_l},
\end{equation}
we have the following asymptotic expansion for \((x,t) \in \mathbb{R}^3 \setminus \overline{\mathbf{\Omega}} \times (0,T)\):
\begin{align}
    u^s(x,t) - \mathbcal{W}^s(x,t) = \mathcal{O}(\delta^\frac{1}{3}) \; \text{as} \; \delta \to 0,
\end{align}
where \(\mathbcal{W}(x,t) = \mathbcal{W}^s + u^\textbf{in}\) is the solution of the following dispersive acoustic model:
\begin{equation} \label{pde}
\begin{cases}
    \big(c_0^{-1} \partial^2_t - \Delta\big)\Big(\rchi_{\mathbf{\Omega}}\; \hbar\ \partial^2_t \mathbcal{W} + \mathbcal{W}\Big) + \rchi_{\mathbf{\Omega}}\; [K + 1] b \cdot \partial^2_t \mathbcal{W} = 0 & \text{in} \; \mathbb{R}^3 \times (0,T),\\
    \partial_t^i \mathbcal{W}(x,0) = 0 & \text{in} \; \mathbb{R}^3;\; i = 0,1,2,3.
\end{cases}
\end{equation}
The dispersive acoustic model (\ref{pde}) is equivalent to the following integro-differential equation with $\mathrm{P}:=\hbar\; \rchi_{\mathbf{\Omega}}\; \partial^2_{t} \mathbcal W+\mathbcal W$:
\begin{align}\label{intego-differential-main}
    \Big( c_0^{-1}\partial^2_t-\Delta + \frac{b}{\hbar}[K+1]\rchi_\mathbf{\Omega} \Big)\mathrm{P}(x,t) - \rchi_{\mathbf{\Omega}}\hbar^{-\frac{3}{2}}[K+1]\int\limits_0^t\sin\Big(\hbar^{-\frac{1}{2}}(t-\tau)\Big)\mathrm{P}(x,\tau)d\tau = 0.
\end{align}
The above integro-differential equation has a unique solution in $\mathrm{H}^{\mathrm{r}}_{0,\sigma}\big(0,\mathrm{T};\mathrm{H}^1(\mathbb R^3)\big)$ for $\mathrm{P}^\textbf{in} \in \mathrm{H}^{\mathrm{r}+3}_{0,\sigma}\big(0,\mathrm{T};\mathrm{L}^2(\mathbb R^3)\big),$ see Section \ref{indi} for more details.
\end{theorem}

\noindent The model (\ref{pde}) can be expressed, recalling that $\mathrm{P}:=\hbar\; \rchi_{\mathbf{\Omega}}\; \partial^2_{t} \mathbcal W+\mathbcal W$, via the dispersive acoustic model:
\begin{align}\label{math-model-dispersive-media}
\begin{cases}   
\partial_{t} \mathrm{U}+\nabla \mathrm{P}& = 0, \\
c_0^{-1}\partial_{t} \mathrm{P}+\nabla \cdot \mathrm{U} & = -b\;[K+1]\ \rchi_{\mathbf{\Omega}}\; \partial_{t} \mathbcal W, \\
 \hbar\; \rchi_{\mathbf{\Omega}}\; \partial^2_{t} \mathbcal W+\mathbcal W-\mathrm{P} &=0.
\end{cases}
\end{align}
Here, $\mathrm{P}$ and $\mathrm{U}$ represent the acoustic pressure and velocity fields respectively. The third equation plays a similar role as the electromagnetic susceptibility in modeling electromagnetic dispersive media. 

\noindent
Notice that when $\hbar\ll 1$, the model (\ref{math-model-dispersive-media}) reduces to the non-dispersive acoustic model
\begin{align}\label{math-model-non-dispersive-media}
\begin{cases}   
\partial_{t} \mathrm{U}+\nabla \mathrm{P}& = 0, \\
\big(c_0^{-1}+b\;[K+1]\ \rchi_{\mathbf{\Omega}}\big)\partial_{t} \mathrm{P}+\nabla \cdot \mathrm{U} & =0.
\end{cases}
\end{align}
Therefore, the coefficient $\hbar$ is responsible for the dispersion effect. Indeed, this coefficient is build up from the Minnaert resonance $\omega_M:=\sqrt{ \frac{\rho_c}{2\overline{k}_b}\mathrm{A}_{\partial B}}$, i.e. $\omega_M^2=\hbar$, as derived from (\ref{Minnaert-coefficient}). The existence of this resonant frequency in the time-harmonic regime translates into a dispersion effect in the time domain regime.

\bigskip

\noindent The derived model, and the needed analysis, provided here give the correct model to the linearized effective model for bubbly media formally derived in the original, and nowadays classical, work \cite{C-M-P-T-2} by  R. Caflisch, M. Miksis, G. Papanicolaou and L. Ting, see the effective equation (2.26) (or its time-harmonic counter part in (4.35)) there and compare it with the equation (\ref{pde}) above. 
\bigskip

\noindent At the analysis level, we mainly use the related Lippmann-Schwinger equation (LSE). We distinguish two main steps.
\begin{enumerate}
\item \textit{Invertibility (i.e. a priori estimate of the solution) of the LSE and regularity property}. This step is handled using the Fourier-Laplace transform and its inverse between the spaces of the type $\mathrm{H}^{\mathrm{r}}_{0,\sigma}\big(0,\mathrm{T};\mathrm{H}^1(\mathbb R^3)\big)$. As $\sigma$ is positive, then in the Fourier-Laplace domain, we can have good control of the coercivity and hence the invertibility in terms of the Fourier-Laplace parameter. Roughly speaking, contrary to the case when we use pure Fourier-transform, here we end up with the analysis in the resolvent set of the related Laplacian in terms of the Fourier-Laplace parameter (while using the Fourier transform, we end up in the spectral set of the Laplacian). The price to pay is to derive estimate in weighted space in terms of time. But these estimates are enough for our purpose. 

\item \textit{Taylor's expansions and matching the linear system (\ref{matrixmulti}) with the LSE}. Before performing such matching, we first rewrite the linear system (\ref{matrixmulti}) according to the distribution of the bubbles as described in Assumption \ref{as2}. The idea here is to reduce this system, originally stated on every bubbles' location point, i.e. on the $D_j$'s, to another system stated only on any selected bubble for each $\Omega_j$, i.e. stated only on the $\Omega_j$'s. The advantage here is that the set of $\Omega_j$ is distributed periodically in $\Omega$. The price to pay is that the reformulated linear system has larger dimensions for their unknowns, see Section \ref{Reformulation-Linear-System} for more explanations. This reformulated linear system indicates which continuous integral equations (and hence Lippmann-Schwinger equation) we can obtain. As the $\Omega_j$'s are periodically distributed in $\Omega$, then matching this linear system with its continuous integral equation needs only the regularity of the solutions, that is derived in step 1.  
\end{enumerate}


\bigskip

We finish this introduction with two observations.
\begin{enumerate}
\item  In the first one, we describe ways how one can ensure the conditions in (\ref{condition}) to generate the effective medium \footnote{In real applications, taking $K$ to be a uniform constant is already satisfactory. In this case, we can choose the shapes of the injected 
 bubbles and also the mass density and the bulk modulus of each bubble to be the same. Therefore, the effective medium can be designed using one type of gas and bubble's shape}. We can simplify the condition mentioned in (\ref{condition}) into the following situation where we define for each bubble, $r_\text{min} = \min\limits_{\substack{1 \le l \le [K + 1] \\ 1 \le m \le M}}\text{radius}(B_{m_l})$,~ 
\\$r_\text{max} = \max\limits_{\substack{1 \le l \le [K + 1] \\ 1 \le m \le M}}\text{radius}(B_{m_l})$ and $A_\text{max} := \max\limits_{\substack{1 \le m,j \le M \\ 1 \le l \le [K(z_{m_l}) + 1]}} \sum\limits_{\substack{j=1 \\ j \neq m}}^M \sum\limits_{\substack{i,l=1 \\ i \neq l}}^{[K(z_{m_l}) + 1]} \frac{\delta}{|z_{m_l} - z_{j_i}|}$.
Additionally, \(A_{\partial B}\) and \(\text{vol}(B)\) are scaled geometric constants with \(B\) being of radius 1.
Let us first recall the condition we obtain
$$\sqrt{K_\text{max}} \sum\limits_{\substack{j=1 \\ j\neq m}}^M \sum\limits_{\substack{i,l=1 \\ i\neq l}}^{[K(z_{m_l})+1]} \mathrm{q}_{\mathrm{z}_{m_l},\mathrm{z}_{j_i}} <  \min\limits_{1\le l\le K+1}\hbar_{m_l},$$
where we set $K_\text{max} := \sup_{z_{m_l}}\big(K(z_{m_l}) + 1\big)$. We also recall that after neglecting the corresponding error terms, we have \( b = \frac{\rho_c}{\overline{\mathrm{k}}_b}\text{vol}(B_{m_l}) \) and \( \hbar_{m_l} = \frac{\rho_c}{2\overline{k}_b}\mathrm{A}_{\partial B_{m_l}} \), respectively. Therefore, we have \( \mathrm{q}_{\mathrm{z}_{m_l},\mathrm{z}_{j_i}} = \frac{\rho_c}{\overline{\mathrm{k}}_b}\text{vol}(B_{m_l})\frac{\delta}{|z_{m_m} - z_{j_i}|} \). Then, if we consider these scaling properties, the aforementioned conditions reduce to:
$$\sqrt{K_\text{max}} \sum\limits_{\substack{j=1 \\ j \neq m}}^M \sum\limits_{\substack{i,l=1 \\ i \neq l}}^{[K(z_{m_l}) + 1]}\frac{\delta}{|z_{m_l} - z_{j_i}|} < \frac{A_{\partial B_{m_l}}}{2\text{vol}(B_{m_l})}.$$
In addition, from the definition of the geometric constant $\mathrm{A}_{\partial B_{m_l}}$ and $\text{vol}(B_{m_l})$, we see that
$$ \mathrm{A}_{\partial B_{m_l}} \le r_\text{min}^2\mathrm{A}_{\partial B}\ \text{and}\ \text{vol}(B_{m_l}) \geq r^3_{\text{max}}\text{vol}(B)$$
where $B$ has a (maximum) radius $1$. Therefore, we arrive at the following condition
\begin{equation}
 K_\text{max}< \Big(\frac{A_{\partial B}}{2A_\text{max}\text{vol}(B)}\Big)^2\frac{r^4_\text{min}}{r^6_\text{max}}.
\end{equation}
In this case, given the function $K$ and the bubbles of the forms $B_{m_l}$ (which can be taken to be the same, i.e. $B_{m_l}$=:$B$), we use gases having mass density and bulk modulus so that the radius $r_\text{min}=:r$ and $r_\text{max}=:r$ should be taken in a way so that  $K_\text{max} <\Big(\frac{A_{\partial B}}{2A_\text{max}\text{vol}(B)}\Big)^2\frac{1}{r^2}$. 

\bigskip

\noindent 
Therefore, if we carefully adjust the material properties of the bubble and have knowledge of \( K_\text{max} \), or if we fine-tune the geometric properties of the bubbles, we can validate the generated effective media. 
\bigskip

\item 
\noindent
As a second observation, let us mention that the condition in (\ref{condition}) is needed to ensure the generation of the effective medium. The usual condition to link the linear algebraic system and the related Lippmann-Schwinger equation is of the form
\begin{equation}\label{discret-continous}
\max_{j=1, ..., M}\Big \vert\frac{1}{M}\sum^M_{j\neq i}\frac{\psi(x_i)}{\vert x_j-x_i\vert}-\int_\Omega\frac{\psi(y)\theta(y)}{\vert x_j-y\vert}dy \Big\vert \longrightarrow 0,~~ M \longrightarrow \infty
\end{equation}
to relate the discrete sum to the integral, where $x_j, j=1, ..., M,$ are points distributed in a bounded $\Omega$ fora  smooth function $\psi$, see \cite{AZ-17, C-M-P-T-1, C-M-P-T-2, FPR} for instance. Here, $\theta$ is intended to model the distribution of the points $x_j$'s. For the special case where these points are distributed periodically, and hence $\theta$ is identical to $1$, this approximation should be possible. In general this is unclear. But, if the points are distributed according to the way how we described it in Assumption \ref{as2}, with $\theta(\cdot):=K(\cdot)+1$, then we have with $\mathrm{N}:= [\varepsilon^{-1}]$ the following property
\begin{equation}\label{discret-continous-localised}
\max_{j=1, ..., \mathrm{N}} \Big \vert \frac{1}{\mathrm{N}} \sum^{N}_{j\neq i}\sum^{[\theta(x^*_{j})]}_{l=1}\frac{\psi(x_{i_l})}{\vert x_{l_j}-x_{i_l}\vert}-\int_\Omega\frac{\psi(y)\big[\theta(y)\big]}{\vert x_j-y\vert}dy \Big\vert \longrightarrow 0, ~~ \mathrm{N} \longrightarrow \infty
\end{equation}
where for every $j=1, ..., N$, $x^*_{j}$ is anyone of the points $x_{j_l}$ in $\Omega_j$.  The property (\ref{discret-continous-localised}) is shown and used in Section \ref{K-natural-number} and Section \ref{K-function}. Observe that, (\ref{discret-continous}) and (\ref{discret-continous-localised}) coincide if $\theta \equiv 1$ (or $K(\cdot)\equiv0$), as $N=M$ in this case. But, they do not coincide if $\theta$ is different from the unity.
\noindent

\end{enumerate}

\bigskip
\noindent
The remaining sections of this work are structured as follows. In Section \ref{indi}, we introduce the necessary function spaces for the mathematical analysis. Following that, we provide the justification for the uniqueness and existence of solutions to problem (\ref{pde}) or (\ref{intego-differential-main}). In Sections \ref{th1} and \ref{th2}, we present the proofs of Theorem \ref{non-periodic} under Assumptions \ref{as2}, respectively. These sections address the asymptotic expansions of the generated equivalent pressure field used to approximate the scattered field produced by the cluster of bubbles.

\section{The well-posedness of the integro-differential equation (\ref{intego-differential-main})}\label{indi} 

\noindent
In this section, our aim is to demonstrate the existence and uniqueness of the time-domain integro-differential equation (\ref{intego-differential-main}). We follow the method outlined by Lubich \cite{lubich}, which involves analyzing the time-domain Lippmann-Schwinger equation using Laplace-Fourier transform. Here, we give the proof for the case when $K \equiv 0$, and a similar proof can be shown for the case when $K \in \mathbb{N}$, or more generally, when $K$ is a variable valued function. 

\noindent
We first rewrite the model (\ref{pde}) in terms of only the pressure $\mathrm{P}:=\hbar\; \rchi_{\mathbf{\Omega}}\; \partial^2_{t} \mathbcal W+\mathbcal W$. We notice that $\mathrm{P}$ satisfies the following integro-differential equation:
\begin{align}\label{intego-differential}
    \Big( c_0^{-1}\partial^2_t-\Delta + \frac{b}{\hbar}\rchi_\mathbf{\Omega} \Big)\mathrm{P}(x,t) - \rchi_{\mathbf{\Omega}}\hbar^{-\frac{3}{2}}\int\limits_0^t\sin\Big(\hbar^{-\frac{1}{2}}(t-\tau)\Big)\mathrm{P}(x,\tau)d\tau = 0.
\end{align}
Now, we study the well-posedness of the above problem. Invertly, it is clear that if $\mathrm{P}$ satisfies (\ref{intego-differential}), then $\mathbcal W$ given by 
$$\begin{cases}
    \hbar\; \rchi_{\mathbf{\Omega}}\; \partial^2_{t} \mathbcal W+\mathbcal W = \mathrm{P},\\
    \mathbcal W(0) = \partial_t \mathbcal W(0) = 0,
\end{cases}$$
satisfies (\ref{pde}). Similarly, the weak formulation of (\ref{pde}) is equivalent to the one of (\ref{intego-differential}).
To begin with, we consider the following elliptic problem: 
\begin{align}\label{laplace}
    \mathrm{c}_0^{-1} \bm{\mathrm{s}}^2 \Tilde{\mathrm{P}}^\textbf{sc}(\mathrm{x},\bm{\mathrm{s}}) - \Delta\Tilde{\mathrm{P}}^\textbf{sc}(\mathrm{x},\bm{\mathrm{s}}) + \frac{b}{\hbar}\rchi_\mathbf{\Omega} \Tilde{\mathrm{P}}^\textbf{sc}(\mathrm{x},\bm{\mathrm{s}}) - \rchi_\mathbf{\Omega}b\ \frac{1}{\hbar}\frac{1}{\hbar\mathbf{\mathrm{s}}^2+1}\Tilde{\mathrm{P}}^\textbf{sc}(\mathrm{x},\bm{\mathrm{s}}) = \rchi_\mathbf{\Omega}b\ \frac{1}{\hbar}\frac{1}{\hbar\mathbf{\mathrm{s}}^2+1}\Tilde{\mathrm{P}}^\textbf{in}(\mathrm{x},\bm{\mathrm{s}}) - \frac{b}{\hbar}\rchi_\mathbf{\Omega} \Tilde{\mathrm{P}}^\textbf{in}(\mathrm{x},\bm{\mathrm{s}}).
\end{align}
The above equation can be seen as the Laplace transform to the equation (\ref{intego-differential}), with respect to the time variable, where $\bm{\mathrm{s}}= \sigma + i\omega \in \mathbb{C}$ is the transform parameter with $\sigma \in \mathbb{R}, \sigma>\sigma_0>0,$ for some constant $\sigma_0$, and $\omega\in \mathbb{R}.$

\noindent
Next, we develop a variational method for the aforementioned problem (\ref{laplace}) and utilize the Lax-Milgram Lemma.
\noindent
By multiplying equation (\ref{laplace}) by the complex conjugate of $\mathrm{v}\in \mathrm{H}^1(\mathbb{R}^3)$, and integrating over $\mathbb R^3$, we obtain a sesquilinear mapping $\mathrm{a}(\Tilde{\mathrm{P}}^\mathrm{s},\mathrm{v}): \mathrm{H}^1(\mathbb R^3) \times \mathrm{H}^1(\mathbb R^3) \to \mathbb{C}$ and an antilinear mapping $\mathrm{b}(\mathrm{v}): \mathrm{H}^1(\mathbb R^3) \to \mathbb{C}$, such that
\begin{align}\label{bi-li}
     \mathrm{a}(\Tilde{\mathrm{P}}^\mathrm{s},\mathrm{v}) = \mathrm{b}(\mathrm{v}) \quad \text{for all}\; \mathrm{v}\in \mathrm{H}^1(\mathbb R^3),
\end{align}
where
\begin{align}
\mathrm{a}(\Tilde{\mathrm{P}}^\mathrm{s},\mathrm{v}) = \int_{\mathbb{R}^3}\mathrm{c}_0^{-1} 
    \bm{\mathrm{s}}^2 \Tilde{\mathrm{P}}^\mathrm{s}\;\overline{\mathrm{v}} d\mathrm{x} + \int_{\mathbb{R}^3} \nabla\Tilde{u}^\mathrm{s}\cdot\overline{\nabla\mathrm{v}} d\mathrm{x} + \int_{\mathbb{R}^3}\frac{b}{\hbar}\rchi_\mathbf{\Omega} \Tilde{\mathrm{P}}^\mathrm{s}\;\overline{\mathrm{v}} d\mathrm{x}- \int_{\mathbb R^3}\rchi_\mathbf{\Omega} \frac{b}{\hbar}\frac{1}{\hbar\mathbf{\mathrm{s}}^2+1}\Tilde{\mathrm{P}}^\mathrm{s}\;\overline{\mathrm{v}} d\mathrm{x},
\end{align}
and
\begin{align}
\nonumber
    \mathrm{b}(\mathrm{v}) = \int_{\mathbb R^3}\rchi_\mathbf{\Omega} \frac{b}{\hbar}\frac{1}{\hbar\mathbf{\mathrm{s}}^2+1}\Tilde{\mathrm{P}}^\mathrm{s}\;\overline{\mathrm{v}} d\mathrm{x} - \int_{\mathbb{R}^3}\frac{b}{\hbar}\rchi_\mathbf{\Omega} \Tilde{\mathrm{P}}^\mathrm{s}\;\overline{\mathrm{v}} d\mathrm{x}.
\end{align}
To verify the coercivity of the above bi-linear form, we choose $\mathrm{v} = \bm{\mathrm{s}}\Tilde{\mathrm{P}}^\mathrm{s}$ and use integration by parts to obtain
\begin{align}
    \mathrm{a}(\Tilde{\mathrm{P}}^\textbf{sc},\bm{\mathrm{s}}\Tilde{\mathrm{P}}^\textbf{sc}) = \int_{\mathbb{R}^3}\mathrm{c}_0^{-1} \overline{\bm{\mathrm{s}}}|\bm{\mathrm{s}}|^2|\Tilde{\mathrm{P}}^\textbf{sc}|^2 d\mathrm{x} + \int_{\mathbb{R}^3} \overline{\bm{\mathrm{s}}}|\nabla\Tilde{\mathrm{P}}^\textbf{sc}|^2 d\mathrm{x} + \int_{\mathbb{R}^3} \frac{b}{\hbar}\rchi_\mathbf{\Omega} \overline{\bm{\mathrm{s}}}|\Tilde{\mathrm{P}}^\textbf{sc}|^2 d\mathrm{x} - \int_{\mathbb{R}^3}\rchi_\mathbf{\Omega} \frac{b}{\hbar}\frac{\overline{\bm{\mathrm{s}}}}{\hbar\mathbf{\mathrm{s}}^2+1}|\Tilde{\mathrm{P}}^\textbf{sc}|^2 d\mathrm{x},
\end{align}
which is after rewriting equivalent to the following
\begin{align}
    \mathrm{a}(\Tilde{\mathrm{P}}^\textbf{sc},\bm{\mathrm{s}}\Tilde{\mathrm{P}}^\textbf{sc}) = \int_{\mathbb{R}^3}\mathrm{c}_0^{-1} \overline{\bm{\mathrm{s}}}|\bm{\mathrm{s}}|^2|\Tilde{\mathrm{P}}^\textbf{sc}|^2 d\mathrm{x} + \int_{\mathbb{R}^3} \overline{\bm{\mathrm{s}}}|\nabla\Tilde{\mathrm{P}}^\textbf{sc}|^2 d\mathrm{x} + \int_{\mathbb{R}^3} \frac{b}{\hbar}\rchi_\mathbf{\Omega} \overline{\bm{\mathrm{s}}}|\Tilde{\mathrm{P}}^\textbf{sc}|^2 d\mathrm{x} - \int_{\mathbb{R}^3}\rchi_\mathbf{\Omega} b\ \Big(\frac{\bm{\mathrm{s}}|\mathbf{s}|^2}{\hbar\mathbf{s}^2}-\frac{\bm{\mathrm{s}}|\mathbf{s}|^2}{\hbar\mathbf{\mathrm{s}}^2+1}\Big)|\Tilde{\mathrm{P}}^\textbf{sc}|^2 d\mathrm{x}.
\end{align}
Therefore, we obtain that
\begin{align}\label{bi}
    \mathrm{a}(\Tilde{\mathrm{P}}^\textbf{sc},\bm{\mathrm{s}}\Tilde{\mathrm{P}}^\textbf{sc}) = \int_{\mathbb{R}^3}\mathrm{c}_0^{-1} \overline{\bm{\mathrm{s}}}|\bm{\mathrm{s}}|^2|\Tilde{\mathrm{P}}^\textbf{sc}|^2 d\mathrm{x} + \int_{\mathbb{R}^3} \overline{\bm{\mathrm{s}}}|\nabla\Tilde{\mathrm{P}}^\textbf{sc}|^2 d\mathrm{x} + \int_{\mathbb{R}^3}\rchi_\mathbf{\Omega} b\ \frac{\bm{\mathrm{s}}|\mathbf{s}|^2}{\hbar\mathbf{\mathrm{s}}^2+1}|\Tilde{\mathrm{P}}^\textbf{sc}|^2 d\mathrm{x}.
\end{align}
and
\begin{align}
    \mathrm{b}(\bm{\mathrm{s}}\Tilde{\mathrm{P}}^\textbf{sc}) &\nonumber= \Big\langle\rchi_\mathbf{\Omega} b\ \frac{\bm{\mathrm{s}}|\mathbf{s}|^2}{\hbar\mathbf{\mathrm{s}}^2\big(\hbar\mathbf{\mathrm{s}}^2+1\big)}\Tilde{\mathrm{P}}^\textbf{in},\Tilde{\mathrm{P}}^\textbf{sc}\Big\rangle -\Big\langle \frac{b}{\hbar}\rchi_\mathbf{\Omega} \overline{\bm{\mathrm{s}}}\Tilde{\mathrm{P}}^\textbf{in},\Tilde{\mathrm{P}}^\textbf{sc}\Big\rangle
    \\ &=\Big\langle\rchi_\mathbf{\Omega} b\ \frac{\bm{\mathrm{s}}|\mathbf{s}|^2}{\hbar\mathbf{\mathrm{s}}^2+1}\Tilde{\mathrm{P}}^\textbf{in},\Tilde{\mathrm{P}}^\textbf{sc}\Big\rangle,
\end{align}
where $\big\langle \cdot,\cdot\big\rangle$ denotes the usual inner product between $\mathrm{H}^{1}(\mathbb R^3)$ and $\mathrm{H}^{-1}(\mathbb R^3).$ Now, after taking real parts, we have
\begin{align}
    \Re\Big(\mathrm{a}(\Tilde{\mathrm{P}}^\textbf{sc},\bm{\mathrm{s}}\Tilde{\mathrm{P}}^\textbf{sc})\Big) = \Re\Bigg(\int_{\mathbb{R}^3}\mathrm{c}_0^{-1} \overline{\bm{\mathrm{s}}}|\bm{\mathrm{s}}|^2|\Tilde{\mathrm{P}}^\textbf{sc}|^2 d\mathrm{x} + \int_{\mathbb{R}^3} \overline{\bm{\mathrm{s}}}|\nabla\Tilde{\mathrm{P}}^\textbf{sc}|^2 d\mathrm{x}\Bigg) + \Re\Big(\int_{\mathbb{R}^3}\rchi_\mathbf{\Omega} b\ \frac{\bm{\mathrm{s}}|\mathbf{s}|^2}{\hbar\mathbf{\mathrm{s}}^2+1}|\Tilde{\mathrm{P}}^\textbf{sc}|^2 d\mathrm{x}\Big).
\end{align}
Next, we deduce
\begin{align}
    \Re\Big(b\frac{\bm{\mathrm{s}}|\mathbf{s}|^2}{\hbar\mathbf{\mathrm{s}}^2+1}\Big) 
    &\nonumber= b\frac{\sigma\Big(\hbar(\sigma^2-\omega^2)+1\Big)+2\sigma\omega^2\hbar}{\Big(\hbar(\sigma^2-\omega^2)+1\Big)^2+4\sigma^2\omega^2\hbar^2}(\sigma^2+\omega^2)
    \\ &\nonumber= b \sigma (\sigma^2+\omega^2) \frac{\hbar(\sigma^2+\omega^2)+1}{\Big(\hbar(\sigma^2-\omega^2)+1\Big)^2+4\sigma^2\omega^2\hbar^2} \ge 0,
\end{align}
which shows that $\displaystyle\Re\Big(\int_{\mathbb{R}^3}\rchi_\mathbf{\Omega} b\ \frac{\bm{\mathrm{s}}|\mathbf{s}|^2}{\hbar\mathbf{\mathrm{s}}^2+1}|\Tilde{\mathrm{P}}^\textbf{sc}|^2 d\mathrm{x}\Big) \ge 0.$

\noindent
Consequently, we show that
\begin{align}\label{one}
    | \mathrm{a}(\Tilde{\mathrm{P}}^\textbf{sc},\bm{\mathrm{s}}\Tilde{\mathrm{P}}^\textbf{sc})| \ge \min \{\sigma,\sigma^3\} \Vert \Tilde{\mathrm{P}}^\textbf{sc}(\cdot,\bm{\mathrm{s}})\Vert^2_{\mathrm{H}^{1}(\mathbb R^3)},
\end{align}
As $\rchi_{\mathbf{\Omega}} \in \mathrm{L}_2(\mathbb R^3)$, we also have 
\begin{align}\label{two}
    |\mathrm{b}(\bm{\mathrm{s}}\Tilde{\mathrm{P}}^\textbf{sc})| \le b\Big\Vert \rchi_\mathbf{\Omega} \frac{\bm{\mathrm{s}}|\mathbf{s}|^2}{\hbar\mathbf{\mathrm{s}}^2+1}\Tilde{\mathrm{P}}^\textbf{in}\Big\Vert_{\mathrm{L}^{2}(\mathbb R^3)} \Vert \Tilde{\mathrm{P}}^\textbf{sc}(\cdot,\bm{\mathrm{s}})\Vert_{\mathrm{H}^{1}(\mathbb R^3)}.
\end{align}
Therefore, considering (\ref{one}) and (\ref{two}) we deduce that the elliptic problem (\ref{laplace}) has unique solution and it satisfies
\begin{align}
    \Vert \Tilde{\mathrm{P}}^\textbf{sc}(\cdot,\bm{\mathrm{s}})\Vert_{\mathrm{H}^{1}(\mathbb R^3)} \le \frac{b}{\hbar}\ \frac{1}{{\sigma}} \Bigg|\frac{\bm{\mathrm{s}}|\mathbf{s}|^2}{\mathbf{\mathrm{s}}^2+\frac{1}{\hbar}}\Bigg| \Big\Vert \rchi_\mathbf{\Omega}\Tilde{\mathrm{P}}^\textbf{in}\Big\Vert_{\mathrm{L}^{2}(\mathbb R^3)}.
\end{align}
Therefore, we use the fact that $\Re(s) = \sigma > \sigma_0; \; \sigma_0 \in \mathbb R^+,$ to deduce the following
\begin{align}\label{fact}
    \Vert \Tilde{\mathrm{P}}^\textbf{sc}(\cdot,\bm{\mathrm{s}})\Vert_{\mathrm{H}^{1}(\mathbb R^3)} \le b \frac{|s|^3}{\sigma}\Big\Vert \rchi_\mathbf{\Omega}\Tilde{\mathrm{P}}^\textbf{in}\Big\Vert_{\mathrm{L}^{2}(\mathbb R^3)}.
\end{align}
Consequently, we use the similar techniques as \cite[Section 4.1]{Arpan-Sini-SIMA} to derive that the equation (\ref{intego-differential}) has a unique solution in $\mathrm{H}^{\mathrm{r}}_{0,\sigma}\big(0,\mathrm{T};\mathrm{H}^1(\mathbb R^3)\big)$ for $\mathbf{P}^\textbf{in} \in \mathrm{H}^{\mathrm{r}+3}_{0,\sigma}\big(0,\mathrm{T};\mathrm{L}^2(\mathbb R^3)\big).$

\noindent
Indeed, let us define the inverse Laplace transform of $\Tilde{\mathrm{P}}^\textbf{sc}(\mathrm{x},\cdot)$ for $\Re(\bm{\mathrm{s}}) =\sigma>0$ as
\begin{align}\label{invlap}
    \mathrm{P}^\textbf{sc}(\mathrm{x},t):= \frac{1}{2\pi i}\int_{\sigma-i\infty}^{\sigma+i\infty}e^{\bm{\mathrm{s}}\mathrm{t}}\Tilde{\mathrm{P}}^\textbf{sc}(\mathrm{x},\bm{\mathrm{s}})d\bm{\mathrm{s}} = \frac{1}{2\pi}\int_{-\infty}^{\infty}e^{(\sigma+i\omega)\mathrm{t}}\Tilde{\mathrm{P}}^\textbf{sc}(\mathrm{x},\sigma+i\omega)d\omega.
\end{align}
Due to the estimate with respect to $`\bm{\mathrm{s}}`$ in (\ref{fact}), $\mathrm{P}^\textbf{sc}(\mathrm{x},\mathrm{t})$ is well-defined. In addition, one can show that $\mathrm{P}^\textbf{sc}(\mathrm{x},\mathrm{t})$ does not depend on $\sigma$ by utilizing a classical method of contour integration, see \cite[pp. 39]{sayas}.
If we consider the Fourier transform w.r.t time variable $\Fourier_\mathrm{t}$, then we have $\Fourier_{\mathrm{t}\to \omega}\big(e^{-\sigma \mathrm{t}} \partial^\mathrm{k}_\mathrm{t}\mathrm{P}^\textbf{sc}(\mathrm{x},\mathrm{t})\big)=\bm{\mathrm{s}}^\mathrm{k}\Tilde{\mathrm{P}}^\textbf{sc}(\mathrm{x},\bm{\mathrm{s}})$ with $\bm{\mathrm{s}}=\sigma+i\omega.$ Thus, we get for $\mathrm{r}\in \mathbb N$ the following
\begin{align}\label{well}
\nonumber
  \|\mathrm{P}^\textbf{sc}\|^2_{\mathrm{H}^{\mathrm{r}}_{0,\sigma}\big((0,\mathrm{T}); \mathrm{H}^1(\mathbb R^3)\big)} & = \int_0^{\mathrm{T}} e^{-2\sigma \mathrm{t}} \sum_{\mathrm{k}=0}^{\mathrm{r}} \mathrm{T}^{2\mathrm{k}} \|\partial^\mathrm{k}_\mathrm{t} \mathrm{P}^\textbf{sc}(\cdot,\mathrm{t})\|^2_{\mathrm{H}^1(\mathbb R^3)} \,d\mathrm{t} \\
  & \nonumber \lesssim \int_{\mathbb{R}_+} \int_{\mathbb{R}^3} e^{-2\sigma \mathrm{t}} \sum_{\mathrm{k}=0}^{\mathrm{r}} \Big[|\partial^\mathrm{k}_\mathrm{t} \mathrm{P}^\textbf{sc}(\mathrm{x},\mathrm{t})|^2 + |\partial^\mathrm{k}_\mathrm{t} \nabla \mathrm{P}^\textbf{sc}(\mathrm{x},\mathrm{t})|^2\Big] \, d\mathrm{x}\, d\mathrm{t} \\
  & \nonumber \lesssim \int_{\mathbb{R}^3} \int_{\mathbb{R}} \sum_{\mathrm{k}=0}^{\mathrm{r}} \Big[\big|\Fourier ( e^{-\sigma \mathrm{t}} \partial^\mathrm{k}_\mathrm{t} \mathrm{P}^\textbf{sc}(\mathrm{x},\mathrm{t})\big|^2+ \big|\Fourier ( e^{-\sigma \mathrm{t}} \partial^\mathrm{k}_\mathrm{t} \nabla u^\mathrm{s}(\mathrm{x},\mathrm{t})\big|^2\Big] \, d\mathrm{t}\, d\mathrm{x} \\
  & \nonumber\lesssim \sum_{\mathrm{k}=0}^{\mathrm{r}} \int_{\sigma+i\mathbb{R}} |\bm{\mathrm{s}}|^{2\mathrm{k}} \|\Tilde{\mathrm{P}}^\text{sc}(\cdot,\bm{\mathrm{s}})\|^2_{\mathrm{H}^1(\mathbb R^3)} \, d\bm{\mathrm{s}} \\
  &  \lesssim \sum_{\mathrm{k}=0}^{\mathrm{r}+3} \int_{\sigma+i\mathbb{R}} |\bm{\mathrm{s}}|^{2\mathrm{k}} \|\rchi_\mathbf{\Omega}\Tilde{\mathrm{P}}^\textbf{in}(\cdot,\bm{\mathrm{s}})\|^2_{\mathrm{L}^{2}(\mathbb R^3)} \, d\bm{\mathrm{s}} \simeq \|\rchi_\mathbf{\Omega}\mathrm{P}^\textbf{in}\|^2_{\mathrm{H}^{\mathrm{r}+3}_{0,\sigma}\big((0,\mathrm{T}); \mathrm{L}^{2}(\mathbb R^3)\big)}.
\end{align}
It is now our aim to show that $\mathrm{P}^\text{sc}$, defined in (\ref{invlap}), is a weak solution to the problem (\ref{intego-differential}). To do so, we consider the following weak formulation of the problem (\ref{intego-differential})
\begin{align}
    \big\langle \mathrm{c}_0^{-1}\partial_t^2\mathrm{P}^\text{sc}(\cdot,\mathrm{t}),\varphi\big\rangle 
    +   \big\langle \nabla \mathrm{P}^\text{sc}(\cdot,\mathrm{t}),\nabla\varphi\big\rangle + \frac{b}{\hbar}\big\langle \mathrm{c}_0^{-1}\mathrm{P}^\text{sc}(\cdot,\mathrm{t}),\varphi\big\rangle 
    - \big\langle \rchi_{\mathbf{\Omega}}\hbar^{-\frac{3}{2}}\sin\big(\hbar^{-\frac{1}{2}}\big)\ast\mathrm{P}^\text{sc}(\cdot,t), \varphi\big\rangle = \big\langle \mathrm{F}(\cdot,\mathrm{t}),\varphi\big\rangle ,
\end{align}
for a.e. $ \mathrm{t}\in (0,\mathrm{T})\; \text{and}\ \forall \mathrm{v}\in \mathrm{H}^1(\mathbb R^3)$.

\noindent
We see that
\begin{align}
    &\nonumber\big\langle \mathrm{c}_0^{-1}\partial_t^2\mathrm{P}^\text{sc}(\cdot,\mathrm{t}),\varphi\big\rangle 
    +   \big\langle \nabla \mathrm{P}^\text{sc}(\cdot,\mathrm{t}),\nabla\varphi\big\rangle + \frac{b}{\hbar}\big\langle \mathrm{c}_0^{-1}\mathrm{P}^\text{sc}(\cdot,\mathrm{t}),\varphi\big\rangle 
    - \big\langle \rchi_{\mathbf{\Omega}}\hbar^{-\frac{3}{2}}\sin\big(\hbar^{-\frac{1}{2}}\big)\ast\mathrm{P}^\text{sc}(\cdot,t), \varphi\big\rangle
    \\ \nonumber&=\nonumber \int_{\mathbb R^3}\mathrm{c}_0^{-1}\int_{\sigma-i\infty}^{\sigma+i\infty}e^{\mathrm{s}\mathrm{t}}\bm{\mathrm{s}}^2\Tilde{\mathrm{P}}^\text{sc}(\mathrm{x},\bm{\mathrm{s}})\overline{\varphi}(\mathrm{x})d\bm{\mathrm{s}}d\mathrm{x}  
    +  \int_{\mathbb R^3}\int_{\sigma-i\infty}^{\sigma+i\infty}e^{\mathrm{s}\mathrm{t}} \nabla\Tilde{\mathrm{P}}^\text{sc}(\mathrm{x},\bm{\mathrm{s}})\cdot\overline{\nabla\varphi}(\mathrm{x})d\bm{\mathrm{s}}d\mathrm{x} 
    + \int_{\mathbb R^3}\frac{b}{\hbar}\int_{\sigma-i\infty}^{\sigma+i\infty}e^{\mathrm{s}\mathrm{t}}\Tilde{\mathrm{P}}^\text{sc}(\mathrm{x},\bm{\mathrm{s}})\overline{\varphi}(\mathrm{x})d\bm{\mathrm{s}}d\mathrm{x}
    \\ \nonumber &- \int_{\mathbb R^3}\int_{\sigma-i\infty}^{\sigma+i\infty}e^{\mathrm{s}\mathrm{t}} \rchi_{\mathbf{\Omega}}\hbar^{-\frac{3}{2}} \widetilde{\sin}\Big(\hbar^{-\frac{1}{2}}\Big)\Tilde{\mathrm{P}}^\text{sc}(\mathrm{x},\bm{\mathrm{s}})\overline{\varphi}(\mathrm{x})d\bm{\mathrm{s}}d\mathrm{x} 
    \\ \nonumber &= \int_{\sigma-i\infty}^{\sigma+i\infty}\int_{\mathbb R^3}e^{\mathrm{s}\mathrm{t}}\Big(\mathrm{c}_0^{-1}\bm{\mathrm{s}}^2\Tilde{\mathrm{P}}^\text{sc}(\mathrm{x},\bm{\mathrm{s}})\overline{\varphi}(\mathrm{x}) + \nabla\Tilde{\mathrm{P}}^\text{sc}(\mathrm{x},\bm{\mathrm{s}})\cdot\overline{\nabla\varphi}(\mathrm{x}) + \frac{b}{\hbar}\Tilde{\mathrm{P}}^\text{sc}(\mathrm{x},\bm{\mathrm{s}})\overline{\varphi}(\mathrm{x})-\rchi_{\mathbf{\Omega}}\hbar^{-\frac{3}{2}} \widetilde{\sin}\Big(\hbar^{-\frac{1}{2}}\Big)\Tilde{\mathrm{P}}^\text{sc}(\mathrm{x},\bm{\mathrm{s}})\overline{\varphi}(\mathrm{x})\Big)d\mathrm{x}d\bm{\mathrm{s}}
    \\ \nonumber &= \int_{\sigma-i\infty}^{\sigma+i\infty} e^{\mathrm{s}\mathrm{t}}\int_{\mathbb R^3}\Bigg(\rchi_{\mathbf{\Omega}}\hbar^{-\frac{3}{2}} \widetilde{\sin}\Big(\hbar^{-\frac{1}{2}}\Big) - \frac{b}{\hbar}\rchi_\mathbf{\Omega}\Bigg)\Tilde{\mathrm{P}}^\text{sc}(x,\mathbf{s})\overline{\varphi} d\mathrm{x} d\bm{\mathrm{s}} \nonumber= \big\langle \mathrm{F}(\cdot,\mathrm{t}),\varphi\big\rangle,
\end{align}
where 
$\big\langle \cdot,\cdot\big\rangle$ denotes the duality pairing between $\mathrm{H}^{1}(\mathbb R^3)$ and $\mathrm{H}^{-1}(\mathbb R^3)$, 
$ \sin\big(\hbar^{-\frac{1}{2}}\big)\ast\mathrm{P}^\text{sc}(\cdot,t):=\displaystyle \int\limits_0^t\sin\Big(\hbar^{-\frac{1}{2}}(t-\tau)\Big)\mathrm{P}^\text{sc}(\cdot,\tau)d\tau$, and 
$\big\langle \mathrm{F}(\cdot,\mathrm{t}),\varphi\big\rangle := \frac{b}{\hbar}\big\langle \mathrm{c}_0^{-1}\mathrm{P}^\text{sc}(\cdot,\mathrm{t}),\varphi\big\rangle 
- \big\langle \rchi_{\mathbf{\Omega}}\hbar^{-\frac{3}{2}}\sin\big(\hbar^{-\frac{1}{2}}\big)\ast\mathrm{P}^\text{sc}(\cdot,\tau), \varphi\big\rangle$
    
\noindent
The proof is thus complete. \qed

\section{Proof of Theorem \ref{non-periodic}: Periodic Distribution i.e. \texorpdfstring{$K\equiv 0$}{K=0}} \label{th1}      

\noindent
We start by recalling the following non-homogeneous second-order matrix differential equation with initial conditions:
\begin{align}\label{compare}
\begin{cases}\displaystyle
    \hbar_i\frac{\mathrm{d}^2}{\mathrm{d}\mathrm{t}^2}\widetilde{\mathrm{Y}}_i(\mathrm{t}) + \widetilde{\mathrm{Y}}_i(\mathrm{t}) + \sum\limits_{\substack{j=1 \\ j\neq i}}^M\mathrm{q}_{ij} \frac{\mathrm{d}^2}{\mathrm{d}\mathrm{t}^2}\widetilde{\mathrm{Y}}_j(\mathrm{t}-\mathrm{c}_0^{-1}|\mathrm{z}_i-\mathrm{z}_j|) = \frac{\partial^2}{\partial t^2} u^\textbf{in} \mbox{ in } (0, \mathrm{T}),
     \\ \widetilde{\mathrm{Y}}_i(\mathrm{0}) = \frac{\mathrm{d}}{\mathrm{d}\mathrm{t}}\widetilde{\mathrm{Y}}_i(\mathrm{0}) = 0.
\end{cases}
\end{align}
\noindent
The main step of the proof lies in comparing (\ref{compare}) with the following Lippmann-Schwinger equation
\begin{equation}\label{effective-equation}
\hbar\; \rchi_{\mathbf{\Omega}}\; \frac{\partial ^2}{\partial t^2} \bm{\mathrm{Y}} (\mathrm{x},\mathrm{t}) + \bm{\mathrm{Y}} (\mathrm{x},\mathrm{t}) + \int_{\mathbf{\Omega}}\frac{b}{4\pi\vert x-y\vert} \frac{\partial ^2}{\partial t^2}{\mathbf{Y} (x, t-c_0^{-1}\vert x-y\vert)} dy = \frac{\partial ^2}{\partial t^2}u^\textbf{in}(\mathrm{x},\mathrm{t}),\mbox{ for } \mathrm{x} \in \mathbb{R}^3, \mathrm{t} \in (0, \mathrm{T}),
\end{equation}
where we include the initial conditions for $\mathbf{Y}$ up to the first order in this equation.

\noindent
Before we move forward, it is essential to establish some regularity results for the solution of the Lippmann-Schwinger equation (\ref{effective-equation}). Observe that (\ref{effective-equation}) is nothing but the integral equation formulation of (\ref{pde}) (or equivalently (\ref{intego-differential})). Nevertheless, for the sake of referencing and its elegant proof, we provide here the invertibility and regularity properties of (\ref{effective-equation}).


\subsection{Well-posedness and regularity of the Lippmann-Schwinger Equation (\ref{effective-equation}) }                 
\noindent
In the first part of this section, our aim is to demonstrate the existence and uniqueness of the time-domain Lippmann-Schwinger equation (\ref{effective-equation}).


\subsubsection{Existence and Operator Estimate}  

\begin{proposition}
The Lippmann-Schwinger equation
\begin{equation}\label{effective-equation_1}
\hbar\; \rchi_{\mathbf{\Omega}}\; \frac{\partial ^2}{\partial t^2} \bm{\mathrm{Y}} (\mathrm{x},\mathrm{t}) + \bm{\mathrm{Y}} (\mathrm{x},\mathrm{t}) + \int_{\mathbf{\Omega}}\frac{b}{4\pi\vert x-y\vert} \frac{\partial ^2}{\partial t^2}{\mathbf{Y} (x, t-c_0^{-1}\vert x-y\vert)} dy = \frac{\partial ^2}{\partial t^2}u^\textbf{in}(\mathrm{x},\mathrm{t}),\mbox{ for } \mathrm{x} \in \mathbb{R}^3, \mathrm{t} \in (0, \mathrm{T}),
\end{equation}
has a unique solution in $\mathrm{H}^{\mathrm{r}}_{0,\sigma}\big(0,\mathrm{T};\mathrm{L}^2(\mathbf{\Omega})\big)$ for $u^\textbf{in} \in \mathrm{H}^{\mathrm{r}+3}_{0,\sigma}\big(0,\mathrm{T};\mathrm{L}^2(\mathbf{\Omega})\big)$ and it satisfies \\ $\Vert \bm{\mathrm{Y}}\Vert_{\mathrm{H}^{\mathrm{r}}_{0,\sigma}\big(0,\mathrm{T};\mathrm{L}^2(\mathbf{\Omega})\big)} \lesssim \Vert u^{\textbf{in}}\Vert_{\mathrm{H}^{\mathrm{r}+3}_{0,\sigma}\big(0,\mathrm{T};\mathrm{L}^2(\mathbf{\Omega})\big)}.$
\end{proposition}
\begin{proof}
First, we define the retarded volume potential $\bm{\mathrm{V}}_\mathrm{D}$ by
\begin{align}
    \bm{\mathrm{V}}_\mathbf{\Omega}\big[f\big](\mathrm{x},\mathrm{t}) := \int_\mathbf{\Omega}\frac{\rho_\mathrm{c} }{4\pi|\mathrm{x}-\mathrm{y}|}f(\mathrm{y},\mathrm{t}-\mathrm{c}_0^{-1}|\mathrm{x}-\mathrm{y}|)d\mathrm{y}, \quad (\mathrm{x},\mathrm{t}) \in \mathbf{\Omega} \times(0,\mathrm{T}).
\end{align}
We, then, analyze the LS equation (\ref{effective-equation_1}) with the Laplace-Fourier transform. Let us consider the transform parameter $\mathbf{s}=\sigma+i\omega\in \mathbb C,$ with $\sigma\in \mathbb R^+$ and $\mathrm{D}\in \mathbb R.$ 
\noindent
We then take the Laplace-Fourier transform of the time domain LS equation, we obtain the following LS equation in the Laplace-Fourier domain:
\begin{align}
    (\hbar\;\mathbf{s}^2+1)\;\hat{\mathbf{Y}} + \mathbf{s}^2\; \widehat{\bm{\mathrm{V}}}_\mathbf{\Omega}\big(\hat{\mathbf{Y}}\big) = \mathbf{s}^2\; \widehat{u}^\textbf{in},\; \text{in}\; \mathbf{\Omega}
\end{align}
where, we define $\displaystyle\widehat{u}:= \widehat{u}(x,\mathbf{s})=\int_0^\infty u(x,t)e^{-\mathbf{s}t}dt.$
\noindent
We therefore, consider the following problem of finding $\widehat{\mathbf{Y}}\in \mathrm{L}^2(\mathbf{\Omega}) $ for a given $\widehat{f} \in \mathrm{L}^2(\mathbf{\Omega})$ such that
\begin{align}\label{lax}
    (\hbar\;\mathbf{s}^2+1)\;\hat{\mathbf{Y}} + \mathbf{s}^2\; \widehat{\bm{\mathrm{V}}}_\mathbf{\Omega}\big(\hat{\mathbf{Y}}\big) = \widehat{f} \; \text{in}\; \mathrm{L}^2(\mathbf{\Omega}),
\end{align}
where, $\widehat{f}:= \frac{\rho_\mathrm{b}}{\mathrm{k}_\mathrm{b}}\; \mathbf{s}^2\; \widehat{u}^\textbf{in}.$
\newline
We aim to establish the well-posedness of the aforementioned problem by adopting the approach outlined in \cite{le-monk}. Essentially, we formulate the variational scheme for this problem and employ the Lax-Milgram Lemma to demonstrate its well-posedness.
\newline
Let us proceed for that. We multiply the equation (\ref{lax}) by $\widehat{g} \in \mathrm{L}^2(\mathbf{\Omega})$ and we integrate over $\mathbf{\Omega}$ to obtain the following variational form to find $\widehat{\mathbf{Y}}\in \mathrm{L}^2(\mathbf{\Omega}) $:
\begin{align}
    \mathbb A\big(\hat{\mathbf{Y}},\widehat{g}\big)= \mathbb B(\widehat{g})\; \text{in}\; \mathrm{L}^2(\mathbf{\Omega}),
\end{align}
where, 
$$\mathbb A\big(\hat{\mathbf{Y}},\widehat{g}\big):= \int_{\mathbf{\Omega}}\Big((\hbar\;\mathbf{s}^2+1)\hat{\mathbf{Y}} + \mathbf{s}^2\; \widehat{\bm{\mathrm{V}}}_\mathrm{D}\big(\hat{\mathbf{Y}}\big)\Big)\;\overline{\widehat{g}}\;dy\quad \text{and}\; \mathbb B(\widehat{g}):= \int_{\mathbf{\Omega}} \widehat{f}\;\overline{\widehat{g}}\; dy.$$
Now, to prove the coercivity of the above variational form, we choose $\widehat{g}= \mathbf{s}\widehat{\mathbf{Y}}$ to obtain
\begin{align}
    \mathbb A\big(\hat{\mathbf{Y}},\mathbf{s}\widehat{\mathbf{Y}}\big) = \mathbf{s}|\mathbf{s}|^2\int_{\mathbf{\Omega}}|\hat{\mathbf{Y}}|^2\;dy +\overline{\mathbf{s}}\int_{\mathbf{\Omega}}|\hat{\mathbf{Y}}|^2\;dy + \mathbf{s}|\mathbf{s}|^2 \int_{\mathbf{\Omega}}\widehat{\bm{\mathrm{V}}}_\mathrm{D}\big(\hat{\mathbf{Y}}\big)\;\overline{\widehat{\mathbf{Y}}}\;dy.
\end{align}
After Taking Real part of the above equation we derive
\begin{align}
    \Re\Big(\mathbb A\big(\hat{\mathbf{Y}},\mathbf{s}\widehat{\mathbf{Y}}\big)\Big) = \Re\Big(\mathbf{s}|\mathbf{s}|^2\int_{\mathbf{\Omega}}|\hat{\mathbf{Y}}|^2\;dy\Big)+ \sigma \int_{\mathbf{\Omega}}|\hat{\mathbf{Y}}|^2\;dy + \Re\Big(\mathbf{s}\mathbf{s}^2 \int_{\mathbf{\Omega}}\widehat{\bm{\mathrm{V}}}_\mathrm{D}\big(\hat{\mathbf{Y}}\big)\;\overline{\widehat{\mathbf{Y}}}\;dy\Big), \; \text{with}\; \sigma>0.
\end{align}
We now denote $\mathrm{z} := \widehat{\bm{\mathrm{V}}}_\mathbf{\Omega}\big(\hat{\mathbf{Y}}\big) $ and then, we have
\begin{align}
    -\Delta \mathrm{z} + \bm{\mathrm{s}}^2\mathrm{z} = \hat{\mathbf{Y}}\quad \text{in} \quad \mathbb R^3.
\end{align}
Therefore, we deduce that
\begin{align}
\nonumber
   \int_{\mathbb R^3}\widehat{\bm{\mathrm{V}}}_\mathrm{D}\big(\hat{\mathbf{Y}}\big)\;\overline{\widehat{\mathbf{Y}}}\;dy = \int_{\mathbb{R}^3} \mathrm{z}\ \overline{\Big(-\Delta \mathrm{z} + \bm{\mathrm{s}}^2\mathrm{z}\Big)}d\mathrm{y}
    = \int_{\mathbb{R}^3} |\nabla\mathrm{z}|^2 + \overline{\bm{\mathrm{s}}}^2\;|\mathrm{z}|^2d\mathrm{y} .
\end{align}
Then it follows that
\begin{align}
\nonumber
    \Re\Big(\mathbf{s}\mathbf{s}^2 \int_{\mathbb R^3}\widehat{\bm{\mathrm{V}}}_\mathrm{D}\big(\hat{\mathbf{Y}}\big)\;\overline{\widehat{\mathbf{Y}}}\;dy\Big) = \Re\Big( \bm{\mathrm{s}}|\bm{\mathrm{s}}|^2\int_{\mathbb{R}^3} |\nabla\mathrm{z}|^2d\mathrm{y} + \frac{\overline{\bm{\mathrm{s}}}|\bm{\mathrm{s}}|^4}{\mathrm{c}^2_0} \int_{\mathbb{R}^3} |\mathrm{z}|^2d\mathrm{y}\Big)\ge 0.
\end{align}
We also have $$\Re\Big(\mathbf{s}|\mathbf{s}|^2\int_{\mathbf{\Omega}}|\hat{\mathbf{Y}}|^2\;dy\Big)\ge 0.$$

Thus, we have shown that
\begin{align}\label{LM1}
    \mathbb A\big(\hat{\mathbf{Y}},\mathbf{s}\widehat{\mathbf{Y}}\big) \ge \sigma \Vert \mathbf{Y}\Vert_{\mathrm{L}^2(\mathbf{\Omega})}^2.
\end{align}
We also have
\begin{align}\label{LM2}
    |\mathbb B(\mathbf{s}\widehat{\mathbf{Y}})|:= \int_{\mathbf{\Omega}} \widehat{f}\;\overline{\mathbf{s}\; \widehat{\mathbf{Y}}}\; dy| \nonumber &= \Big||\mathbf{\Omega}|\;\int_{\mathbf{\Omega}}  \mathbf{s}\mathbf{s}^2\; \widehat{u}^\textbf{in}\; \hat{\mathbf{Y}}\;d\mathrm{x}\Big|
    \\ &\le |\bm{\mathrm{s}}|^3\Vert \widehat{u}^\textbf{in}\Vert_{\mathrm{L}^2(\mathbf{\Omega})}\Vert \hat{\mathbf{Y}}\Vert_{\mathrm{L}^2(\mathbf{\Omega})}.
\end{align}
Therefore, we deduce that 
\begin{align}
    \Bigg\Vert \Big((\hbar\;\mathbf{s}^2+1)\mathbf{I} + \mathbf{s}^2\; \widehat{\bm{\mathrm{V}}}_\mathbf{\Omega}\Big)^{-1}\Bigg\Vert_{\mathrm{L}^2(\mathbf{\Omega})\to \mathrm{L}^2(\mathbf{\Omega})} \le \frac{|\bm{\mathrm{s}}|^3}{\sigma}.
\end{align}
Let us then define the following operator:
\begin{align}
    \bm{\mathcal{A}}_\mathbf{\Omega} := \big(\hbar\;\partial_\mathrm{t}^2+1\big)\mathbf{I} + \bm{\mathrm{V}}_\mathbf{\Omega} \partial_\mathrm{t}^2.
\end{align}
Following Lubich's notation \cite{lubich}, and using the techniques of the Laplace-Fourier transform as discussed in \cite{Arpan-Sini-SIMA, sini-wang, le-monk} (or see (\ref{well})), we can show that
\begin{equation*}
    \bm{\mathcal{A}}_{\mathbf{\Omega}}^{-1} : \mathrm{H}^{\mathrm{r}+3}_{0,\sigma}\big(0,\mathrm{T};\mathrm{L}^2(\mathbf{\Omega})\big) \to \mathrm{H}^{\mathrm{r}}_{0,\sigma}\big(0,\mathrm{T};\mathrm{L}^2(\mathbf{\Omega})\big)
\end{equation*}
is bounded.
\newline
Thus, we show that the equation (\ref{effective-equation}) has a unique solution in $\mathrm{H}^{\mathrm{r}}_{0,\sigma}\big(0,\mathrm{T};\mathrm{L}^2(\mathbf{\Omega})\big)$ for $u^\textbf{in} \in \mathrm{H}^{\mathrm{r}+3}_{0,\sigma}\big(0,\mathrm{T};\mathrm{L}^2(\mathbf{\Omega})\big).$  
\newline
This completes the proof.
\end{proof}


\subsubsection{Higher Regularity of the Solution} \label{regularity}   
\noindent
In this section, we analyze the higher regularity of the solution to the problem (\ref{effective-equation_1}). As a consequence of the previous lemma, we state the following Corollary.

\begin{corollary}\label{cor}
    We consider the Lippmann-Schwinger equation (\ref{effective-equation}). Then, we have $\frac{\partial ^3}{\partial t^3}\mathbf{Y}\in\mathrm{C}\big(0,\mathrm{T};L^\infty(\mathbf{\Omega})\big)$ and $\partial_{x_i}\frac{\partial^2}{\partial t^2}\mathbf{Y}\in\mathrm{C}\big(0,\mathrm{T};L^\infty(\mathbf{\Omega})\big)$.
\end{corollary}
\begin{proof}
We first rewrite (\ref{effective-equation}) as follows:
\begin{align}
    \hbar\; \frac{\partial ^2}{\partial t^2} \bm{\mathrm{Y}} (\mathrm{x},\mathrm{t}) + \bm{\mathrm{Y}} (\mathrm{x},\mathrm{t}) = \frac{\partial ^2}{\partial t^2}u^\textbf{in}(\mathrm{x},\mathrm{t}) - \int_{\mathbf{\Omega}}\frac{b}{4\pi\vert x-y\vert} \frac{\partial ^2}{\partial t^2}{\mathrm{Y} (x, t-c_0^{-1}\vert x-y\vert)} dy.
\end{align}
Given that the incident wave \( u^\textbf{in} \) belongs to the space \( \mathrm{H}^{10}_{0,\sigma}\big(0,T;\mathrm{L}^2(\mathbf{\Omega})\big) \), as per the preceding lemma, it follows that \( \mathbf{Y} \) also belongs to \( \mathrm{H}^7_{0,\sigma}\big(0,T;\mathrm{L}^2(\mathbf{\Omega})\big) \) with \( r=6 \). Consequently, we have \( \frac{\partial^2}{\partial t^2}\mathbf{Y} \in \mathrm{H}^5_{0,\sigma}\big(0,T;\mathrm{L}^2(\mathbf{\Omega})\big) \).

\noindent
It is known that the retarded volume potential belongs to \( H^4_{0,\sigma}\big(0,\mathrm{T};\mathrm{H}^2(\mathbf{\Omega})\big) \) for a density \( \frac{\partial^2}{\partial t^2}\mathbf{Y}\ \text{in}\ \mathrm{H}^5_{0,\sigma}\big(0,\mathrm{T};\mathrm{L}^2(\mathbf{\Omega})\big) \), see \cite[Theorem 3.2]{le-monk}. Due to the required smoothness assumption of the incident wave field, we have \( \mathbf{F} \in H^4_{0,\sigma}\big(0,\mathrm{T};\mathrm{H}^2(\mathbf{\Omega})\big) \), where we denote
\begin{equation}
    \mathbf{F}(x,t) :=  \frac{\partial ^2}{\partial t^2}u^\textbf{in}(\mathrm{x},\mathrm{t}) - \int_{\mathbf{\Omega}}\frac{b}{4\pi\vert x-y\vert} \frac{\partial ^2}{\partial t^2}{\mathrm{Y} (y, t-c_0^{-1}\vert x-y\vert)} dy.
\end{equation}
Next, we write the above equation as a second-order non-homogeneous ordinary differential equation:
\begin{align}
    \begin{cases}\displaystyle
    \hbar\frac{\mathrm{d}^2}{\mathrm{d}\mathrm{t}^2}\mathbf{Y}(\cdot,\mathrm{t}) + \mathbf{Y}(\cdot,\mathrm{t}) = \mathbf{F}(\cdot,\mathrm{t}) \mbox{ in } (0, \mathrm{T}),
     \\ \mathbf{Y}(\mathrm{0}) = \frac{\mathrm{d}}{\mathrm{d}\mathrm{t}}\mathbf{Y}(\mathrm{0}) = 0,
\end{cases}
\end{align}
Hence, as we have $\mathbf{F}\in \mathrm{H}^4_{0,\sigma}\big(0,\mathrm{T};H^2(\mathbf{\Omega})\big)$, by Sobolev embedding, we deduce that $\mathbf{F}\in\mathrm{H}^4_{0,\sigma}\big(0,\mathrm{T};L^\infty(\mathbf{\Omega})\big),$ see \cite{sobolev regularity}, which consequently implies from the well-posedness of the non-homogeneous differential equation that $\mathbf{Y}\in\mathrm{H}^4_{0,\sigma}\big(0,\mathrm{T};L^\infty(\mathbf{\Omega})\big).$ In particular, by utilizing the Sobolev embedding \( H^4(0,T) \xhookrightarrow{} C^3(0,T) \), we infer that $\frac{\partial ^3}{\partial t^3}\mathbf{Y}\in\mathrm{C}\big(0,\mathrm{T};L^\infty(\mathbf{\Omega})\big).$

\noindent
So, we obtain that $\mathbf{F}\in\mathrm{H}^4_{0,\sigma}\big(0,\mathrm{T};L^\infty(\mathbf{\Omega})\big)$. Therefore, we have the following after taking the partial derivative with respect to \( x \):
\begin{align}
    &\nonumber\big|\partial_{x_i} \mathbf{F}(x,t)\big| 
    \\ &\nonumber\lesssim \int_{\mathbf{\Omega}}\frac{b}{4\pi\vert x-y\vert^2} \Big|\frac{\partial ^2}{\partial t^2}{\mathrm{Y} (y, t-c_0^{-1}\vert x-y\vert)}\Big| dy + \int_{\mathbf{\Omega}}\frac{1}{4\pi\vert x-y\vert} \Big|\frac{\partial ^3}{\partial t^3}{\mathrm{Y} (y, t-c_0^{-1}\vert x-y\vert)}\Big| dy + \big|\partial_x u^\textbf{in}(x,t)\big|
    \\ &\nonumber\lesssim \int_{\mathbf{\Omega}}\frac{dy}{4\pi\vert x-y\vert^2} \Vert\frac{\partial ^2}{\partial t^2}{\mathrm{Y} (y, t-c_0^{-1}\vert x-y\vert)}\Vert_{\mathrm{C}\big(0,\mathrm{T};L^\infty(\mathbf{\Omega})\big)} + \Big(\int_{\mathbf{\Omega}}\frac{dy}{4\pi\vert x-y\vert^2}\Big)^\frac{1}{2} \Vert\frac{\partial ^3}{\partial t^3}{\mathrm{Y} (y, t-c_0^{-1}\vert x-y\vert)}\Vert_{\mathrm{C}\big(0,\mathrm{T};L^\infty(\mathbf{\Omega})\big)} + \mathcal{O}(1)
    \\ &= \mathcal{O}(1),
\end{align}
which implies that $\partial_{x_i}\mathbf{F}\in\mathrm{H}^4_{0,\sigma}\big(0,\mathrm{T};L^\infty(\mathbf{\Omega})\big).$ Therefore, if we repeat the above steps, i.e., using the well-posedness of the non-homogeneous differential equation, we have that $\partial_{x_i}\mathbf{Y} \in C^3\big(0,\mathrm{T};\mathrm{L}^{\infty}(\mathbf{\Omega})\big)$ or, in particular, $\partial_{x_i}\frac{\partial ^2}{\partial t^2}\mathbf{Y} \in C\big(0,\mathrm{T};\mathrm{L}^{\infty}(\mathbf{\Omega})\big).$

\noindent
This completes the proof.
\end{proof}


\subsection{End of the Proof of Theorem \ref{non-periodic}. for \texorpdfstring{$\mathrm{K} \equiv 0$}{K=0}}
 \label{end}   

\noindent
We start this section by estimating $\displaystyle\sum_{\substack{i=1}}^M|\widetilde{\mathbf{Y}}_i(t)-\mathbf{Y}(z_i,t)|^2$, where $(\widetilde{\mathbf{Y}}_i)_{i=1}^M$ is the vector solution to the non-homogeneous second-order matrix differential equation (\ref{matrixmulti}) and $\mathbf{Y}(z_i,t)$, $i=1,\ldots,M$, is the solution of the corresponding Lippmann-Schwinger equation (\ref{effective-equation}) with zero initial conditions, respectively.

\noindent
Let us rewrite the expression (\ref{effective-equation}) in a discretized form at $\mathrm{x}=\mathrm{z}_i$
\begin{align}
    \hbar\; \frac{\partial ^2}{\partial t^2} \bm{\mathbf{Y}} (\mathrm{z}_i,\mathrm{t}) + \bm{\mathbf{Y}} (\mathrm{z}_i,\mathrm{t}) + \sum\limits_{\substack{j=1 \\ j\neq i}}^M \frac{b}{4\pi|z_i-z_j|}\; \varepsilon\; \frac{\partial ^2}{\partial t^2}{\mathbf{Y} (z_i, t-c_0^{-1}|z_i-z_j|)} = \frac{\partial ^2}{\partial t^2}u^\textbf{in}(z_i,\mathrm{t}) + E_{(\mathbf{1})} + E_{(\mathbf{2})} + E_{(\mathbf{3})},
\end{align}
where,
$$E_{(\mathbf{1})}:= -\int_{\mathbf{\Omega}\setminus\bigcup\limits_{j=1}^{[\varepsilon^{-1}]}\Omega_j}\frac{b}{4\pi\vert z_i-y\vert} \frac{\partial ^2}{\partial t^2}{\mathbf{Y} (z_i, t-c_0^{-1}\vert z_i-y\vert)} dy,$$
$$E_{(\mathbf{2})}:=-\int_{\Omega_i}\frac{b}{4\pi\vert z_i-y\vert} \frac{\partial ^2}{\partial t^2}{\mathbf{Y} (z_i, t-c_0^{-1}\vert z_i-y\vert)} dy,$$
and
$$E_{(\mathbf{3})}:=-\sum\limits_{\substack{j=1 \\ j\neq i}}^{[\varepsilon^{-1}]} \int_{\Omega_i}\frac{b}{4\pi|z_i-y|} \frac{\partial ^2}{\partial t^2}{\mathbf{Y} (y, t-c_0^{-1}|z_i-y|)}\;dy + \sum\limits_{\substack{j=1 \\ j\neq i}}^{[\varepsilon^{-1}]} \frac{b}{4\pi|z_i-z_j|}\;\varepsilon\; \frac{\partial ^2}{\partial t^2}{\mathbf{Y} (z_i, t-c_0^{-1}|z_i-z_j|)}$$
\noindent
Before we estimate the terms mentioned earlier, let us first recall the following lemma.
\begin{lemma}\cite{habib-sini} \label{counting}
    \textbf{A Counting Lemma.} For any arbitrary distribution of the points $z_j, j=1, ..., M$, in a bounded domain of $\mathbb{R}^3$, (with a given minimum distance $d$ between them), we have the following estimates uniformly with respect to $j$: 
    \begin{align}
        \sum\limits_{\substack{j=1 \\ j\neq i}}^{M}\frac{1}{|z_i-z_j|^k} = 
        \begin{cases}
            \mathcal{O}(d^{-3})\; &\text{if}\; k<3 \\
            \mathcal{O}\Big(d^{-3}(1+|\log(d)|)\Big)\; & \text{if}\; k=3 \\
            \mathcal{O}(d^{-k})\; &\text{if}\; k>3. \\
        \end{cases}
    \end{align}
\end{lemma}
\noindent
To estimate the term $E_{(\mathbf{1})}$, we borrow the idea and notations from \cite{habib-sini, sini-wang}. We address the following two scenarios:


\begin{figure}
\begin{center}
\begin{tikzpicture}[scale=1.2]

\draw (0,0) ellipse (3cm and 1.5cm);
\node[scale=1] at (0,-1.8) {$\mathbf{\Omega}$};

\draw[<-,red, thick] (1.6, 1.15) -- (2.5, 1.5);
\node[above, scale=0.5] at (2.5,1.5) {$\mathbf{A}_{(1)}\ (\text{The region above the blue line})$};

\draw[<-,red, thick] (-1.6, -1.15) -- (-2.5, -1.5);
\node[above, scale=0.5] at (-2.5,-1.75) {$\mathbf{A}_{(2)} (\text{The region below the blue line})$};

\draw[<-,red, thick] (0.1,1.2) -- (1, 1.8);
\node[above, scale=0.5] at (1,1.8) {$\mathrm{D}_i$};

\node[above, scale=.5] at (0,0.8) {$z_i$};
\node[above, scale=0.5] at (0,0.9) {$\bullet$};

\coordinate (B) at (0,1); 
\draw (B) circle (0.2);

\draw[dotted, blue, line width=1.2pt] (-3.5,0.5) -- (3.5,0.5);

\clip (0,0) ellipse (3cm and 1.5cm);

\foreach \x in {-3.5,-2.5,...,3.5} {
    \draw (\x,-2) -- (\x,2);
}
\def\xellip(#1){4*sqrt(1 - (#1/2)^2)}

\foreach \y in {-1.75,-1.25,...,1.75} {
    \draw ({\xellip(\y)},\y) -- ({-\xellip(\y)},\y);
}

\begin{scope}
    \clip (2.5,-1.5) rectangle (3.5,1.5);
    \foreach \i in {-1.5,-1.3,...,3.5} {
        \draw[black, thick] (\i,-0.8) -- ++(3.5,1.5);
    }
\end{scope}
\begin{scope}
    \clip (-2.9,-1.5) rectangle (3.5,-1.25);
    \foreach \i in {-2.9,-2.65,...,0} {
        \draw[black, thick] (\i,-1.7) -- ++(3.5,1);
    }
\end{scope}

\begin{scope}
    \clip (-3.9,-1.5) rectangle (-2.5,2.7);
    \foreach \i in {-3.9,-3.8,...,3.9} {
        \draw[black, thick] (\i,-.9) -- ++(3.5,4.9);
    }
\end{scope}

\begin{scope}
    \clip (-3,1) rectangle (5.5,2);
    \foreach \i in {-3,-2.8,...,2} {
        \draw[black, thick] (\i,1.25) -- ++(5.5,1.5);
    }
\end{scope}

\begin{scope}
    \clip (1.5,0.75) rectangle (5.5,1.5);
    \foreach \i in {-3,-2.8,...,2} {
        \draw[black, thick] (\i,0) -- ++(5.5,1.5);
    }
\end{scope}

\begin{scope}
    \clip (-6.5,0.75) rectangle (-1.5,1.5);
    \foreach \i in {-6.5,-6.3,...,2} {
        \draw[black, thick] (\i,0) -- ++(5.5,1.5);
    }
\end{scope}

\begin{scope}
    \clip (-6.5,-2) rectangle (-1.5,-.75);
    \foreach \i in {-6.5,-6.3,...,2} {
        \draw[black, thick] (\i,-.5) -- ++(5.5,-2);
    }
\end{scope}

\begin{scope}
    \clip (1.5,-.75) rectangle (5.5,-5);
    \foreach \i in {-3,-2.8,...,5} {
        \draw[black, thick] (\i,0) -- ++(5.5,-3);
    }
\end{scope}

\draw (0,0) ellipse (3cm and 1.5cm);

\end{tikzpicture}
\end{center}
\caption{A schematic illustration for the split of the region $\mathbf{\Omega}\setminus\bigcup\limits_{j=1}^{[\varepsilon^{-1}]}\Omega_j.$ } \label{pic1}
\end{figure}
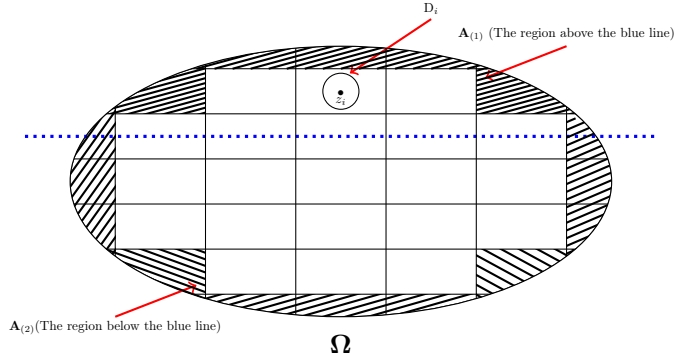    


\begin{enumerate}
    \item When the point $z_i$ is away from the boundary $\partial\mathbf{\Omega}$, the function $|z_i - z|^{-1}$ remains bounded in the vicinity of the boundary. Consequently, in this scenario, we find $E_{(\mathbf{1})} = \mathcal{O}\left(\text{vol}\Big(\mathbf{\Omega}\setminus\bigcup\limits_{j=1}^{[\varepsilon^{-1}]}\Omega_j\Big)\right) = \mathcal{O}(\varepsilon^\frac{1}{3}).$

\item When the point $z_i$ is close to one of the $\Omega_j$'s, touching the boundary $\partial\mathrm{D}$, we divide the estimation into two distinct segments. We designate the portion involving $\Omega_j$'s close to $z_i$ as $\mathbf{A}_{(1)}$, while the remaining segment is referred to as $\mathbf{A}_{(2)}$. The integral over $\mathbf{A}_{(2)}$ can be assessed in a manner similar to that in the previous case. Notably, $\mathbf{A}_{(2)}\subset \mathbf{\Omega}\setminus\bigcup\limits_{j=1}^{[\varepsilon^{-1}]}\Omega_j$, hence the \text{vol} ($\mathbf{A}_{(2)}$) scales as $\varepsilon^\frac{1}{3}$ as $\varepsilon\to 0$.

To estimate the integral over $\mathbf{A}_{(1)}$, we initially evaluate the number of $\Omega_j$'s near $z_i$. It is observed that the $\Omega_j$'s in proximity to $z_i$ are positioned near a small region of the boundary $\partial\mathbf{\Omega}$. Assuming that the boundary is sufficiently smooth, we can approximate this region as flat and centered at $z_i$. Consequently, we partition this flat region into concentric square layers (centered at $z_i$). This construction is illustrated in Figure (\ref{pic1}). In particular, considering the flat region to be of order 1 in terms of the parameter $\varepsilon$, and given that the maximum radius of the squares (or the $\Omega_j$'s) is $\varepsilon^\frac{1}{3}$, the count of layers is at most of order $[\varepsilon^{-\frac{1}{3}}]$. In this context, within the $n$th layer, for $n = 0, \ldots , [\varepsilon^{-\frac{1}{3}}]$, there are at most $(2n + 1)^2$ squares (and hence cubes intersecting the surface). The count of inclusions in the $n^\text{th}$ layer (excluding $n=0$) will be at most $[(2n + 1)^2 - (2n-1)^2]$, and their distance from $\hbar_i$ is at least $n(\varepsilon^\frac{1}{3}-\frac{\varepsilon}{2})$.
 \end{enumerate}
 
\noindent
Therefore, we write the following term as follows:
\begin{align}
    |E_{(\mathbf{1})}| 
    \nonumber&= \Bigg|\int_{\mathbf{\Omega}\setminus\bigcup\limits_{j=1}^{[\varepsilon^{-1}]}\Omega_j}\frac{b}{4\pi\vert z_i-y\vert} \frac{\partial ^2}{\partial t^2}{\mathbf{Y} (z_i, t-c_0^{-1}\vert z_i-y\vert)} dy\Bigg|
    \\ \nonumber &\le \Bigg|\int_{\mathbf{A}_{(1)}}\frac{b}{4\pi\vert z_i-y\vert} \frac{\partial ^2}{\partial t^2}{\mathbf{Y} (z_i, t-c_0^{-1}\vert z_i-y\vert)} dy\Bigg| + \Bigg|\int_{\mathbf{A}_{(2)}}\frac{b}{4\pi\vert z_i-y\vert} \frac{\partial ^2}{\partial t^2}{\mathbf{Y} (z_i, t-c_0^{-1}\vert z_i-y\vert)} dy\Bigg|
    \\ \nonumber &\le \sum\limits_{\substack{j=1}}^{[\varepsilon^{-\frac{1}{3}}]}\frac{1}{d_{ij}}b\;\Vert \frac{\partial ^2}{\partial t^2}{\mathbf{Y}\Vert_{\mathrm{C}^1\big(0,\mathrm{T};\mathrm{L}^\infty(\mathbf{\Omega})\big)}}\text{vol}\big(\Omega_j\big)+ b\;\Vert \frac{\partial ^2}{\partial t^2}{\mathbf{Y}\Vert_{\mathrm{C}^1\big(0,\mathrm{T};\mathrm{L}^\infty(\mathbf{\Omega})\big)}}\text{vol}\big(\mathbf{A}_{(2)}\big)
    \\ \nonumber &\le \mathcal{O}\Big(b\ \sum\limits_{\substack{j=1}}^{[\varepsilon^{-\frac{1}{3}}]}\frac{1}{d_{ij}} + b\ \varepsilon^\frac{1}{3}\Big)
    \\ \nonumber &\le \mathcal{O}\Big(b\varepsilon\ \big[(2n+1)^2-(2n-1)^2\big]\frac{1}{n(\varepsilon^\frac{1}{3}-\frac{\varepsilon}{2})} + b\varepsilon^\frac{1}{3}\Big)
    \\ \nonumber &\le \mathcal{O}\Big(b\varepsilon \mathcal{O}(\varepsilon^{-\frac{2}{3}}) + b\varepsilon^{\frac{1}{3}}\Big)  
\end{align}
Hence, we obtain 
\begin{align}\label{esti1}
    |E_{(\mathbf{1})}|  =  \mathcal{O}\big(b\varepsilon^\frac{1}{3}\big).
\end{align}
Next, since we have $\frac{\partial ^2}{\partial t^2}{\mathbf{Y}}\in \mathrm{C}^1\big(0,\mathrm{T};\mathrm{L}^\infty(\mathbf{\Omega})\big),$ we deduce that
\begin{align}
    |E_{(\mathbf{2})}| \nonumber&= \mathcal{O}\Big(b\;\Vert \frac{\partial ^2}{\partial t^2}{\mathbf{Y}}\Vert_{\mathrm{C}^1\big(0,\mathrm{T};\mathrm{L}^\infty(\mathbf{\Omega})\big)}\int_{\Omega_i}|y-z_i|^{-1}\;dy\Big).
\end{align}
To analyze the term $\displaystyle\int_{\Omega_i}|y-z_i|^{-1}\;dy$, we divide it into two parts:
\begin{align}
    \int_{\Omega_i}|y-z_i|^{-1}\;dy = \int_{B(z_i,r)}|y-z_i|^{-1}\;dy + \int_{\Omega_i\setminus B(z_i,r)}|y-z_i|^{-1}\;dy.
\end{align}
Now, expressing these two terms in polar coordinates, we get:
\begin{align}
    \int_{\Omega_i}|y-z_i|^{-1}\;dy \nonumber&= \int_{B(z_i,r)}|y-z_i|^{-1}\;dy + \int_{\Omega_i\setminus B(z_i,r)}|y-z_i|^{-1}\;dy
    \\ \nonumber &= 2\pi r^2 + \frac{1}{r}(\varepsilon- \frac{4}{3}\pi r^3),\; \text{since}\; \text{vol}(\Omega_i\setminus B(z_i,r)) = \varepsilon- \frac{4}{3}\pi r^3.
\end{align}
Now, as this expression has a critical point at $r_{\text{sol}} = (\frac{3}{4\pi})^\frac{1}{3}\;\varepsilon^\frac{1}{3}$, we conclude that $\displaystyle\int_{\Omega_i}|y-z_i|^{-1}\;dy = \mathcal{O}(\varepsilon^\frac{2}{3}).$ Therefore, we deduce that 
\begin{align}\label{esti2}
    |E_{(\mathbf{2})}| = \mathcal{O}\Big(b\varepsilon^\frac{2}{3}\Big).
\end{align}
Let us now proceed to estimate the third term $E_{(\mathbf{3})}.$ We have
\begin{align}
    E_{(\mathbf{3})} \nonumber&:= -\sum\limits_{\substack{j=1 \\ j\neq i}}^{[\varepsilon^{-1}]}b\int_{\Omega_j} \Bigg[\frac{\partial ^2}{\partial t^2}{\mathbf{Y} (y, t-c_0^{-1}|z_i-y|)} -  \frac{1}{4\pi|z_i-z_j|} \frac{\partial ^2}{\partial t^2}{\mathbf{Y} (z_i, t-c_0^{-1}|z_i-z_j|)}\Bigg]
    \\ \nonumber &= -\sum\limits_{\substack{j=1 \\ j\neq i}}^{[\varepsilon^{-1}]}b\int_{\Omega_j}\Bigg[ \underbrace{\frac{\partial ^2}{\partial t^2}{\mathbf{Y} (y, t-c_0^{-1}|z_i-y|)} \Big[\frac{1}{4\pi|z_i-y|}-\frac{1}{4\pi|z_i-z_j|}\Big]}_{\textbf{err}_{(1)}}
    \\  &+ \underbrace{\frac{1}{4\pi|z_i-z_j|} \Big[\frac{\partial ^2}{\partial t^2}{\mathbf{Y} (y, t-c_0^{-1}|z_i-y|)}-\frac{\partial ^2}{\partial t^2}{\mathbf{Y} (z_i, t-c_0^{-1}|z_i-z_j|)}\Big]}_{\textbf{err}_{(2)}}\Bigg]\;dy.
\end{align}
Then, we estimate $\text{err}_{(1)}$ as
\begin{align}
    \textbf{err}_{(1)} := \Bigg|&\nonumber\sum\limits_{\substack{j=1 \\ j\neq i}}^{[\varepsilon^{-1}]}b\int_{\Omega_j}\frac{\partial ^2}{\partial t^2}{\mathbf{Y} (y, t-c_0^{-1}|z_i-y|)} \Big[\frac{1}{4\pi|z_i-y|}-\frac{1}{4\pi|z_i-z_j|}\Big]dy\Bigg|
    \\ \nonumber &= \Bigg|\frac{1}{4\pi}\sum\limits_{\substack{j=1 \\ j\neq i}}^{[\varepsilon^{-1}]}b\int_{\Omega_j}\frac{\partial ^2}{\partial t^2}{\mathbf{Y} (y, t-c_0^{-1}|z_i-y|)}(y-z_j)\nabla_y\frac{1}{|z_i-z_j^*|}dy\Bigg|,\quad \text{with}\; z_j^*\in \Omega_j
    \\ &=\nonumber \mathcal{O}\Bigg(b\sum\limits_{\substack{j=1 \\ j\neq i}}^{[\varepsilon^{-1}]}\frac{1}{d^2_{ij}}\; \Vert \frac{\partial ^2}{\partial t^2}{\mathbf{Y}\Vert_{\mathrm{C}^1\big(0,\mathrm{T};\mathrm{L}^\infty(\mathbf{\Omega})\big)}}\; \int_{\Omega_j} |y-z_j|dy\Bigg).
\end{align}
Next, we deduce using Taylor's series expansion and $\frac{\partial ^3}{\partial t^3}{\mathbf{Y}}\in \mathrm{C}\big(0,\mathrm{T};\mathrm{L}^\infty(\mathbf{\Omega})\big),$ that
\begin{align}
    &\nonumber \frac{\partial ^2}{\partial t^2}{\mathbf{Y} (y, t-c_0^{-1}|z_i-y|)}-\frac{\partial ^2}{\partial t^2}{\mathbf{Y} (z_i, t-c_0^{-1}|z_i-z_j|)}
    \\ \nonumber &= \frac{\partial ^2}{\partial t^2}{\mathbf{Y} (y, t-c_0^{-1}|z_i-y|)}-\frac{\partial ^2}{\partial t^2}{\mathbf{Y} (z_i, t-c_0^{-1}|z_i-y|)} +\frac{\partial ^2}{\partial t^2}{\mathbf{Y} (z_i, t-c_0^{-1}|z_i-y|)} - \frac{\partial ^2}{\partial t^2}{\mathbf{Y} (z_i, t-c_0^{-1}|z_i-z_j|)}
    \\ \nonumber &= (y-z_i)\;\frac{\partial}{\partial_x}\frac{\partial ^2}{\partial t^2}{\mathbf{Y} (z^*, t-c_0^{-1}|z_i-y|)} + (y-z_j)\; \nabla_y\frac{1}{|z_i-z^*|}\;\frac{\partial ^3}{\partial t^3}{\mathbf{Y} (y, t^*)},
\end{align}
where, $z^*\in \Omega_j$ and $t^*\in (t-c_0^{-1}|z_i-y|,t-c_0^{-1}|z_i-z_j|).$ 
\newline
Therefore, we deduce the following using the fact that $\partial_{x_i}\frac{\partial ^2}{\partial t^2}\mathbf{Y} \in L^\infty\big(0,\mathrm{T};\mathrm{L}^{\infty}(\mathbf{\Omega})\big):$
\begin{align}
    \textbf{err}_{(2)} &\nonumber:= \Bigg|\sum\limits_{\substack{j=1 \\ j\neq i}}^{[\varepsilon^{-1}]}b\int_{\Omega_j}(y-z_i)\;\frac{\partial}{\partial_x}\frac{\partial ^2}{\partial t^2}{\mathbf{Y} (z^*, t-c_0^{-1}|z_i-y|)}dy + \sum\limits_{\substack{j=1 \\ j\neq i}}^{[\varepsilon^{-1}]}b\int_{\Omega_j}(y-z_j)\; \nabla_y\frac{1}{|z_i-z^*|}\;\frac{\partial ^3}{\partial t^3}{\mathbf{Y} (y, t^*)}\Bigg|
    \\ \nonumber &= \mathcal{O}\Bigg(b\sum\limits_{\substack{j=1 \\ j\neq i}}^{[\varepsilon^{-1}]}\frac{1}{d^2_{ij}}\; \Vert \frac{\partial ^3}{\partial t^3}{\mathbf{Y}\Vert_{\mathrm{C}\big(0,\mathrm{T};\mathrm{L}^\infty(\mathbf{\Omega})\big)}}\; \int_{\Omega_j} |y-z_j|dy\Bigg).
\end{align}
Hence, we have 
\begin{align}\label{esti3}
    E_{(\mathbf{3})} = \mathcal{O}\Big(b\sum\limits_{\substack{j=1 \\ j\neq i}}^{[\varepsilon^{-1}]}\frac{1}{d^2_{ij}}\Big)\varepsilon^\frac{4}{3} = \mathcal{O}\Big(b\varepsilon^\frac{1}{3}\Big).
\end{align}
Gathering (\ref{esti1}), (\ref{esti2}) and (\ref{esti3}), we get
\begin{align}
    \sum_{i=1}^M \Big(|E_{(\mathbf{1})}|^2 + |E_{(\mathbf{2})}|^2 + |E_{(\mathbf{3})}|^2\Big) = \mathcal{O}\Big(M\;\varepsilon^\frac{2}{3} + M\;\varepsilon^\frac{4}{3}\Big) = \mathcal{O}(\varepsilon^{-\frac{1}{3}}).
\end{align}
Consequently, we arrive at the following system with $\mathbf{Z}(z_i,,t) := \mathrm{Y}_i(t)-\mathbf{Y}(z_i,t)$ :
\begin{align} \label{matrix2}
    \begin{cases}
             \mathcal{A}\frac{\mathrm{d}^2}{\mathrm{d}\mathrm{t}^2}\bm{\mathrm{Z}}(z_i,\mathrm{t}) + \bm{\mathrm{Z}}(z_i,\mathrm{t}) =  \mathcal{O}(\varepsilon^\frac{1}{3}) \mbox{ in } (0, \mathrm{T}),
             \\ \bm{\mathrm{Z}}(z_i,\mathrm{0}) = \frac{\mathrm{d}}{\mathrm{d}\mathrm{t}}\bm{\mathrm{Y}}(z_i,\mathrm{0}) = 0,   
    \end{cases}
\end{align}
where, we define the operator $\mathcal{A}: (\mathrm{L}_r^2)^M\to (\mathrm{L}_r^2)^M$ as
\begin{align}
    \mathcal{A} = \mathcal{A}(t):=
     \begin{pmatrix}
        \hbar & \dots  & \overline{\mathrm{q}}_{1\mathrm{M}}\mathcal{T}_{-\mathrm{c}_0^{-1}|\mathrm{z}_1-\mathrm{z}_M|}\\
       \vdots & \ddots & \vdots\\
       \overline{\mathrm{q}}_{\mathrm{M}1}\mathcal{T}_{-\mathrm{c}_0^{-1}|\mathrm{z}_M-\mathrm{z}_1|} & \dots  & \hbar 
    \end{pmatrix}, 
\end{align}
with $\mathrm{L}_\ell^2:= \{ f \in \mathrm{L}^2(-r,\mathrm{T}): f=0\; \text{in}\ (-r,0)\}$, the translation operators $\mathcal{T}_{-\mathrm{c}_0^{-1}|\mathrm{z}_i-\mathrm{z}_j|}$, i.e. $\mathcal{T}_{-\mathrm{c}_0^{-1}|\mathrm{z}_i-\mathrm{z}_j|}(f)(t):=f(t-\mathrm{c}_0^{-1}|\mathrm{z}_i-\mathrm{z}_j|)$, and $\ell:=\max_{i\neq j}{\mathrm{c}^{-1}_0\vert z_i- z_j\vert}$.
\newline
Hence, we use the well-possedness of the problem (\ref{matrix2}) as discussed in \cite[Section 2.4]{Arpan-Sini-JEvEq} to obtain
\begin{align}\label{mainesti}
    \sum_{i=1}^M|\widetilde{\mathrm{Y}}_i(t)-\mathbf{Y}(z_i,t)|^2 = \mathcal{O}(\varepsilon^{-\frac{1}{3}}),\; \text{as}\; \varepsilon\to 0.
\end{align}
We introduce the unknown variable $\mathbf{Y} = \frac{\partial^2}{\partial t^2}\mathbf{U}$, where $\mathbf{U}$ satisfies the following Lippmann-Schwinger equation
\begin{equation}
\hbar\; \rchi_{\mathbf{\Omega}}\; \frac{\partial ^2}{\partial t^2} \bm{\mathrm{U}} (\mathrm{x},\mathrm{t}) + \bm{\mathrm{U}} (\mathrm{x},\mathrm{t}) + \int_{\mathbf{\Omega}}\frac{b}{4\pi\vert x-y\vert} \frac{\partial ^2}{\partial t^2}{\mathbf{U} (x, t-c_0^{-1}\vert x-y\vert)} dy = \frac{\rho_\mathrm{b}}{\mathrm{k}_\mathrm{b}}\;u^\textbf{in}(\mathrm{x},\mathrm{t}),\mbox{ for } \mathrm{x} \in \mathbb{R}^3, \mathrm{t} \in (0, \mathrm{T}),
\end{equation}
with zero initial conditions for $\mathbf{U}$ up to the first order and define
Let us now define
\begin{align}
    \mathbcal{V}(x,t) := \begin{cases}
                       \mathbf{U}(x,t) + \hbar\; \frac{\partial ^2}{\partial t^2} \bm{\mathrm{U}} (\mathrm{x},\mathrm{t}) & \text{if}\; (x,t) \in \mathbf{\Omega}\times(0,T) \\ \displaystyle
                       u^\textbf{in}(\mathrm{x},\mathrm{t}) - \int_{\mathbf{\Omega}}\frac{b}{4\pi\vert x-y\vert} \frac{\partial ^2}{\partial t^2}{\mathbf{U} (x, t-c_0^{-1}\vert x-y\vert)} dy & \text{for}\; (x,t) \in \mathbb R^3\setminus\mathbf{\Omega}\times(0,T).
                 \end{cases}
\end{align}
We set $\mathbcal{W}(x,t) := u^\textbf{in}(\mathrm{x},\mathrm{t}) - \mathbcal{V}(x,t).$
\newline
From now on our aim is to estimate $|\mathbcal{W}(x,t)-\mathbf{U}(x,t)|.$ To do this let us assume that $x$ is away from $\mathbf{\Omega}\cup \{x_0\}.$ Therefore, we have
\begin{align}
    \mathbcal{W}(x,t) \nonumber &= \int_{\mathbf{\Omega}}\frac{b}{4\pi\vert x-y\vert} \frac{\partial ^2}{\partial t^2}{\mathbf{U} (y, t-c_0^{-1}\vert x-y\vert)} dy
    \\ \nonumber &= \int_{\mathbf{\Omega}}\frac{b}{4\pi\vert x-y\vert} \mathbf{Y} (y, t-c_0^{-1}\vert x-y\vert) dy
    \\ \nonumber &=\sum\limits_{\substack{i=1}}^{[\varepsilon^{-1}]} \frac{b}{4\pi|x-z_i|}\; \varepsilon\; {\mathbf{Y} (z_i, t-c_0^{-1}|x-z_i|)} - \int_{\mathbf{\Omega}\setminus{\bigcup\limits_{i=1}^{{[\varepsilon^{-1}]}}\Omega_i}}\frac{b}{4\pi\vert x-y\vert} {\mathbf{Y} (y, t-c_0^{-1}\vert x-y\vert)}\; dy
    \\ \nonumber&- \sum\limits_{\substack{i=1 }}^{[\varepsilon^{-1}]}b\int_{\Omega_i} \Big(\frac{1}{4\pi|x-y|}{\mathbf{Y} (y, t-c_0^{-1}|x-y|)} -  \frac{1}{4\pi|x-z_i|} {\mathbf{Y} (z_i, t-c_0^{-1}|x-z_i|)}\Big)\;dy.
\end{align}
Using the similar techniques as discussed to estimate $E_{(2)}$ and $E_{(3)}$, we can show that the second and third term of the above expression can be estimated as $\mathcal{O}(\varepsilon^\frac{1}{3})$\; as $\varepsilon\to 0.$ Therefore, we deduce that
\begin{align}
    \mathbcal{W}(x,t) \nonumber &= \sum\limits_{\substack{i=1 }}^{[\varepsilon^{-1}]} \frac{b}{4\pi|x-z_i|}\; \varepsilon\; {\mathbf{Y} (z_i, t-c_0^{-1}|x-z_i|)} + \mathcal{O}(\varepsilon^\frac{1}{3})
    \\ \nonumber&= \sum\limits_{\substack{i=1}}^{[\varepsilon^{-1}]} \frac{b}{4\pi|x-z_i|}\; \varepsilon\; {\widetilde{\mathrm{Y}}_i (t-c_0^{-1}|x-z_i|)}+ \sum\limits_{\substack{i=1}}^{[\varepsilon^{-1}]} \frac{b}{4\pi|x-z_i|}\; \varepsilon\; \Big({\widetilde{\mathrm{Y}}_i (t-c_0^{-1}|x-z_i|)}-\mathbf{Y} (z_i,t-c_0^{-1}|x-z_i|)\Big) 
    \\ &+ \mathcal{O}(\varepsilon^\frac{1}{3}).
\end{align}
We then use Cauchy-Schwartz's inequality and the estimate (\ref{mainesti}) to arrive at the following estimate
\begin{align}
    \textbf{err}_{(3)} \nonumber&:= \sum\limits_{\substack{i=1}}^{[\varepsilon^{-1}]} \frac{b}{4\pi|x-z_i|}\; \varepsilon\; \Big({\widetilde{\mathrm{Y}}_i (t-c_0^{-1}|x-z_i|)}-\mathbf{Y}(z_i,t-c_0^{-1}|x-z_i|)\Big) 
    \\ &\nonumber= \mathcal{O}\Bigg(\varepsilon\;b \Big(\sum\limits_{\substack{i=1}}^{[\varepsilon^{-1}]} \frac{1}{|x-z_i|^2}\Big)^\frac{1}{2}\;\Big(\sum\limits_{\substack{i=1}}^{[\varepsilon^{-1}]}|\widetilde{\mathrm{Y}}_i (t-c_0^{-1}|x-z_i|)-\mathbf{Y} (z_i,t-c_0^{-1}|x-z_i|)|^2\Big)^\frac{1}{2}\Bigg)
    \\ &\nonumber= \mathcal{O}\Big(\varepsilon\;\varepsilon^{-\frac{1}{2}}\;\varepsilon^{-\frac{1}{6}}\Big) = \mathcal{O}(\varepsilon^\frac{1}{3}).
\end{align}
Consequently, we conclude that 
\begin{align}
        u(x,t) - \mathbcal{W}(x,t) = \mathcal{O}(\varepsilon^\frac{1}{3})\; \text{as}\; \varepsilon\to 0,
    \end{align}
which completes the proof.\qed


\section{Proof of Theorem \ref{non-periodic}: for Non-Periodic Distribution, i.e. \texorpdfstring{$\mathrm{K}\not\equiv 0$}{K not equal to 0}} \label{th2} 

We divide this section into two parts. First, we will consider the case when \(\mathrm{K} \in \mathbb{N}\). Then, following a similar process, we will discuss and extend the results of the following subsection to the case when $K$ of class $C^1$ with $\nabla K(\cdot) \ne 0$ everywhere in the set of discontinuity of $\Big[K(\cdot)+1\Big].$


\subsection{Case when \texorpdfstring{$K|_{\mathbf{\Omega}} \in \mathbb{N}$}{K in N} and the bubbles are distributed according to the Assumption \ref{as2}}
\label{K-natural-number}   

\noindent
The key idea here is to rewrite the algebraic system (\ref{matrixmulti}) into a general algebraic system based on the distribution of the bubbles, as described in Assumption \ref{as2}, and define its corresponding integral equation.

\noindent
Therefore, we start by recalling that the $\big(\widetilde{\mathrm{Y}}_i\big)_{i=1}^\mathrm{M}$ is the vector solution to the following non-homogeneous second-order matrix differential equation with initial zero conditions:
\begin{align}\label{matrixmulti-new}
\begin{cases}\displaystyle
    \hbar_i\frac{\mathrm{d}^2}{\mathrm{d}\mathrm{t}^2}\widetilde{\mathrm{Y}}_i(\mathrm{t}) + \widetilde{\mathrm{Y}}_i(\mathrm{t}) + \sum\limits_{\substack{j=1 \\ j\neq i}}^M\mathrm{q}_{ij} \frac{\mathrm{d}^2}{\mathrm{d}\mathrm{t}^2}\widetilde{\mathrm{Y}}_j(\mathrm{t}-\mathrm{c}_0^{-1}|\mathrm{z}_i-\mathrm{z}_j|) = \frac{\partial^2}{\partial t^2} u^\textbf{in}(z_i) \mbox{ in } (0, \mathrm{T}),
     \\ \widetilde{\mathrm{Y}}_i(\mathrm{0}) = \frac{\mathrm{d}}{\mathrm{d}\mathrm{t}}\widetilde{\mathrm{Y}}_i(\mathrm{0}) = 0.
\end{cases}
\end{align}
Assuming that the bubbles are globally periodic but locally non-periodic, as explained in the Assumption \ref{as2}, we rewrite the system mentioned above for the subdomains $D_{m_l}$, where $m = 1,2,\ldots,[\varepsilon^{-1}]$ and $l=1,2,\ldots,K+1,\ \mathrm{K}\in \mathbb{N}$, in the following way:
\begin{align}
    \begin{cases}\displaystyle
    \hbar_{m_l}\frac{\mathrm{d}^2}{\mathrm{d}\mathrm{t}^2}\widetilde{\mathrm{Y}}_{m_l}(\mathrm{t}) + \widetilde{\mathrm{Y}}_{m_l}(\mathrm{t}) + \sum\limits_{\substack{j=1 }}^{[\varepsilon^{-1}]}\sum\limits_{\substack{i=1\\ j_i\neq m_l}}^{K+1}\mathrm{q}_{z_{m_l},z_{j_i}} \frac{\mathrm{d}^2}{\mathrm{d}\mathrm{t}^2}\widetilde{\mathrm{Y}}_{j_i}(\mathrm{t}-\mathrm{c}_0^{-1}|\mathrm{z}_{m_l}-\mathrm{z}_{j_i}|) 
    = \frac{\partial^2}{\partial t^2} u^\textbf{in}(z_{m_l}) \mbox{ in } (0, \mathrm{T}),
     \\ \widetilde{\mathrm{Y}}_{m_l}(\mathrm{0}) = \frac{\mathrm{d}}{\mathrm{d}\mathrm{t}}\widetilde{\mathrm{Y}}_{m_l}(\mathrm{0}) = 0,
\end{cases}
\end{align}
which is equivalent to
\begin{align}\label{matrixdef}
    \begin{cases}\displaystyle
    \hbar_{m_l}\frac{\mathrm{d}^2}{\mathrm{d}\mathrm{t}^2}\widetilde{\mathrm{Y}}_{m_l}(\mathrm{t}) + \widetilde{\mathrm{Y}}_{m_l}(\mathrm{t}) &+ \textcolor{black}{\sum\limits_{\substack{i=1 \\ i\neq l}}^{K+1}\mathrm{q}_{z_{m_l},z_{m_i}} \frac{\mathrm{d}^2}{\mathrm{d}\mathrm{t}^2}\widetilde{\mathrm{Y}}_{m_i}(\mathrm{t}-\mathrm{c}_0^{-1}|\mathrm{z}_{m_l}-\mathrm{z}_{m_i}|)} 
    \\ &+ \sum\limits_{\substack{j=1 \\ j\neq m}}^{[\varepsilon^{-1}]}\sum\limits_{\substack{i=1}}^{K+1}\mathrm{q}_{z_{m_l},z_{j_i}} \frac{\mathrm{d}^2}{\mathrm{d}\mathrm{t}^2}\widetilde{\mathrm{Y}}_{j_i}(\mathrm{t}-\mathrm{c}_0^{-1}|\mathrm{z}_{m_l}-\mathrm{z}_{j_i}|) = \frac{\partial^2}{\partial t^2} u^\textbf{in}(z_{m_l}) \mbox{ in } (0, \mathrm{T}),
     \\ \widetilde{\mathrm{Y}}_{m_l}(\mathrm{0}) = \frac{\mathrm{d}}{\mathrm{d}\mathrm{t}}\widetilde{\mathrm{Y}}_{m_l}(\mathrm{0}) = 0.
\end{cases}
\end{align}
We observe that
\begin{align}
    \sum\limits_{\substack{i=1 \\ i\neq l}}^{K+1}\mathrm{q}_{z_{m_l},z_{m_i}} \frac{\mathrm{d}^2}{\mathrm{d}\mathrm{t}^2}\widetilde{\mathrm{Y}}_{m_i}(\mathrm{t}-\mathrm{c}_0^{-1}|\mathrm{z}_{m_l}-\mathrm{z}_{m_i}|) \le \underbrace{\Big(\sum\limits_{\substack{i=1 \\ i\neq l}}^{K+1}|\mathrm{q}_{z_{m_l},z_{m_i}}|^2\Big)^\frac{1}{2}}_{=\; \mathcal{O}(d^2)}\;\Big( \sum\limits_{\substack{i=1 \\ i\neq l}}^{K+1}|\frac{\mathrm{d}^2}{\mathrm{d}\mathrm{t}^2}\widetilde{\mathrm{Y}}_{m_i}(\mathrm{t}-\mathrm{c}_0^{-1}|\mathrm{z}_{m_l}-\mathrm{z}_{m_i}|)|^2\Big)^\frac{1}{2}.
\end{align}
Therefore, due to the fact that $d\ll 1$ and regularity in time of the solution $\widetilde{\mathrm{Y}}_{m_l}$, we can neglect this term as it goes to zero.

\noindent
Let us now consider two representative locations, denoted as $z_m \in \Omega_m$ and $z_j \in \mathrm{D}_j$, each belonging to distinct inclusion groups. Subsequently, in order to discuss the existence and uniqueness of the solution of the equation (\ref{matrixdef}), we rewrite the system above as follows:
\begin{align}\label{tran1}
    \begin{cases}
             \mathbb{A}\frac{\mathrm{d}^2}{\mathrm{d}\mathrm{t}^2}\bm{\mathbb{Y}}_m(\mathrm{t}) + \bm{\mathbb{Y}}_m(\mathrm{t}) + \sum\limits_{\substack{j=1 \\ j\neq m}}^{[\varepsilon^{-1}]}\mathbb C_{mj} \cdot\frac{\mathrm{d}^2}{\mathrm{d}\mathrm{t}^2}\mathbb Y_j(\mathrm{t}-\mathrm{c}_0^{-1}|\mathrm{z}_{m}-\mathrm{z}_{j}|) = \mathbcal H_m^\textbf{in} + \textbf{err}_{(4)} \mbox{ in } (0, \mathrm{T}),
             \\ \bm{\mathbb{Y}}_m(\mathrm{0}) = \frac{\mathrm{d}}{\mathrm{d}\mathrm{t}}\bm{\mathbb{Y}}_m(\mathrm{0}) = 0,   
    \end{cases}
\end{align}
where $\textbf{err}_{(4)}:= \sum\limits_{\substack{j=1 \\ j\neq m}}^{[\varepsilon^{-1}]}\sum\limits_{\substack{i=1}}^{K+1}\mathrm{q}_{z_{m_l},z_{j_i}} \frac{\mathrm{d}^2}{\mathrm{d}\mathrm{t}^2}\Big(\widetilde{\mathrm{Y}}_{j_i}(\mathrm{t}-\mathrm{c}_0^{-1}|\mathrm{z}_{m_l}-\mathrm{z}_{j_i}|)-\widetilde{\mathrm{Y}}_{j_i}(\mathrm{t}-\mathrm{c}_0^{-1}|\mathrm{z}_{m}-\mathrm{z}_{j}|)\Big)$,  and
$$\mathbb C_{mj}:=\begin{pmatrix}
                         0 & q_{z_{m_1},z_{j_2}} & \ldots & q_{z_{m_1},z_{j_{K+1}}} \\
                         q_{z_{m_2},z_{j_1}} & 0 & \ldots & q_{z_{m_2},z_{j_{K+1}}} \\
                         \vdots & \vdots & \ddots & \vdots \\
                         q_{z_{m_{K+1}},z_{j_1}} & q_{z_{m_{K+1}},z_{j_2}} & \ldots & 0
                  \end{pmatrix}
$$
is describing the $(K+1)^2$-block interactions between the inclusions located in $\Omega_m$ and $\mathrm{D}_j$ for $j\ne m$ and the incident source $\mathbcal H^\textbf{in}_m := \Big(\frac{\partial^2}{\partial t^2}u^\textbf{in}(z_{m_1}), \frac{\partial^2}{\partial t^2}u^\textbf{in}(z_{m_2}), \ldots, \frac{\partial^2}{\partial t^2}u^\textbf{in}(z_{m_l})\Big)^t$. We define the operator $\mathbb{A}: (\mathrm{L}_r^2)^{K+1}\to (\mathrm{L}_r^2)^{K+1}$ as
\begin{align}
    \mathbb{A} = \big[\mathbcal a_{ij}\big]_{i,j=1}^{K+1}:=
     \begin{pmatrix}
        \hbar_{m_1} & 0 & \dots  & 0\\
        0 & \hbar_{m_1} & \dots  & 0\\
       \vdots & \vdots & \ddots & \vdots\\
       0 & 0 & \dots  & \hbar_{m_{K+1}} 
    \end{pmatrix}, 
\end{align}
with $\mathrm{L}_\ell^2:= \{ f \in \mathrm{L}^2(-r,\mathrm{T}): f=0\; \text{in}\ (-\ell,0)\}$ and $\bm{\mathbb{Y}}_m=\Big(\widetilde{Y}_{m_1},\widetilde{Y}_{m_2},\ldots,\widetilde{Y}_{m_{K+1}}\Big)^t.$
\newline

\noindent
Using Taylor's series expansion, the counting lemma (\ref{counting})  and the regularity of the solution of problem (\ref{matrixmulti-new}), we deduce the following estimates
\begin{align}
    \textbf{err}_{(4)} \nonumber&:= \sum\limits_{\substack{j=1 \\ j\neq m}}^{[\varepsilon^{-1}]}\sum\limits_{\substack{i=1}}^{K+1}\mathrm{q}_{z_{m_l},z_{j_i}} \frac{\mathrm{d}^2}{\mathrm{d}\mathrm{t}^2}\Big(\widetilde{\mathrm{Y}}_{j_i}(\mathrm{t}-\mathrm{c}_0^{-1}|\mathrm{z}_{m_l}-\mathrm{z}_{j_i}|)-\widetilde{\mathrm{Y}}_{j_i}(\mathrm{t}-\mathrm{c}_0^{-1}|\mathrm{z}_{m}-\mathrm{z}_{j}|)\Big)
    \\ &=\nonumber \sum\limits_{\substack{j=1 \\ j\neq m}}^{[\varepsilon^{-1}]}\sum\limits_{\substack{i=1}}^{K+1}\mathrm{q}_{z_{m_l},z_{j_i}} \frac{\mathrm{d}^2}{\mathrm{d}\mathrm{t}^2}\Big(\widetilde{\mathrm{Y}}_{j_i}(\mathrm{t}-\mathrm{c}_0^{-1}|\mathrm{z}_{m_l}-\mathrm{z}_{j_i}|)-\widetilde{\mathrm{Y}}_{j_i}(\mathrm{t}-\mathrm{c}_0^{-1}|\mathrm{z}_{m_l}-\mathrm{z}_{j}|)\Big)
    \\ &+\nonumber\sum\limits_{\substack{j=1 \\ j\neq m}}^{[\varepsilon^{-1}]}\sum\limits_{\substack{i=1}}^{K+1}\mathrm{q}_{z_{m_l},z_{j_i}} \frac{\mathrm{d}^2}{\mathrm{d}\mathrm{t}^2}\Big(\widetilde{\mathrm{Y}}_{j_i}(\mathrm{t}-\mathrm{c}_0^{-1}|\mathrm{z}_{m_l}-\mathrm{z}_{j}|)-\widetilde{\mathrm{Y}}_{j_i}(\mathrm{t}-\mathrm{c}_0^{-1}|\mathrm{z}_{m}-\mathrm{z}_{j}|)\Big)
    \\ &\nonumber= \mathcal{O}\Big(\delta\ d \ \sum\limits_{\substack{j=1 \\ j\neq m}}^{[\varepsilon^{-1}]} \frac{1}{d_{mj}}\Big) = \mathcal{O}(\delta^\frac{1}{3}).
\end{align}
Therefore, we can rewrite the system (\ref{tran1}) as follows
\begin{align}\label{tran5}
    \begin{cases}
             \mathbb{B}\frac{\mathrm{d}^2}{\mathrm{d}\mathrm{t}^2}\bm{\mathbb{Y}}(\mathrm{t}) + \bm{\mathbb{Y}}(\mathrm{t}) = \mathbcal H^\textbf{in} + \textbf{err}_{(4)} \mbox{ in } (0, \mathrm{T}),
             \\ \bm{\mathbb{Y}}(\mathrm{0}) = \frac{\mathrm{d}}{\mathrm{d}\mathrm{t}}\bm{\mathbb{Y}}(\mathrm{0}) = 0,   
    \end{cases}
\end{align}
where, we define the operator $\mathbb{B}: \Big((\mathrm{L}_r^2)^{K+1}\Big)^M\to \Big((\mathrm{L}_r^2)^{K+1}\Big)^M$ as
\begin{align}
    \mathbb{B}= \big[\mathbcal B_{ij}\big]_{i,j=1}^M = \mathbb{B}(t):=
     \begin{pmatrix}
        \mathbb A_{11} & \dots  & \mathbb C_{1M}\cdot\mathbb{T}_{1M}\\
       \vdots & \ddots & \vdots\\
       \mathbb C_{M1}\cdot \mathbb{T}_{M1} & \dots  & \mathbb A_{MM}
    \end{pmatrix}, 
\end{align}
with $\mathrm{L}_r^2:= \{ f \in \mathrm{L}^2(-r,\mathrm{T}): f=0\; \text{in}\ (-r,0)\}$, the translation operators $\mathbb{T}_{1M}:= \big(\mathcal{T}_{11},\mathcal{T}_{12},\ldots,\mathcal{T}_{1M}\big)^t$, i.e. $\mathcal T_{1M}(f)(t):=f(t-\mathrm{c}_0^{-1}|\mathrm{z}_1-\mathrm{z}_M|)$, and $\ell:=\max_{i\neq j}{\mathrm{c}^{-1}_0\vert z_1- z_M\vert}$ and $\bm{\mathbb{Y}}=\Big(\mathbb Y_1,\mathbb Y_2,\ldots,\mathbb Y_M\Big)^t.$
\newline

\noindent
We observe that $\mathcal{B}_{ii}$ is non-singular as $ d_{m_l} \ne 0\; \text{for}\ m= 1,2,...,M.
$
Additionally, we require the supplementary assumption that
$
\sum\limits_{\substack{j=1 \\ j\neq i}}^{\left[\varepsilon^{-1}\right]} \left\| \mathbcal{B}_{ii}^{-1} \mathbcal{B}_{ij} \right\|_{L^2_r} < 1 \quad \text{for} \quad i=1,2,\ldots,M,
$
which is equivalent to the condition
$
\sum\limits_{\substack{j=1 \\ j\neq i}}^{\left[\varepsilon^{-1}\right]} \left\| \mathbcal{B}_{ij} \right\|_{L^2_r} < \left( \left\| \mathbcal{B}_{ii}^{-1} \right\|_{L^2_r} \right)^{-1}.
$
Now, considering $\mathbb{B} = [\mathcal{B}_{ij}]$ with blocks $\mathcal{B}_{ij} \in \mathbb{R}^{K+1 \times K+1}$ for $i,j = 1,2,\ldots,M$, from the previously established result regarding the equivalence of norms, we can infer that
\begin{align} \label{4.9}
    \Vert \mathcal{B}_{ij} \Vert_{L^2_r} \le \sqrt{K+1} \Vert \mathcal{B}_{ij} \Vert_{L^\infty} 
    = \sqrt{K+1} \sum_{\substack{i,l=1 \\ i\neq l}}^{K+1} \mathrm{q}_{\mathrm{z}_{m_l},\mathrm{z}_{j_i}},
\end{align}
where $K$ is finite and $\mathrm{q}_{\mathrm{z}_{m_l},\mathrm{z}_{j_i}}$ is positive. More explicitly, the following condition holds under the assumption\\ \textcolor{black}{$\sqrt{K+1} \sum\limits_{\substack{j=1 \\ j\neq m}}^M \sum\limits_{\substack{i,l=1 \\ i\neq l}}^{K+1} \mathrm{q}_{\mathrm{z}_{m_l},\mathrm{z}_{j_i}} <  \min\limits_{1\le l\le K+1}d_{m_l} $}. If both of these conditions hold, we refer to $\mathbb{B}$ as row block diagonally dominant with respect to the operator norm $\| \cdot \|_{L^2_r}$. It is a well-known result in numerical linear algebra that if $\mathbb{B}$ is row block strictly diagonally dominant, then it is non-singular. Consequently, after neglecting the error order terms, the system of differential equations (\ref{tran5}) reduces to the following form, and similarly, the well-posedness of the system can be established, as discussed in \cite[Section 2.4]{Arpan-Sini-JEvEq}:
\begin{align}\label{wellp}
    \begin{cases}
             \frac{\mathrm{d}^2}{\mathrm{d}t^2}\bm{\mathbb{Y}}(t) + \mathbb{B}^{-1}\bm{\mathbb{Y}}(t) =\mathbb{B}^{-1}\cdot \frac{\mathrm{d}^2}{\mathrm{d}t^2} \mathbcal H^{\textbf{in}} \; \text{in} \; (0, T),
             \\ \bm{\mathbb{Y}}(0) = \frac{\mathrm{d}}{\mathrm{d}t}\bm{\mathbb{Y}}(0) = 0.   
    \end{cases}
\end{align}
The aforementioned discussion on the well-posedness of the system is needed in the following two subsections, namely \ref{4.1.2} and \ref{K-function}, to derive an estimate similar to that in (\ref{mainesti}).


\subsubsection{Reformulation of the Linear System and the Related Lippmann-Schwinger Equation}\label{Reformulation-Linear-System}   

\noindent
We start by recalling the algebraic system
\begin{align}\label{tran2}
    \begin{cases}
             \mathbb{A}\cdot\frac{\mathrm{d}^2}{\mathrm{d}\mathrm{t}^2}\bm{\mathbb{Y}}_m(\mathrm{t}) + \bm{\mathbb{Y}}_m(\mathrm{t}) + \sum\limits_{\substack{j=1 \\ j\neq m}}^{[\varepsilon^{-1}]}\mathbb C_{mj} \cdot\frac{\mathrm{d}^2}{\mathrm{d}\mathrm{t}^2}\mathbb Y_j(\mathrm{t}-\mathrm{c}_0^{-1}|\mathrm{z}_{m}-\mathrm{z}_{j}|) = \frac{\mathrm{d}^2}{\mathrm{d}\mathrm{t}^2} \mathbcal H_m^\textbf{in} + \textbf{err}_{(4)} \mbox{ in } (0, \mathrm{T}),
             \\ \bm{\mathbb{Y}}_m(\mathrm{0}) = \frac{\mathrm{d}}{\mathrm{d}\mathrm{t}}\bm{\mathbb{Y}}_m(\mathrm{0}) = 0. 
    \end{cases}
\end{align}
We intend to link the system mentioned above to an integral equation, similar to what we did for the case of the periodic distribution. We consider the following related integral equation by neglecting the error term as follows:
\begin{equation}\label{effective-equation-np-1}
\rchi_{\mathbf{\Omega}}\;\mathbb{A}\ \cdot\ \frac{\partial ^2}{\partial t^2} \mathbcal{F}_m (\mathrm{x},\mathrm{t}) + \mathbcal{F}_m (\mathrm{x},\mathrm{t}) + \int_{\mathbf{\Omega}}\mathbcal C(x,y) \ \cdot \ \frac{\partial ^2}{\partial t^2}{ \mathbcal{F}_m(x, t-c_0^{-1}\vert x-y\vert)} dy = \frac{\partial ^2}{\partial t^2}\mathbcal{F}_m^\textbf{in}(\mathrm{x},\mathrm{t}),\mbox{ for } \mathrm{x} \in \mathbb{R}^3, \mathrm{t} \in (0, \mathrm{T}),
\end{equation}
where we have $\mathbcal{F}_m^\textbf{in}:=\big(u^\textbf{in},u^\textbf{in},\ldots,u^\textbf{in}\big)^t$. We will also include the initial conditions for $\mathbcal{F}:= \big(\mathcal{F}_{m_1},\mathcal{F}_{m_2},\ldots,\mathcal{F}_{m_{K+1}}\big)^t$ up to the first order in this equation. Here, $\mathbcal C$ is the representation of the corresponding interaction matrix
    \begin{align}
    \mathbcal{C}(x,y) = \begin{pmatrix}
                         0 & \mathbcal q(x,y) & \ldots &\mathbcal q(x,y) \\
                         \mathbcal q(x,y) & 0 & \ldots & \mathbcal q(x,y) \\
                         \vdots & \vdots & \ddots & \vdots \\
                        \mathbcal q(x,y) & \mathbcal q(x,y) & \ldots & 0
                  \end{pmatrix},\; \text{with}\; \mathbcal q(x,y) = \frac{b}{|x-y|}\;\text{and}\; x\ne y.
\end{align}
and 
\begin{align}
    \mathbb{A} = \big[\mathbcal a_{ij}\big]_{i,j=1}^{K+1}:=
     \begin{pmatrix}
        \hbar & 0 & \dots  & 0\\
        0 & \hbar & \dots  & 0\\
       \vdots & \vdots & \ddots & \vdots\\
       0 & 0 & \dots  & \hbar
    \end{pmatrix}.
\end{align}
If we now observe the interaction matrix $\mathbcal C$ and the definition of $\mathbcal F_m,$ we rewrite the above system as follows
\begin{equation}\label{sum}
\rchi_{\mathbf{\Omega}}\;\mathbb{A}\ \cdot\ \frac{\partial ^2}{\partial t^2} \mathbcal F_m (x,t) + \mathbcal{F}_m (\mathrm{x},\mathrm{t}) + \int_{\mathbf{\Omega}}
\begin{pmatrix}
    \mathbcal q(x,y) \cdot \sum\limits_{l=1}^{K+1}\mathcal{F}_{m_l} (\mathrm{y},\mathrm{t}-c_0^{-1}\vert x-y\vert) \\
    \vdots \\
    \mathbcal q(x,y) \cdot \sum\limits_{l=1}^{K+1}\mathcal{F}_{m_l} (\mathrm{y},\mathrm{t}-c_0^{-1}\vert x-y\vert)
\end{pmatrix} dy = \frac{\partial ^2}{\partial t^2}\mathbcal{F}_m^\textbf{in}(\mathrm{x},\mathrm{t}),\mbox{ for } \mathrm{x} \in \mathbb{R}^3, \mathrm{t} \in (0, \mathrm{T}).
\end{equation}
It is evident from the preceding equation that the vector \( \mathbcal{F}(\cdot,\cdot) \) can be reconstructed as the sum of its components, i.e., \( \sum\limits_{\ell=1}^{K+1} \mathcal{F}_{m_\ell}(\cdot,\cdot) \). Indeed, summing in (\ref{sum}), we obtain the following integral equation for $ \mathrm{x} \in \mathbb{R}^3, \mathrm{t} \in (0, \mathrm{T}),$ for $\sum\limits_{\ell=1}^{K+1} \mathcal{F}_{m_\ell}(\cdot,\cdot)$
\begin{equation}
\rchi_{\mathbf{\Omega}}\;\hbar\  \frac{\partial ^2}{\partial t^2} \sum\limits_{l=1}^{K+1}\mathcal F_{m_l} (x,t) + \sum\limits_{l=1}^{K+1}\mathcal F_{m_l} (x,t) + \int_{\mathbf{\Omega}}(K+1)\
    \mathbcal q(x,y) \ \frac{\partial ^2}{\partial t^2}\sum\limits_{l=1}^{K+1}\mathcal{F}_{m_l} (\mathrm{y},\mathrm{t}-c_0^{-1}\vert x-y\vert) dy = (\mathrm{K}+1) \frac{\partial ^2}{\partial t^2}u^\textbf{in}(\mathrm{x},\mathrm{t}).
\end{equation}
Then, we arrive at the following expression after rewriting the unknown as $\sum\limits_{l=1}^{K+1}\mathcal R_{m_l} (x,t) := \frac{1}{\mathrm{K}+1}\sum\limits_{l=1}^{K+1}\mathcal F_{m_l} (x,t)$,
\begin{equation}\label{effective-equation-np}
\rchi_{\mathbf{\Omega}}\;\hbar\  \frac{\partial ^2}{\partial t^2} \sum\limits_{l=1}^{K+1}\mathcal R_{m_l} (x,t) + \sum\limits_{l=1}^{K+1}\mathcal R_{m_l} (x,t) + \int_{\mathbf{\Omega}}(K+1)\
    \mathbcal q(x,y) \ \frac{\partial ^2}{\partial t^2}\sum\limits_{l=1}^{K+1}\mathcal{R}_{m_l} (\mathrm{y},\mathrm{t}-c_0^{-1}\vert x-y\vert) dy = \frac{\partial ^2}{\partial t^2}u^\textbf{in}(\mathrm{x},\mathrm{t}).
\end{equation}


\subsubsection{The generated effective medium} \label{4.1.2}        

\noindent
Similar to the periodic case, we first define the unknown variable $\sum\limits_{l=1}^{K+1}\mathcal R_{m_l} = \frac{\partial ^2}{\partial t^2} \sum\limits_{l=1}^{K+1}\mathbcal G_{m_l},$ where $\mathbcal G_m:=\sum\limits_{l=1}^{K+1}\mathbcal G_{m_l}$ satisfies the following Lippmann-Schwinger equation
\begin{equation}
\rchi_{\mathbf{\Omega}}\;\hbar\ \frac{\partial ^2}{\partial t^2} \mathbcal G_m (x,t) + \mathbcal G_m (x,t) + \int_{\mathbf{\Omega}}
    (K+1)\ \mathbcal q(x,y)  \frac{\partial ^2}{\partial t^2}\mathbcal G_m (\mathrm{y},\mathrm{t}-c_0^{-1}\vert x-y\vert) dy = u^\textbf{in}(\mathrm{x},\mathrm{t}),\mbox{ for } \mathrm{x} \in \mathbb{R}^3, \mathrm{t} \in (0, \mathrm{T}),
\end{equation}
with zero initial conditions for $\mathbcal G_m$ up to the first order in this equation. Then, we consider
\begin{align}
    \mathbcal V(x,t) := \begin{cases}
                       \mathbcal G_m (x,t) + \hbar\ \frac{\partial ^2}{\partial t^2} \mathbcal G_m (x,t) & \text{if}\; (x,t) \in \mathbf{\Omega}\times(0,T) \\ \displaystyle
                       u^\textbf{in}(\mathrm{x},\mathrm{t}) - \int_{\mathbf{\Omega}}(K+1)\frac{b}{4\pi\vert x-y\vert} \frac{\partial ^2}{\partial t^2}{\mathbcal G _m(y, t-c_0^{-1}\vert x-y\vert)} dy & \text{for}\; (x,t) \in \mathbb R^3\setminus\mathbf{\Omega}\times(0,T).
                 \end{cases}
\end{align}
We set $\mathbcal W(x,t) := u^\textbf{in}(\mathrm{x},\mathrm{t}) - \mathbcal V(x,t).$ Therefore, we deduce that
\begin{align}\label{integralW}
    \mathbcal W(x,t) &\nonumber= \int_{\mathbf{\Omega}}(K+1)\frac{b}{4\pi\vert x-y\vert} \frac{\partial ^2}{\partial t^2}{ \mathbcal G_m (y, t-c_0^{-1}\vert x-y\vert)} dy
    \\ &\nonumber= \int_{\mathbf{\Omega}}(K+1)\frac{b}{4\pi\vert x-y\vert} \frac{\partial ^2}{\partial t^2}{\sum\limits_{l=1}^{K+1}\mathcal G_{m_l} (y, t-c_0^{-1}\vert x-y\vert)} dy
    \\ & = \int_{\mathbf{\Omega}}(K+1)\frac{b}{4\pi\vert x-y\vert} {\sum\limits_{l=1}^{K+1}\mathcal R_{m_l} (y, t-c_0^{-1}\vert x-y\vert)} dy.
\end{align}
Let us begin by recalling the related scattered field approximation (\ref{assymptotic-expansion-us}) for \( x \in \mathbb{R}^3 \setminus \mathbf{\Omega} \). We will rewrite the approximation under the assumption that the bubbles inside each \(\Omega_m\) have identical shapes and material properties, with \(\mathrm{K} \in \mathbb{N}\), as follows after rewriting the unknown as $\sum\limits_{l=1}^{K+1}\widetilde{\mathcal P}_{m_l} (x,t) := \frac{1}{\mathrm{K}+1}\sum\limits_{l=1}^{K+1}\widetilde{\mathrm{Y}}_{m_l} (x,t)$
\begin{align}\label{est1}
    u^\mathrm{s}(\mathrm{x},\mathrm{t}) 
    \nonumber& = \sum_{m=1}^\mathrm{M}\frac{\alpha_m \rho_\mathrm{c}}{4\pi|\mathrm{x}-\mathrm{z}_i|}\;|\mathrm{D}_m|\;\widetilde{\mathrm{Y}}_m\big(\mathrm{t}-\mathrm{c}_0^{-1}|\mathrm{x}-\mathrm{z}_m|\big) + \mathcal{O}(\delta^{2-l})\; \text{as}\; \delta\to 0
    \\ \nonumber&= \sum\limits_{\substack{m=1 }}^{[\varepsilon^{-1}]}\sum\limits_{\substack{l=1}}^{K+1}\frac{\alpha_{m_l} \rho_\mathrm{c}}{4\pi|\mathrm{x}-\mathrm{z}_{m_1}|}\;|\mathrm{D}_{m_l}|\;\frac{\rho_{b_l}}{\mathrm{k}_{b_l}}\widetilde{\mathrm{Y}}_{m_l}(\mathrm{t}-\mathrm{c}_0^{-1}|x-\mathrm{z}_{m_l}|) 
    \\ \nonumber&+ \sum\limits_{\substack{m=1 }}^{[\varepsilon^{-1}]}\sum\limits_{\substack{l=1}}^{K+1}\frac{\alpha_{m_l} \rho_\mathrm{c}}{4\pi}\;|\mathrm{D}_{m_l}|\;\frac{\rho_{b_l}}{\mathrm{k}_{b_l}}\Big(\frac{1}{|\mathrm{x}-\mathrm{z}_{m_l}|}-\frac{1}{|\mathrm{x}-\mathrm{z}_{m_1}|}\Big)\; \widetilde{\mathrm{Y}}_{m_l}(\mathrm{t}-\mathrm{c}_0^{-1}|x-\mathrm{z}_{m_l}|) + \mathcal{O}(\delta^{2-l})
    \\ \nonumber&= \sum\limits_{\substack{m=1 }}^{[\varepsilon^{-1}]}\frac{b(\mathrm{K}+1)}{4\pi|\mathrm{x}-\mathrm{z}_{m_1}|}\sum\limits_{\substack{l=1}}^{K+1}\widetilde{\mathcal{P}}_{m_l}(\mathrm{t}-\mathrm{c}_0^{-1}|x-\mathrm{z}_{m_l}|) \quad \quad \Big[\text{as},\; \alpha_{m_l}\;|\mathrm{D}_{m_l}|\;\frac{\rho_{b_l}}{\mathrm{k}_{b_l}} = b + \mathcal{O}(\delta^2)  \Big]
    \\ &+ \underbrace{\sum\limits_{\substack{m=1 }}^{[\varepsilon^{-1}]}\sum\limits_{\substack{l=1}}^{K+1}\frac{b}{4\pi}\Big(\frac{1}{|\mathrm{x}-\mathrm{z}_{m_l}|}-\frac{1}{|\mathrm{x}-\mathrm{z}_{m_1}|}\Big)\; \widetilde{\mathrm{Y}}_{m_l}(\mathrm{t}-\mathrm{c}_0^{-1}|x-\mathrm{z}_{m_l}|)}_{:=\; \textbf{err}^{(1)}} + \mathcal{O}(\delta^{2-l}).
\end{align}
Let us estimate $\textbf{err}^{(1)}$.
\begin{align}\label{est2}
    \textbf{err}^{(1)}&\nonumber:= \sum\limits_{\substack{m=1 }}^{[\varepsilon^{-1}]}\sum\limits_{\substack{l=1}}^{K+1}\frac{b}{4\pi}\Big(\frac{1}{|\mathrm{x}-\mathrm{z}_{m_l}|}-\frac{1}{|\mathrm{x}-\mathrm{z}_{m_1}|}\Big)\; \widetilde{\mathrm{Y}}_{m_l}(\mathrm{t}-\mathrm{c}_0^{-1}|x-\mathrm{z}_{m_l}|)
    \\ &\nonumber \lesssim \varepsilon \Big(\sum\limits_{\substack{m=1 }}^{[\varepsilon^{-1}]}\sum\limits_{\substack{l=1}}^{K+1} |z_{m_l}-z_{m_1}|^2\Big)^\frac{1}{2}\;  \Big(\sum\limits_{\substack{m=1 }}^{[\varepsilon^{-1}]}\sum\limits_{\substack{l=1}}^{K+1} |\widetilde{\mathrm{Y}}_{m_l}|^2\Big)^\frac{1}{2}
    \\ & \lesssim \varepsilon^\frac{5}{6}\Big( \sum\limits_{\substack{m=1 }}^{[\varepsilon^{-1}]} |\mathbcal F_m|^2 + \sum\limits_{\substack{m=1 }}^{[\varepsilon^{-1}]} |\mathbb Y_m - \mathbcal F_m|^2\Big)^\frac{1}{2} \lesssim \varepsilon^\frac{1}{3}.
\end{align}
The reason we obtain the estimate above is to rewrite the unknown \(\sum\limits_{l=1}^{K+1}\widetilde{\mathrm{Y}}_{m_l}\) as the unknown \(\mathbb{Y}_m\). We then use the regularity assumption for \(\mathbcal{F}_m\) (we can prove the regularity similar to what was discussed in Corollary \ref{cor}) and the estimate \(\sum\limits_{m=1}^{[\varepsilon^{-1}]} |\mathbb{Y}_m - \mathbcal{F}_m|^2 \lesssim \varepsilon^{-\frac{1}{3}}\). To achieve this estimate, we start by discretizing the integral equation (\ref{effective-equation-np-1}). Next, we estimate the difference between this integral equation and the general algebraic system (\ref{tran2}). To do this, we rely on the regularity assumptions of \(\mathbcal{F}_m\), which can be derived similarly to the discussion in Section \ref{regularity}. This leads us to a system similar to that derived in (\ref{matrix2}). Subsequently, we utilize the well-posedness of the system (\ref{wellp}) to derive the aforementioned estimate.
\noindent
Consequently, we obtain that
\begin{align}\label{res}
    u^\mathrm{s}(\mathrm{x},\mathrm{t}) 
    \nonumber& = \sum\limits_{\substack{m=1 }}^{[\varepsilon^{-1}]}\frac{b(\mathrm{K}+1)}{4\pi|\mathrm{x}-\mathrm{z}_{m_1}|} \sum\limits_{\substack{l=1}}^{K+1}\widetilde{\mathcal{P}}_{m_l}(\mathrm{t}-\mathrm{c}_0^{-1}|x-\mathrm{z}_{m_l}|) + \mathcal{O}(\delta^\frac{1}{3})
    \\ \nonumber& = \sum\limits_{\substack{m=1 }}^{[\varepsilon^{-1}]}\frac{b(\mathrm{K}+1)}{4\pi|\mathrm{x}-\mathrm{z}_{m_1}|} \sum\limits_{\substack{l=1}}^{K+1}\mathcal{R}_{m_l}(z_{m_l},\mathrm{t}-\mathrm{c}_0^{-1}|x-\mathrm{z}_{m_l}|)
    \\ &+ \underbrace{\sum\limits_{\substack{m=1 }}^{[\varepsilon^{-1}]}\frac{b(\mathrm{K}+1)}{4\pi|\mathrm{x}-\mathrm{z}_{m_1}|} \sum\limits_{\substack{l=1}}^{K+1}\widetilde{\mathcal{P}}_{m_l}(t-c_0^{-1}|x-z_{m_l})-\mathcal R_{m_l} (z_{m_l},t-c_0^{-1}|x-z_{m_l}|)}_{:=\; \textbf{err}^{(2)}}+ \mathcal{O}(\delta^\frac{1}{3}).
\end{align}
To estimate the term $\textbf{err}^{(2)}$, we do the following
\begin{align}\label{e2}
    \textbf{err}^{(2)} \nonumber&:= \sum\limits_{\substack{m=1 }}^{[\varepsilon^{-1}]}\frac{b(\mathrm{K}+1)}{4\pi|\mathrm{x}-\mathrm{z}_{m_1}|}\frac{1}{\mathrm{K}+1}\sum\limits_{\substack{l=1}}^{K+1}\Big(\widetilde{Y}_{m_l}-\mathcal{F}_{m_l}\Big)(\mathrm{t}-\mathrm{c}_0^{-1}|x-\mathrm{z}_{m_l}|)
    \\ &\nonumber \lesssim \mathcal{O}\Bigg(\varepsilon\;b \Big(\sum\limits_{\substack{m=1}}^{[\varepsilon^{-1}]}\sum\limits_{\substack{l=1}}^{K+1} \frac{1}{|x-z_{m_1}|^2}\Big)^\frac{1}{2}\;\Big(\sum\limits_{\substack{m=1}}^{[\varepsilon^{-1}]}\sum\limits_{\substack{l=1}}^{K+1}|\widetilde{\mathrm{Y}}_{m_l}(t-c_0^{-1}|x-z_{m_l})-\mathcal F_{m_l} (z_{m_l},t-c_0^{-1}|x-z_{m_l}|)|^2\Big)^\frac{1}{2}\Bigg)
    \\ & \lesssim \mathcal{O}\Bigg(\varepsilon^\frac{1}{2}\;\Big(\sum\limits_{\substack{m=1}}^{[\varepsilon^{-1}]}|\mathbb{Y}_{m}-\mathbcal F_{m}|^2\Big)^\frac{1}{2}\Bigg)
    \lesssim \varepsilon^\frac{1}{3}.
\end{align}
Therefore, we obtain from (\ref{res}) and (\ref{e2}) that
\begin{align}\label{us}
    u^\mathrm{s}(\mathrm{x},\mathrm{t}) = \sum\limits_{\substack{m=1 }}^{[\varepsilon^{-1}]}\frac{b(\mathrm{K}+1)}{4\pi|\mathrm{x}-\mathrm{z}_{m_1}|} \sum\limits_{\substack{l=1}}^{K+1}\mathcal{R}_{m_l}(z_{m_l},\mathrm{t}-\mathrm{c}_0^{-1}|x-\mathrm{z}_{m_l}|) + \mathcal{O}(\delta^\frac{1}{3}).
\end{align}
Consequently, we have from (\ref{integralW}) and (\ref{us}) that
\begin{align}
    \nonumber &\mathbcal W(x,t) -u^\mathrm{s}(\mathrm{x},\mathrm{t}) =\underbrace{\int_{\mathbf{\Omega}\setminus{\bigcup\limits_{m=1}^{{[\varepsilon^{-1}]}}\Omega_m}}\frac{(K+1)b}{4\pi\vert x-y\vert} \sum\limits_{l=1}^{K+1}\mathcal R_{m_l}(y, t-c_0^{-1}|x-y|)\; dy}_{:=\; \textbf{err}^{1}}
    \\ \nonumber&+ \underbrace{\sum\limits_{\substack{m=1 }}^{[\varepsilon^{-1}]}\int_{\Omega_m}b\frac{(K+1)}{4\pi|x-y|}\sum\limits_{l=1}^{K+1}\mathcal R_{m_l}(y, t-c_0^{-1}|x-y|)\ dy -  \;
    b(\mathrm{K}+1)\frac{\varepsilon}{4\pi|x-z_{m_1}|} {\sum\limits_{l=1}^{K+1}\mathcal R_{m_l}(z_{m_l}, t-c_0^{-1}|x-z_{m_l}|)}}_{:=\;\textbf{err}^{(2)}}.
\end{align}
Next, we utilizing the fact that $\text{vol}(\Omega_m) = \varepsilon,$ we rewrite the following term as
\begin{align}
    \nonumber &\mathbcal W(x,t) -u^\mathrm{s}(\mathrm{x},\mathrm{t}) =\underbrace{\int_{\mathbf{\Omega}\setminus{\bigcup\limits_{m=1}^{{[\varepsilon^{-1}]}}\Omega_m}}\frac{(K+1)b}{4\pi\vert x-y\vert} \sum\limits_{l=1}^{K+1}\mathcal R_{m_l}(y, t-c_0^{-1}|x-y|)\; dy}_{:=\; \textbf{err}^{(1)}}
    \\ &+ \underbrace{\sum\limits_{\substack{m=1 }}^{[\varepsilon^{-1}]}b(K+1)\int_{\Omega_m}\Bigg( \frac{1}{4\pi|x-y|}\sum\limits_{l=1}^{K+1}\mathcal R_{m_l}(y, t-c_0^{-1}|x-y|)\ dy -   \frac{1}{4\pi|x-z_{m_1}|} {\sum\limits_{l=1}^{K+1}\mathcal R_{m_l}(z_{m_l}, t-c_0^{-1}|x-z_{m_l}|)}\Bigg)dy}_{:=\;\textbf{err}^{(2)}}
\end{align}
To estimate the values of $\textbf{err}^{(1)}$ and $\textbf{err}^{(2)}$, we follow a similar process as described in Section \ref{end}. Here, we require the smoothness of $\sum\limits_{l=1}^{K+1}\mathcal R_{m_l}$. We need it to belong to the space $\mathrm{C}^1\big(0,\mathrm{T};\mathrm{L}^\infty(\mathbf{\Omega})\big)$ and its partial derivative $\partial_{x_i}\sum\limits_{l=1}^{K+1}\mathcal R_{m_l}$ to be in $\mathrm{L}^\infty\big(0,\mathrm{T};\mathrm{L}^{\infty}(\mathbf{\Omega})\big)$. Looking at the integral equation (\ref{effective-equation-np}) for $\sum\limits_{l=1}^{K+1}\mathcal R_{m_l}$, we notice it is similar to that of (\ref{effective-equation}). Therefore, we apply the same argument as discussed in Section \ref{regularity} to establish the necessary smoothness. Therefore, we obtain the estimate
\begin{align}
    \nonumber\mathbcal W(x,t) -u^\mathrm{s}(\mathrm{x},\mathrm{t}) = \mathcal{O}(\varepsilon^\frac{1}{3}).
\end{align}


\subsection{Case when \texorpdfstring{$K \in C^1$}{K} with \texorpdfstring{$\nabla K(x) \ne 0$ everywhere in the set of discontinuity of $\Big[K(\cdot)+1\Big]$}{K}}
\label{K-function}         

\noindent
To describe correctly this number of bubbles, let us be given $K : \mathbb R^3 \to \mathbb R$, is a positive real valued function which has continuous first derivative and $\nabla \mathrm{K}(x) \neq 0$ everywhere in the set of discontinuity of the function $\Big[K(\cdot)+1\Big].$ Then, according to Assumption \ref{as2}, each $\Omega_m$ contains exactly $\big[K(z_{m_l})+1\big]$ number of bubbles for $m=1,2,\ldots,[\varepsilon^{-1}]$ i.e. $\mathrm{D}:= \bigcup\limits_{m=1}^{{[\varepsilon^{-1}]}}\bigcup\limits_{l=1}^{\big[K(z_{m_l})+1\big]}\mathrm{D}_{m_l}.$ Additionally, we assume that the bubbles included in each $\Omega_m$ possesses identical shapes and material properties.

\noindent
We start with recalling the algebric system and the corresponding general Lippmann-Schwinger equation similar to (\ref{tran2}) and (\ref{effective-equation-np-1}) as follows:
\begin{align}\label{tran3}
    \begin{cases}
             \mathbb{A}\cdot\frac{\mathrm{d}^2}{\mathrm{d}\mathrm{t}^2}\bm{\mathbb{Y}}_m(\mathrm{t}) + \bm{\mathbb{Y}}_m(\mathrm{t}) + \sum\limits_{\substack{j=1 \\ j\neq m}}^{[\varepsilon^{-1}]}\mathbb C_{mj} \cdot\frac{\mathrm{d}^2}{\mathrm{d}\mathrm{t}^2}\mathbb Y_j(\mathrm{t}-\mathrm{c}_0^{-1}|\mathrm{z}_{m}-\mathrm{z}_{j}|) = \frac{\mathrm{d}^2}{\mathrm{d}\mathrm{t}^2} \mathbcal H_m^\textbf{in} + \textbf{err}_{(4)} \mbox{ in } (0, \mathrm{T}),
             \\ \bm{\mathbb{Y}}_m(\mathrm{0}) = \frac{\mathrm{d}}{\mathrm{d}\mathrm{t}}\bm{\mathbb{Y}}_m(\mathrm{0}) = 0, 
    \end{cases}
\end{align}
and
\begin{equation}\label{effective-equation-np-3}
\rchi_{\mathbf{\Omega}}\;\overline{\mathbb{A}}\ \cdot\ \frac{\partial ^2}{\partial t^2} \mathbcal{F}_m (\mathrm{x},\mathrm{t}) + \mathbcal{F}_m (\mathrm{x},\mathrm{t}) + \int_{\mathbf{\Omega}}\overline{\mathbcal C}(x,y) \ \cdot \ \frac{\partial ^2}{\partial t^2}{ \mathbcal{F}_m(x, t-c_0^{-1}\vert x-y\vert)} dy = \frac{\partial ^2}{\partial t^2}\mathbcal{F}_m^\textbf{in}(\mathrm{x},\mathrm{t}),\mbox{ for } \mathrm{x} \in \mathbb{R}^3, \mathrm{t} \in (0, \mathrm{T}),
\end{equation}
where $$\mathbb C_{mj}:=\begin{pmatrix}
                         0 & q_{z_{m_1},z_{j_2}} & \ldots & q_{z_{m_1},z_{j_{[K(z_{m_l})+1]}}} \\
                         q_{z_{m_2},z_{j_1}} & 0 & \ldots & q_{z_{m_2},z_{j_{[K(z_{m_l})+1]}}} \\
                         \vdots & \vdots & \ddots & \vdots \\
                         q_{z_{m_{[K(z_{m_l})+1]}},z_{j_1}} & q_{z_{m_{[K(z_{m_l})+1]}},z_{j_2}} & \ldots & 0
                  \end{pmatrix};\; 
                  \mathbb{A} = \big[\mathbcal a_{ij}\big]_{i,j=1}^{[K(z_{m_l})+1]}:=
     \begin{pmatrix}
        \hbar_{m_1} & 0 & \dots  & 0\\
        0 & \hbar_{m_1} & \dots  & 0\\
       \vdots & \vdots & \ddots & \vdots\\
       0 & 0 & \dots  & \hbar_{m_{[K(z_{m_l})+1]}} 
    \end{pmatrix}
$$
is describing the $[K(z_{m_l})+1]^2$-block interactions between the inclusions located in $\Omega_m$ and $\Omega_j$ for $j\ne m$ and the incident source $\mathbcal H^\textbf{in}_m := \Big(\frac{\partial^2}{\partial t^2}u^\textbf{in}(z_{m_1}), \frac{\partial^2}{\partial t^2}u^\textbf{in}(z_{m_2}), \ldots, \frac{\partial^2}{\partial t^2}u^\textbf{in}(z_{m_l})\Big)^t$. We define the operator $\mathbb{A}: (\mathrm{L}_r^2)^{[K(z_{m_l})+1]}\to (\mathrm{L}_r^2)^{[K(z_{m_l})+1]}$ with $\mathrm{L}_r^2:= \{ f \in \mathrm{L}^2(-r,\mathrm{T}): f=0\; \text{in}\ (-r,0)\}$ and $\bm{\mathbb{Y}}_m=\Big(\widetilde{Y}_{m_1},\widetilde{Y}_{m_2},\ldots,\widetilde{Y}_{m_{[K(z_{m_l})+1]}}\Big)^t.$

\noindent
Due to the assumption regarding the similar shape and material properties of the bubbles and in relation to the Lippmann-Schwinger equation, we also have:
\[
\overline{\mathbcal{C}}(x,y) = \begin{pmatrix}
    0 & \frac{b}{|x-y|} & \ldots & \frac{b}{|x-y|} \\
    \frac{b}{|x-y|} & 0 & \ldots & \frac{b}{|x-y|} \\
    \vdots & \vdots & \ddots & \vdots \\
    \frac{b}{|x-y|} & \frac{b}{|x-y|} & \ldots & 0
\end{pmatrix}, \quad x \ne y, \quad \text{and}\;
\overline{\mathbb{A}} = \big[\mathbcal{a}_{ij}\big]_{i,j=1}^{[K(z_{m_l})+1]} :=
\begin{pmatrix}
    \hbar & 0 & \ldots & 0 \\
    0 & \hbar & \ldots & 0 \\
    \vdots & \vdots & \ddots & \vdots \\
    0 & 0 & \ldots & \hbar
\end{pmatrix},
\]
where we define \(\mathbcal{F}_m^\textbf{in} := \big(u^\textbf{in}, u^\textbf{in}, \ldots, u^\textbf{in}\big)^t\). We will also include the zero initial conditions up to the first order for \(\mathbcal{F} := \big(\mathcal{F}_{m_1}, \mathcal{F}_{m_2}, \ldots, \mathcal{F}_{m_{[K(z_{m_l})+1]}}\big)^t\) in the equation \((\ref{effective-equation-np-3})\). Here, \(\mathbcal{C}\) represents the corresponding interaction matrix.

\vspace{0.2in}
\noindent
Similar to the previous section as derived in (\ref{effective-equation-np}), we obtain an equivalent representation of the Lippmann-Schwinger equation by considering the unknown as \(\sum\limits_{l=1}^{[K(z_{m_l})+1]}\widetilde{\mathcal{R}}_{m_l}(x,t) := \frac{1}{[K(z_{m_l})+1]}\sum\limits_{l=1}^{[K(z_{m_l})+1]}\mathcal{F}_{m_l}(x,t)\):
\begin{equation}\label{effective-equation-np-4}
\rchi_{\mathbf{\Omega}} \; \hbar \frac{\partial^2}{\partial t^2} \sum\limits_{l=1}^{[K(z_{m_l})+1]}\widetilde{\mathcal{R}}_{m_l}(x,t) + \sum\limits_{l=1}^{[K(z_{m_l})+1]}\widetilde{\mathcal{R}}_{m_l}(x,t) + \int_{\mathbf{\Omega}} [K(y)+1] \frac{b}{|x-y|} \frac{\partial^2}{\partial t^2} \sum\limits_{l=1}^{[K((z_{m_l})+1]}\widetilde{\mathcal{R}}_{m_l}(y,t-c_0^{-1}|x-y|) \, dy = \frac{\partial^2}{\partial t^2} u^{\text{in}}(x,t).
\end{equation}
The key difference in this section is that the dimensions of the matrix operators now depend on the space variable through \( [K(z_{m_l})+1] \). We also need similar regularity assumptions for \( \sum\limits_{l=1}^{[K(z_{m_l})+1]}\widetilde{\mathcal{R}}_{m_l}(x,t) \) as derived in Section \ref{regularity}. From the previous section, we observe that since \( \sum\limits_{l=1}^{[K(z_{m_l})+1]}\mathcal{F}_{m_l}(x,t) \in H^4_{0,\sigma}(0,\mathrm{T};\mathrm{L}^2(\mathbf{\Omega})) \) and \( [K(z_{m_l})+1] \in \mathrm{L}^\infty(\mathbf{\Omega}) \), it follows that \( \sum\limits_{l=1}^{[K(z_{m_l})+1]}\widetilde{\mathcal{R}}_{m_l}(x,t) \) belongs to \( H^4_{0,\sigma}(0,\mathrm{T};\mathrm{L}^2(\mathbf{\Omega})) \). Therefore, by applying similar arguments as in Section \ref{regularity}, we can derive the necessary regularity results for the solution to the equation (\ref{effective-equation-np-4}) or equivalently (\ref{effective-equation-np-3}) to perform the estimates.

\bigskip

\noindent
Similar to the expression for the scattered field as derived in (\ref{us}), using techniques similar to those employed in deriving estimates \(\eqref{est1}\), \(\eqref{est2}\), \(\eqref{res}\), and \(\eqref{e2}\), we can also obtain the following expression for the scattered field when \(K\) satisfies the properties mentioned at the beginning of this section:
\begin{align}
    u^\mathrm{s}(\mathrm{x},\mathrm{t}) 
    \nonumber &= \sum\limits_{\substack{m=1}}^{[\varepsilon^{-1}]}\frac{b}{4\pi|\mathrm{x}-\mathrm{z}_{m_1}|}\ \big[K(z_{m_l})+1\big]\ \varepsilon \sum\limits_{\substack{l=1}}^{\big[K(z_{m_l})+1\big]}\widetilde{\mathcal{R}}_{m_l}(z_{m_l},\mathrm{t}-\mathrm{c}_0^{-1}|x-\mathrm{z}_{m_l}|) + \mathcal{O}(\varepsilon^\frac{1}{3}).
\end{align}
Then, similar to the case for deriving the expression for \(\mathbcal{W}\) as in (\ref{integralW}), we state the following integral, which can also be seen from (\ref{effective-equation-np-4})
\begin{align}
     \mathbcal{W}(x,t) &= \int_{\mathbf{\Omega}} \big[K(y)+1\big]\frac{b}{4\pi\vert x-y\vert} \sum\limits_{l=1}^{\big[K(z_{m_l})+1\big]}\widetilde{\mathcal{R}}_{m_l}(y, t-c_0^{-1}\vert x-y\vert) \, dy.
\end{align}
Then, we have
\begin{align}
    \mathbcal W(x,t) -u^\mathrm{s}(\mathrm{x},\mathrm{t}) \nonumber&=  \int_{\mathbf{\Omega}}\big[K(y)+1\big]\frac{b}{4\pi\vert x-y\vert} {\sum\limits_{l=1}^{\big[K(z_{m_l})+1\big]}\widetilde{\mathcal R}_{m_l} (y, t-c_0^{-1}\vert x-y\vert)} dy 
    \\ \nonumber &- \sum\limits_{\substack{m=1 }}^{[\varepsilon^{-1}]}\frac{b}{4\pi|\mathrm{x}-\mathrm{z}_{m_1}|}[K(z_{m_l})+1]\ \varepsilon\sum\limits_{\substack{l=1}}^{\big[K(z_{m_l})+1\big]}\widetilde{\mathcal{R}}_{m_l}(z_{m_l},\mathrm{t}-\mathrm{c}_0^{-1}|x-\mathrm{z}_{m_l}|)
    + \mathcal{O}(\varepsilon^\frac{1}{3})
    \\ \nonumber& =  \sum\limits_{\substack{m=1 }}^{[\varepsilon^{-1}]}\int_{\Omega_m}\big[K(y)+1\big]\frac{b}{4\pi\vert x-y\vert} {\sum\limits_{l=1}^{\big[K(z_{m_l})+1\big]}\widetilde{\mathcal R}_{m_l} (y, t-c_0^{-1}\vert x-y\vert)} dy
    \\ \nonumber&+ \sum\limits_{\substack{m=1 }}^{[\varepsilon^{-1}]}\int_{\mathbf{\Omega}\setminus{\bigcup\limits_{m=1}^{{[\varepsilon^{-1}]}}\Omega_m}}\big[K(y)+1\big]\frac{b}{4\pi\vert x-y\vert} {\sum\limits_{l=1}^{\big[K(z_{m_l})+1\big]}\widetilde{\mathcal R}_{m_l} (y, t-c_0^{-1}\vert x-y\vert)} dy
    \\ \nonumber&- \sum\limits_{\substack{m=1 }}^{[\varepsilon^{-1}]}\frac{b}{4\pi|\mathrm{x}-\mathrm{z}_{m_1}|}\ [K(z_{m_l})+1]\ \varepsilon \sum\limits_{\substack{l=1}}^{\big[K(z_{m_l})+1\big]}\widetilde{\mathcal{R}}_{m_l}(z_{m_l},\mathrm{t}-\mathrm{c}_0^{-1}|x-\mathrm{z}_{m_l}|)
    + \mathcal{O}(\varepsilon^\frac{1}{3}).
\end{align}
Similar to the estimate for $E_{(1)}$ (\ref{esti1}) and using the fact that $\big[K(y)+1\big]\ \text{in}\ \mathrm{L}^\infty(\Omega)$, we deduce that
$$\sum\limits_{\substack{m=1 }}^{[\varepsilon^{-1}]}\int_{\mathbf{\Omega}\setminus{\bigcup\limits_{m=1}^{{[\varepsilon^{-1}]}}\Omega_m}}\big[K(y)+1\big]\frac{b}{4\pi\vert x-y\vert} {\sum\limits_{l=1}^{\big[K(z_{m_l})+1\big]}\widetilde{\mathcal R}_{m_l} (y, t-c_0^{-1}\vert x-y\vert)} dy = \mathcal{O}(\varepsilon^\frac{1}{3}).$$
Therefore, we have, 
\begin{align}\label{term}
  \mathbcal W(x,t) -u^\mathrm{s}(\mathrm{x},\mathrm{t}) \nonumber& =  
  \sum\limits_{\substack{m=1 }}^{[\varepsilon^{-1}]}\int_{\Omega_m}\big[K(y)+1\big]\frac{b}{4\pi\vert x-y\vert} {\sum\limits_{l=1}^{\big[K(z_{m_l})+1\big]}\widetilde{\mathcal R}_{m_l} (y, t-c_0^{-1}\vert x-y\vert)} dy
  \\ &- \sum\limits_{\substack{m=1 }}^{[\varepsilon^{-1}]}\frac{b}{4\pi|\mathrm{x}-\mathrm{z}_{m_1}|}\ [K(z_{m_l})+1]\ \varepsilon \sum\limits_{\substack{l=1}}^{\big[K(z_{m_l})+1\big]}\widetilde{\mathcal{R}}_{m_l}(z_{m_l},\mathrm{t}-\mathrm{c}_0^{-1}|x-\mathrm{z}_{m_l}|) + \mathcal{O}(\varepsilon^\frac{1}{3})
\end{align}
We now observe that the set of discontinuities \(\mathrm{S}\) of the function \(\big[\mathrm{K}(x)\big]\) is given by:
\[ \mathrm{S} := \big\{ x \in \mathbf{\Omega} \subseteq \mathbb{R}^3 \mid \mathrm{K}(x) \in \mathbb{Z} \big\}. \]
Then, as a corollary of the Implicit Function Theorem, we see that the level sets \(\big\{ x \in \mathbf{\Omega} \subseteq \mathbb{R}^3 \mid \mathrm{K}(x) = n \big\}\) (where \( n \) is an integer) are 2-dimensional surfaces in \(\mathbb{R}^3\) only when \(\mathrm{K}(x)\) has continuous first derivatives, i.e., \(\mathrm{K}(x) \in C^1\), with \(\nabla \mathrm{K}(x) \ne 0\) anywhere on these level sets (see \cite[Theorem 6.5]{calculus}). Therefore, \(\mathrm{S}\) will be a 2-dimensional surface, as it is a countable union of 2-dimensional surfaces.

\noindent
To estimate the terms involve in (\ref{term}), we consider the following. The point $z_{m_l}$ is located near one of the $\Omega_m$'s touching the boundary $\mathrm{S}$. In this case, we split the two integrals in (\ref{term}) into two regions. We denote by $\mathcal{N}_{(1)}$ the part that involves $\Omega_m$'s close to $z_{m_l}$ and the remaining part as $\mathcal{N}_{(2)}.$

\noindent
We proceed with similar technique as we estimated $\mathrm{E}_{(1)}$ (\ref{esti1}) to estimate the following term:
\begin{align}\label{es1}
    &\nonumber\Bigg|\sum\limits_{\substack{m=1 }}^{[\varepsilon^{-1}]}\int_{\mathcal{N}_{(1)}}\frac{b}{4\pi\vert x-y\vert}\big[K(y)+1\big] \sum\limits_{l=1}^{\big[K(z_{m_l})+1\big]}\mspace{-25mu}\widetilde{\mathcal R}_{m_l} (y, t-c_0^{-1}\vert x-y\vert) dy\Bigg|
    \\ \nonumber &\lesssim \Big\Vert[K(z_{m_l})+1]\Big\Vert_{\mathrm{L}^\infty(\Omega)} \sum\limits_{\substack{m=1}}^{[\varepsilon^{-\frac{1}{3}}]}\frac{1}{d_{ij}}\Big\Vert \sum\limits_{l=1}^{\big[K(z_{m_l})+1\big]}\mspace{-25mu}\widetilde{\mathcal R}_{m_l} (y, t-c_0^{-1}\vert x-y\vert)\Big\Vert_{\mathrm{C}^1\big(0,\mathrm{T};\mathrm{L}^\infty(\mathbf{\Omega})\big)}\text{vol}\big(\Omega_m\big)
    \\  &\lesssim \mathcal{O}\Big(\varepsilon \sum\limits_{\substack{j=1}}^{[\varepsilon^{-\frac{1}{3}}]}\frac{1}{d_{ij}} \Big)
    \lesssim \mathcal{O}\Big(\varepsilon\ \big[(2n+1)^2-(2n-1)^2\big]\frac{1}{n(\varepsilon^\frac{1}{3}-\frac{\varepsilon}{2})} \Big)
    \lesssim \mathcal{O}(\varepsilon^{\frac{1}{3}}).
\end{align}
In a similar way, as $\text{vol}(\Omega_m) = \varepsilon$, we have
\begin{align}\label{es2}
  \Bigg| \sum\limits_{\substack{m=1 }}^{[\varepsilon^{-1}]}\int_{\mathcal{N}_{(1)}}\frac{b}{4\pi|\mathrm{x}-\mathrm{z}_{m_1}|}\ [K(z_{m_l})+1] \sum\limits_{\substack{l=1}}^{\big[K(z_{m_l})+1\big]}\widetilde{\mathcal{R}}_{m_l}(z_{m_l},\mathrm{t}-\mathrm{c}_0^{-1}|x-\mathrm{z}_{m_l}|) \Bigg|  \lesssim \mathcal{O}(\varepsilon^{\frac{1}{3}}).
\end{align}
Therefore, based on this previous two estimates (\ref{es1}) and (\ref{es2}), we reduce the expression (\ref{term}) as follows:
\begin{align}
    \mathbcal W(x,t) -u^\mathrm{s}(\mathrm{x},\mathrm{t}) \nonumber& =  
  \frac{b}{4\pi}\sum\limits_{\substack{m=1 }}^{[\varepsilon^{-1}]}\int_{\mathcal{N}_{(2)}}\big[K(z_{m_l})+1\big]\Big(\frac{1}{\vert x-y\vert} -\frac{1}{|\mathrm{x}-\mathrm{z}_{m_1}|} \Big){\sum\limits_{l=1}^{\big[K(z_{m_l})+1\big]}\widetilde{\mathcal R}_{m_l} (y, t-c_0^{-1}\vert x-y\vert)} dy + \mathcal{O}(\varepsilon^\frac{1}{3}).
\end{align}
Thereafter, similar to the estimate for $E_{(3)}$ (\ref{esti3}) and using the fact that $\Big[\mathrm{K}(z_{m_l})+1\Big] \in \mathrm{L}^\infty(\mathbf{\Omega})$, we deduce that
\begin{align}\nonumber
    \Bigg|\frac{b}{4\pi}\sum\limits_{\substack{m=1 }}^{[\varepsilon^{-1}]}\int_{\mathcal{N}_{(2)}}\big[K(z_{m_l})+1\big]\Big(\frac{1}{\vert x-y\vert} -\frac{1}{|\mathrm{x}-\mathrm{z}_{m_1}|} \Big){\sum\limits_{l=1}^{\big[K(z_{m_l})+1\big]}\widetilde{\mathcal R}_{m_l} (y, t-c_0^{-1}\vert x-y\vert)} dy\Bigg| \lesssim \varepsilon^\frac{1}{3},
\end{align}
from which we conclude that
\begin{align}
    \nonumber\mathbcal W(x,t) -u^\mathrm{s}(\mathrm{x},\mathrm{t}) = \mathcal{O}(\varepsilon^\frac{1}{3}).
\end{align}

\end{document}